\documentclass[a4paper, reqno, 12pt]{amsart}

\usepackage[usenames,dvipsnames]{color}
\usepackage{amsthm,amsfonts,amssymb,amsmath,amsxtra}
\usepackage[all]{xy}
\SelectTips{cm}{}
\usepackage{xr-hyper}
\usepackage{verbatim}

\usepackage[margin=1.25in]{geometry}
\usepackage{mathrsfs}

\RequirePackage{xspace}
\RequirePackage{etoolbox}
\RequirePackage{varwidth}
\RequirePackage{enumitem}
\RequirePackage{tensor}
\RequirePackage{mathtools}
\RequirePackage{longtable}
\RequirePackage{multirow}

\usepackage[
 citecolor=Red,
  linkcolor=Black,
   urlcolor=Blue]{hyperref}
\usepackage{cleveref}

\usepackage{tikz-cd}
\usepackage{tikz}
\usetikzlibrary{shapes.geometric, arrows}
\tikzstyle{block} = [thick, rectangle, font = \small, minimum width=3cm, minimum height=1cm, text centered, text width=3.5cm, draw=black]
\tikzstyle{bigblock} = [rectangle, dashed, minimum width=7.5cm, minimum height=10cm, draw=black]
\tikzstyle{arrow} = [ultra thick,->,>=stealth]
\tikzstyle{darrow} = [dashed,ultra thick,<->,>=stealth]
\tikzstyle{harrow} = [ultra thick,right hook->,>=stealth]

\def\ge{\geqslant}
\def\le{\leqslant}
\def\a{\alpha}
\def\b{\beta}

\def\l{\lambda}

\def\i{^{-1}}

\def\<{\langle}
\def\>{\rangle}

\newcommand{{\BG}}{\ensuremath{\mathbb {G}}\xspace}

\newcommand{{\BK}}{\ensuremath{\mathbb {K}}\xspace}

\newcommand{\BR}{\ensuremath{\mathbb {R}}\xspace}

\newcommand{\CB}{\ensuremath{\mathcal {B}}\xspace}

\newcommand{\CX}{\ensuremath{\mathcal {X}}\xspace}

\newcommand{\CZ}{\ensuremath{\mathcal {Z}}\xspace}

\def\JB{\prescript{J}{}{B}}
\def\JU{\prescript{J}{}{U}}
\def\oB{\mathring{\mathcal B}}
\def\oZ{\mathring{\mathcal Z}}
\def\IB{\prescript{I}{}{\tilde B}}

\newtheorem{theorem}{Theorem}
\newtheorem{proposition}[theorem]{Proposition}
\newtheorem{lemma}[theorem]{Lemma}
\newtheorem {conjecture}[theorem]{Conjecture}
\newtheorem{corollary}[theorem]{Corollary}

\theoremstyle{definition}
\newtheorem{definition}[theorem]{Definition}
\newtheorem{example}[theorem]{Example}

\newtheorem{remark}[theorem]{Remark}

\numberwithin{equation}{section}
\numberwithin{theorem}{section}

\setitemize[0]{leftmargin=*,itemsep=\the\smallskipamount}
\setenumerate[0]{leftmargin=*,itemsep=\the\smallskipamount}

\renewcommand{\to}{%
   \ifbool{@display}{\longrightarrow}{\rightarrow}%
   }
\let\shortmapsto\mapsto
\renewcommand{\mapsto}{%
   \ifbool{@display}{\longmapsto}{\shortmapsto}%
   }
\newlength{\olen}
\newlength{\ulen}
\newlength{\xlen}
\newcommand{\xra}[2][]{%
   \ifbool{@display}%
      {\settowidth{\olen}{$\overset{#2}{\longrightarrow}$}%
       \settowidth{\ulen}{$\underset{#1}{\longrightarrow}$}%
       \settowidth{\xlen}{$\xrightarrow[#1]{#2}$}%
       \ifdimgreater{\olen}{\xlen}%
          {\underset{#1}{\overset{#2}{\longrightarrow}}}%
          {\ifdimgreater{\ulen}{\xlen}%
             {\underset{#1}{\overset{#2}{\longrightarrow}}}
             {\xrightarrow[#1]{#2}}}}%
      {\xrightarrow[#1]{#2}}
   }
\makeatother
\newcommand{\xyra}[2][]{%
   \settowidth{\xlen}{$\xrightarrow[#1]{#2}$}%
   \ifbool{@display}%
      {\settowidth{\olen}{$\overset{#2}{\longrightarrow}$}%
       \settowidth{\ulen}{$\underset{#1}{\longrightarrow}$}%
       \ifdimgreater{\olen}{\xlen}%
          {\mathrel{\xymatrix@M=.12ex@C=3.2ex{\ar[r]^-{#2}_-{#1} &}}}%
          {\ifdimgreater{\ulen}{\xlen}%
             {\mathrel{\xymatrix@M=.12ex@C=3.2ex{\ar[r]^-{#2}_-{#1} &}}}
             {\mathrel{\xymatrix@M=.12ex@C=\the\xlen{\ar[r]^-{#2}_-{#1} &}}}}}%
      {\mathrel{\xymatrix@M=.12ex@C=\the\xlen{\ar[r]^-{#2}_-{#1} &}}}%
   }
\makeatletter
\newcommand{\xla}[2][]{%
   \ifbool{@display}%
      {\settowidth{\olen}{$\overset{#2}{\longleftarrow}$}%
       \settowidth{\ulen}{$\underset{#1}{\longleftarrow}$}%
       \settowidth{\xlen}{$\xleftarrow[#1]{#2}$}%
       \ifdimgreater{\olen}{\xlen}%
          {\underset{#1}{\overset{#2}{\longleftarrow}}}%
          {\ifdimgreater{\ulen}{\xlen}%
             {\underset{#1}{\overset{#2}{\longleftarrow}}}
             {\xleftarrow[#1]{#2}}}}%
      {\xleftarrow[#1]{#2}}
   }
\newcommand{\isoarrow}{%
   \ifbool{@display}{\overset{\sim}{\longrightarrow}}{\xrightarrow\sim}%
   }

\begin{document}

\title[]{Total positivity in twisted flag varieties}
\author[Xuhua He]{Xuhua He}
\address{Department of Mathematics and New Cornerstone Science Laboratory, The University of Hong Kong, Pokfulam, Hong Kong, Hong Kong SAR, China}
\email{xuhuahe@hku.hk}
\author{Kaitao Xie}
\address{Department of Mathematics and New Cornerstone Science Laboratory, The University of Hong Kong, Pokfulam, Hong Kong, Hong Kong SAR, China}
\email{kaitaoxie@connect.hku.hk}

\thanks{}



\begin{abstract}
Let $G$ be a Kac-Moody group, split over $\BR$. The totally nonnegative part of $G$ and its (ordinary) flag variety $G/B^+$ was introduced by Lusztig. It is known that the totally nonnegative parts of $G$ and $G/B^+$ have remarkable combinatorial and topological properties. 

In this paper, we consider the totally nonnegative part of the $J$-twisted flag variety $G/{}^J B^+$, where ${}^J B^+$ is the Borel subgroup opposite to $B^+$ in the standard parabolic subgroup $P_J^+$ of $G$. The $J$-twisted flag varieties include the ordinary flag variety $G/B^+$ as a special case. Our main result show that the totally nonnegative part of $G/{}^J B^+$ decomposes into cells, and the closure of each cell is a regular CW complex. This generalizes the work of Galashin-Karp-Lam \cite{GKL22} and the joint work of Bao with the first author \cite{BH24} for ordinary flag varieties. 

As an application, we deduce that the totally nonnegative part of the double flag variety $G/B^+ \times G/B^-$ with respect to the diagonal $G$-action has similar nice properties. We also establish some connections between the totally nonnegative part of the double flag with the canonical basis of the tensor product of a lowest weight module with a highest weight module of $G$. 

As another application, we show that the link of identity in a totally nonnegative reduced double Bruhat cell of $G$ is a regular CW complex. This generalizes the work of Hersh \cite{Her14} on the link of $U_{\geq0}^-$ and gives a positive answer to an open question of Fomin and Zelevinsky. 
\end{abstract}

\maketitle


\section{Introduction}

\subsection{Total positivity and remarkable polyhedral spaces}
The theory of total positivity in reductive groups and their flag varieties was introduced by Lusztig \cite{L94} and has been studied extensively in many areas including cluster algebras, higher Teichmüller theory, and theoretical physics. We refer to \cite{L20} and references therein for an overview.

We briefly recall the totally nonnegative flag variety. Let $G$ be a symmetrizable minimal Kac-Moody group that is split over $\mathbb R$. Let $B^+,B^-$ be a pair of Borel subgroups of $G$ that are opposite to each other, and let $W$ be the Weyl group of $G$. Lusztig \cite{L94, L20} defined a submonoid $G_{\geq0}$ of $G$, which he called the \textit{totally nonnegative monoid}.

Let $\CB^+ :=G/B^+$ be the (ordinary) flag variety. It admits a stratification $\CB^+ = \bigsqcup\oB_{v,w}^+$, where each $\oB_{v,w}^+$ is an intersection of a $B^-$-orbit with a $B^+$-orbit, and $v$ and $w$ are elements in the Weyl group $W$ indexing these orbits.

Lusztig defined the \textit{totally nonnegative flag variety} $\CB^+_{\geq0}$ as the Hausdorff closure of $G_{\geq0}B^+/B^+$ in $\CB(\mathbb R)$, and he introduced the decomposition $$\CB^+_{\geq0} = \bigsqcup (\oB_{v,w}^+\cap\CB^+_{\geq0}).$$ He called $\CB^+_{\geq0}$ a \textit{remarkable polyhedral space} without giving a precise definition. Based on later works, we provide the following definition: 

\begin{definition}
Let $X=\sqcup X_\a$ be a stratification of topological spaces over $\BR$ and $Y \subseteq X$ be a semi-algebraic subset. Set $Y_\a=Y \cap X_\a$. We say that $Y=\sqcup_\a Y_\a$ is a {\it remarkable polyhedral space} if for any $\a$: 
\begin{enumerate}
    \item (Connected components) $Y_\a$ is either empty or a connected component of $X_\a(\mathbb R)$;
    \item (Cell structure) If $Y_\a \neq \emptyset$, then $Y_\a \simeq\mathbb R_{>0}^{\dim_{\BR} X_\a}$ is a semi-algebraic cell;
    \item (Cellular decomposition) The Hausdorff closure $\overline{Y_\a}$ is the disjoint union of totally positive cells $\bigsqcup_{\b\leq\a} Y_\b$, where $\beta\leq\alpha$ if $X_\beta\subseteq\overline{X_{\alpha}}$.
   
\end{enumerate}

Moreover, we say that $Y$ has the {\it regularity property} if $Y$ is a regular CW complex, i.e., for every $\a$, there exists a homeomorphism from a closed ball to $\overline{Y_{\alpha}}$, which restricts to a homeomorphism from the interior of the ball to $Y_\alpha$.
\end{definition}

The following result summarizes some key properties for the totally nonnegative flag variety $\CB_{\geq0}^+$.

\begin{theorem}\label{thm:flag-main}
The totally nonnegative flag variety $\CB^+_{\geq 0}=\bigsqcup(\oB^+_{v, w}\cap\CB^+_{\geq0})$ is a remarkable polyhedral space with regularity property.
\end{theorem}

For reductive groups, Rietsch \cite{Ri99,Ri06} proved that $\CB_{\geq 0}^+$ is a remarkable polyhedral space, and Galashin, Karp and Lam \cite{GKL22} established the regularity property. For Kac-Moody groups, these statements were established by Bao and the first author in \cite{BH21a,BH24}.

\subsection{Motivations}
The motivation in this paper is to study the totally nonnegative double flag varieties and the totally nonnegative reduced double Bruhat cells.

The {\it double flag varieties} we consider is $G/B^+\times G/B^-$ and its {\it totally nonnegative part} is defined as
$$(G/B^+\times G/B^-)_{\geq0} = \overline{G_{\text{diag},\geq0}\cdot(B^+/B^+, B^-/B^-)}\ \ \text{(Hausdorff closure).}$$
The {\it double Bruhat cells} of a Kac-Moody group are the intersections of the $B^+\times B^+$-orbits and the $B^-\times B^-$-orbits on $G$. They play an important role in the theory of cluster algebras. The maximal torus $T$ of $G$ acts freely on each double Bruhat cell $G^{w,u}$ of $G$, and the quotient $L^{w,u}:= G^{w,u}/T$ is called a {\it reduced double Bruhat cell}. There is a natural embedding $$G/T \to G/B^+\times G/B^-,$$ which is compatible with total positivity.

The totally nonnegative double flag variety was first considered by Webster and Yakimov \cite{WY07}. They conjectured that it admits a cellular decomposition and verified $G=GL_2$ case. For many years, this was the only known case. In \cite{BH22}, Bao and the first author developed a thickening method and established the desired properties for another family of double flag varieties: $G/B^+ \times G/B^+$.  The key idea is to construct a larger Kac-Moody group $\tilde G$ and a sub-complex of $(\tilde G/\tilde B^+)_{\geq0}$ that forms a fiber bundle over $(G/B^+\times G/B^+)_{\geq0}$. When $G$ is of finite type, the totally nonnegative part of $G/B^+ \times G/B^+$ and $G/B^- \times G/B^+$ coincide. For general Kac–Moody groups, however, this identification fails, and the two double flag varieties are essentially different. The method of \cite{BH22} does not apply directly to the case $G/B^+ \times G/B^-$ beyond the finite-type cases, and a new approach is required.

To overcome this obstacle and to study the totally nonnegative double flag variety and totally nonnegative reduced double Bruhat cells for an arbitrary Kac–Moody group, we are led to a generalization of the ordinary flag variety—the \textit{twisted flag variety}—and to the investigation of its totally nonnegative part. The twisted setting provides the flexibility needed to apply a thickening argument to $G/B^+ \times G/B^-$ without the finite-type assumption. 

\subsection{Generalization to twisted flag variety}
Let $I$ be the index set of simple roots of $G$. For a subset $J\subseteq I$, we denote $W_J$ the corresponding parabolic subgroup of $W$, and $W^J$ the set of minimal representatives in the cosets in $W/W_J$. We also have a ``twisted" Bruhat order $\leq^J$, which, roughly speaking, is a mixture of the opposite Bruhat order in $W_J$ and the Bruhat order in $W^J$.

Let $P^+_J$ be the corresponding standard parabolic subgroup. Let $\JB^+$ be the opposite Borel subgroup of $B^+$ inside $P_J^+$. The {\it $J$-twisted flag variety} is defined to be $\CB^J := G/\JB^+$. When the Dynkin diagram of $J$ is not of finite type, $\CB^J$ is NOT $G$-equivariantly isomorphic to $\CB^+$.

The \textit{totally nonnegative $J$-twisted flag variety} is defined to be the Hausdorff closure of $G_{\geq0}\JB^+/\JB^+$ in $\CB^J(\mathbb R)$. In \S\ref{sec:J-twsited}, we define a stratification $\CB^J = \bigsqcup_{v \leq^J w}\oB^J_{v,w}$, and obtain the induced decomposition $$\CB^J_{\geq0} = \bigsqcup_{v \leq^J w} (\oB^J_{v, w}\cap\CB^J_{\geq0}).$$
Our first main result extends Theorem \ref{thm:flag-main} to the twisted setting.

\begin{theorem}\label{thm:J-flag}
The totally nonnegative $J$-twisted flag variety $\CB^J_{\geq 0}=\bigsqcup(\oB^J_{v,w}\cap\CB^J_{\geq0})$ is a remarkable polyhedral space with regularity property. 
\end{theorem}

To study the double flag variety $G/B^- \times G/B^+$, we construct an auxiliary larger Kac–Moody group $\tilde G$ and a twisted flag variety $\tilde \CB^{\tilde J}$ of $\tilde G$. Inside $\tilde \CB^{\tilde J}$, we define a subvariety $\tilde Z$ that forms a fiber bundle over $G/B^+ \times G/B^-$; the bundle projection is compatible with the stratifications and with total positivity. Applying Theorem \ref{thm:J-flag} to this setting yields our second main result.

\begin{theorem}\label{thm:double-flag}
 The totally nonnegative double flag variety $(G/B^+\times G/B^-)_{\geq0}$ is a remarkable polyhedral space with regularity property.
\end{theorem}


Since the embedding of $G/T$ into the double flag variety is is compatible with the stratifications and with total positivity, we also deduce
\begin{theorem}
The link of identity in a totally nonnegative reduced double Bruhat cell of a Kac-Moody group is a regular CW complex. In particular, it is homeomorphic to a closed ball.
\end{theorem}

This generalizes the results of Hersh \cite{Her14} on the link of identity in the totally nonnegative unipotent monoid $U^-_{\geq0}$, and gives a positive answer to an open problem of Fomin and Zelevinsky (see \cite[Conjecture 10.2 (1)]{GKL22}).

\smallskip

Below we outline the relationships among the main objects discussed above.

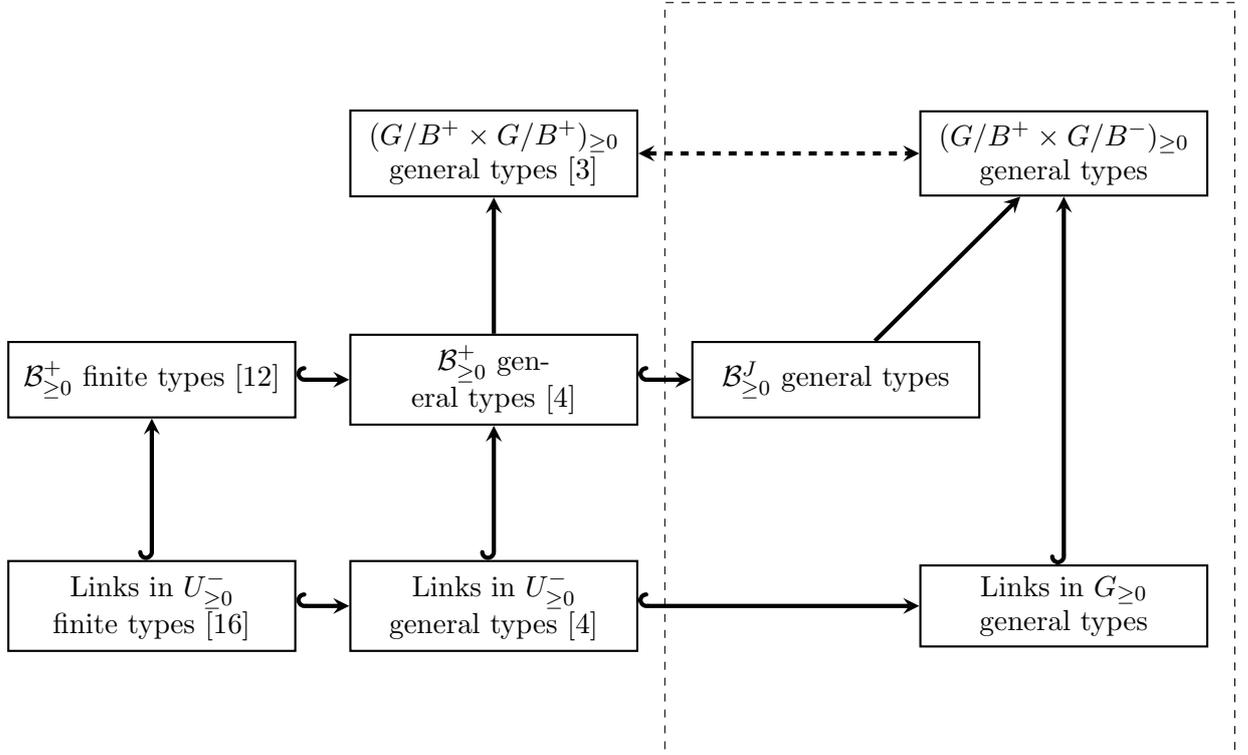
\begin{center}
\begin{tikzpicture}
\node [block] at (-2,0) (FlagFinite){$\CB^+_{\geq0}$ finite types \cite{GKL22}};
\node [block] at (2.5,0) (FlagKM){$\CB^+_{\geq0}$ general types \cite{BH24}};
\node [block] at (7,0) (twisted){$\CB^J_{\geq0}$ general types};
\node [block] at (2.5,3) (++){$(G/B^+\times G/B^+)_{\geq0}$ general types \cite{BH22}};
\node [block] at (10,3) (+-){$(G/B^+\times G/B^-)_{\geq0}$ general types};
\node [block] at (-2,-3) (LinkFinite){Links in $U^-_{\geq0}$ finite types \cite{Her14}};
\node [block] at (2.5,-3) (LinkKM){Links in $U^-_{\geq0}$ general types \cite{BH24}};
\node [block] at (10,-3) (LinkDouble){Links in $G_{\geq0}$ general types}; 
\node [bigblock] at (8.5,0) (new){};

\draw [harrow] (FlagFinite) -- (FlagKM);
\draw [harrow] (LinkFinite) -- (FlagFinite);
\draw [harrow] (FlagKM) -- (twisted);
\draw [arrow] (FlagKM) -- (++);
\draw [darrow] (++) --  (+-);
\draw [harrow] (LinkFinite) -- (LinkKM);
\draw [harrow] (LinkKM) -- (LinkDouble);
\draw [harrow] (LinkKM) -- (FlagKM);
\draw [arrow] (twisted)--(+-);
\draw [harrow] (LinkDouble)--(+-);
\end{tikzpicture}
\end{center}

In this diagram, $\hookrightarrow$ indicates that the topological property of the object at the tail is a special case of the topological property of the object at the head; $\rightarrow$ indicates that the topological property of the object at the head can be deduced from the topological property of the object at the tail.
Note that the two double flag varieties $G/B^+\times G/B^+$ and $G/B^+\times G/B^-$ coincide when $G$ is of finite type.
The main results of this paper establish the topological properties for the three objects enclosed in the dotted box.

\subsection{Difficulty and strategy}
The fundamental obstacle in extending total positivity to twisted flag varieties, compared with the ordinary ($J=\emptyset$) case, is the absence of a representation-theoretic interpretation for twisted flag varieties. In the ordinary case, one first uses the embedding of the flag variety into the projective space of a highest weight module and the associated canonical basis to establish the Marsh–Rietsch parametrization of totally positive cells \cite{MR04,BH21a}. This explicit parametrization immediately implies that each totally positive cell lies inside certain ``big cells'' of the flag variety, a key geometric property that allows the product structure method of \cite{BH24} to be applied. The product structure then reduces the study of the closure of a cell to the study of two smaller cells, leading to an inductive proof of the closure relation and the regularity theorem.

For a twisted flag variety with $J$ not of finite type, no analogous extreme‑weight theory is available, and the parametrization cannot be established a priori. Consequently, one cannot separate the parametrization step from the geometric analysis; the existence of a suitable parametrization and the verification of the required geometric properties (notably the containment in certain ``big cells'') must be proved simultaneously.

Our strategy overcomes this entanglement by a carefully designed induction. We first extend the method of product structure to the twisted setting. Then we introduce explicit candidate parametrizing sets for the totally positive cells. Starting with a restricted family of strata (those with $v \in W_J$ and $w \in W^J$), we prove that the candidate sets for these strata are indeed the totally positive cells and satisfy the required inclusion property. Using the product structure, we then extend these conclusions to a larger family (for $v \in W$ and $w \in W^J$), and finally to arbitrary pairs $(v, w)$. At each stage the product structure allows us to propagate the parametrization and the geometric properties to the next stage. In this way the parametrization and the product structure are built together, culminating in the full proof of Theorem \ref{thm:J-flag}.

\subsection{Representation-theoretic interpretation of double flag varieties}
Although we do not have a representation-theoretic interpretation of the twisted flag varieties, we have a natural connection between the double flag varieties and representation theory. 

In this subsection, we assume that $G$ is simply-laced. Let $X^{++}$ be the set of regular dominant weights of $G$. For $\lambda\in X^{++}$, let $\Lambda$ be the simple highest weight module of $G$ associated with a regular dominant weight, and $\eta_\lambda$ be a highest weight vector. Let $\mathbb P(\Lambda_\lambda)$ be the projective space of $\Lambda_\lambda$. For $v\in\Lambda_\lambda$ with $v\neq0$, denote $[v]\in\mathbb P(\Lambda_\lambda)$ the corresponding line in the projective space. We have the natural embedding
$$\CB^+\to\mathbb P(\Lambda_\lambda),\quad gB^+/B^+\to [g\eta_\lambda].$$
We identify $\CB^+$ with its image in $\mathbb P(\Lambda_\lambda)$.

Lusztig’s theory of canonical bases for quantum groups \cite{L92a} provides a basis $\mathbf B(\Lambda_\lambda)$ of $\Lambda$. Let $\mathbb P(\Lambda_\lambda)_{\geq0}$ be the set of lines spanned by a vector whose coordinates with respect to $\mathbf B(\Lambda)$ are all nonnegative. For reductive groups Lusztig \cite[Proposition 8.17]{L94} and for Kac–Moody groups Bao and He \cite{BH21a} proved that
$$\CB^+_{\geq0} = \CB^+\bigcap\mathbb P(\Lambda_\lambda)_{\geq0}.$$
This equality gives a representation‑theoretic characterization of the totally nonnegative flag variety.

It is natural to ask whether the totally nonnegative double flag variety admits a similar description.

For $\lambda\in X^{++}$, let $^{\omega}\Lambda_\lambda$ be a simple lowest weight module of $G$ associated to $\lambda$, and let $\xi_{-\lambda}$ be a lowest weight vector. For $\lambda_1,\lambda_2\in X^{++}$, we consider the tensor product $^{\omega}\Lambda_{\lambda_1}\otimes\Lambda_{\lambda_2}$ and its projective space $\mathbb P(^\omega\Lambda_{\lambda_1}\otimes\Lambda_{\lambda_2})$. There is a natural embedding
$$G/B^+\times G/B^-\to\mathbb P(^\omega\Lambda_{\lambda_1}\otimes\Lambda_{\lambda_2}),\ \ (g_1B^+/B^+,g_2B^-/B^-)\to [g_2\xi_{-\lambda_1}\otimes g_1\eta_{\lambda_2}].$$
We identify $G/B^+\times G/B^-$ with its image in $\mathbb P(^\omega\Lambda_{\lambda_1}\otimes\Lambda_{\lambda_2})$. Lusztig \cite{L92b}  introduced the canonical basis $\mathbf B(^\omega\Lambda_{\lambda_1}\otimes\Lambda_{\lambda_2})$ for  $^\omega\Lambda_{\lambda_1}\otimes\Lambda_{\lambda_2}$.

\begin{conjecture} \label{conjecture_canonical_basis}
Let $\mathbb P(^\omega\Lambda_{\lambda_1}\otimes\Lambda_{\lambda_2})_{\geq0}$ be the set of lines spanned by a vector whose coordinates with respect to $\mathbf B(^\omega\Lambda_{\lambda_1}\otimes\Lambda_{\lambda_2})$ are all nonnegative. Then
$$(G/B^+\times G/B^-)_{\geq0} = (G/B^+\times G/B^-)\bigcap\mathbb P(^\omega\Lambda_{\lambda_1}\otimes\Lambda_{\lambda_2})_{\geq0}.$$
\end{conjecture}

In a recent joint work of the first author with Fang \cite{FH25}, some positivity property of $\mathbf B(^\omega\Lambda\otimes\Lambda)$ was established. Using these results, we prove one inclusion of the conjecture in \S\ref{sec:canonical_basis}:
$$(G/B^+\times G/B^-)_{\geq0}\subseteq \mathbb P(^\omega\Lambda\otimes\Lambda)_{\geq0}.$$ 

\subsection{Acknowledgment}
We thank Huanchen Bao for many fruitful discussions on total positivity. The idea of using the twisted flag varieties to the study of double flag varieties arose from discussions with him. XH is partially supported by the New Cornerstone Science Foundation through the New Cornerstone Investigator Program and the Xplorer Prize, as well as by the Hong Kong RGC grant 14300023. KX is partially supported by  partially supported by the New Cornerstone Science Foundation through the New Cornerstone Investigator Program awarded to XH. This manuscript was prepared with the assistance of AI tools to improve linguistic readability.

\section{Preliminary}

\subsection{Kac-Moody groups}
Let $I$ be a finite set, and let $A = (a_{ij})_{i,j\in I}$ be a symmetrizable generalized Cartan matrix. A {\it Kac-Moody datum} associated with $A$ is a  sextuple $\mathcal D = (I,A,X,Y,(\alpha_i)_{i\in I}, (\alpha^\vee_i)_{i\in I})$ where $X$ and $Y$ are two free $\mathbb Z$-modules dual to each other,  and the simple roots $\alpha_i\in X$ and simple coroots $\alpha_j^\vee\in Y$ satisfy $\langle\alpha_i,\alpha^\vee_j\rangle = a_{ij}$ for all $i,j\in I$. The {\it minimal Kac-Moody group} $G$ over $\mathbb C$ associated with $\mathcal D$ is the split group over $\mathbb C$ generated by the split torus $T$ associated with $Y$ and the root groups $U_{\pm\alpha_i}$ for all $i\in I$, subject to the Tits relation. Let $U^\pm\subseteq G$ be the subgroup generated by all $U_{\pm\alpha_i}$, and let $B^\pm\subseteq G$ be the Borel subgroups generated by $T$ and $U^\pm$, respectively. For every $i\in I$, we fix isomorphisms $x_i:\mathbb R\to U_{\alpha_i}$, $y_i:\mathbb R\to U_{-\alpha_i}$ such that the map
$$\begin{pmatrix} 1 & a \\ 0 & 1 \end{pmatrix}\mapsto x_i(a), \begin{pmatrix} 1 & 0 \\ b & 1 \end{pmatrix}\mapsto y_i(b),
\begin{pmatrix} c & 0 \\ 0 & c^{-1} \end{pmatrix}\mapsto \alpha^\vee_i(c), a,b\in\mathbb C, c\in\mathbb C^*,$$
defines a homomorphism $SL_2\to G$. The datum $(T,B^+,B^-,x_i,y_i;i\in I)$ is called a {\it pinning} for $G$. Denote by $\Phi = \Phi^+\sqcup\Phi^-$ the system of real roots of $G$. 

Let $W$ be the {\it Weyl group} of $G$. For $i\in I$, let $s_i\in W$ be the corresponding simple reflection, and we write $\dot s_i = x_i(1)y_i(-1)x_i(1)\in G$. 

For any $w \in W$, choose a {\it reduced expression} $w = s_{i_1} s_{i_2} \cdots s_{i_n}$. The {\it length} of $w$, denoted $\ell(w)$, is defined to be $n$. We define the lift $\dot{w} \in G$ by $\dot{w} = \dot{s}_{i_1} \dot{s}_{i_2} \cdots \dot{s}_{i_n}$. This definition is independent of the choice of reduced expression.

\subsection{Flag varieties}\label{sec:flag}
Denote $\mathcal B^+ := G/B^+$ the {\it (thin) flag variety}, equipped with the ind-variety structure. For $v,w\in W$, define
\begin{itemize}
    \item The {\it Schubert cell} $\oB^+_w := B^+\dot wB^+/B^+$ corresponding to $w$, 
    \item The {\it opposite Schubert cell} $\oB^{+,v} := B^-\dot vB^+/B^+$ corresponding to $v$, 
    \item The {\it open Richardson variety} $\oB^+_{v,w} := \oB^+_w\cap \oB^{+,v}$ corresponding to $(v,w)$.
\end{itemize}

It follows from \cite{KL79} and \cite{Kum17} that 

(a) $\oB^+_{v,w}\neq\emptyset$ if and only if $v\leq w$ in the Bruhat order.

(b) For $v \leq w$, $\dim \oB^+_{v,w}=\ell(w)-\ell(v)$.

We have the following stratifications
$$\mathcal B^+ = \bigsqcup_{w\in W}\oB_w^+ = \bigsqcup_{v\in W}\oB^{+,v} = \bigsqcup_{v\leq w}\oB^+_{v,w}.$$
Let $\mathcal B_w^+$, $\mathcal B^{+,v}$, and $\mathcal B^+_{v,w}$ be the Zariski closures of $\oB^+_w$, $\oB^{+,v}$ and $\oB^+_{v,w}$, respectively. We have

(c) $\mathcal B^+_w = \bigsqcup_{w'\leq w}\oB^+_{w'},\  \mathcal B^{+,v} = \bigsqcup_{v'\geq v}\oB^{+,v'},\ \mathcal B^+_{v,w} = \bigsqcup_{v\leq v'\leq w'\leq w}\oB^+_{v',w'}.$

\subsection{$J$-twisted flag varieties} \label{sec:J-twsited}
Let $J\subseteq I$. Let $P_J^+$ be the {\it parabolic subgroup} of $G$ generated by $B^+$ and the root groups $U_{-\alpha_i}$ for $i\in J$. Let $P_J^-$ be the {\it opposite parabolic subgroup} of $G$ generated by $B^-$ and the the root groups $U_{\alpha_i}$ for $i\in J$. Let $L_J = P_J^+\cap P_J^-$, and $U_{P_J^\pm}$ be the unipotent radical of $P_J^\pm$. We have the Levi decomposition $P_J^\pm = L_J\ltimes U_{P_J^\pm}$. We denote by $\Phi_J = \Phi_J^+\sqcup\Phi_J^-$ the system of real roots of $L_J$.

For $J\subseteq I$, let $W_J\subseteq W$ be the corresponding parabolic subgroup, and $W^J$ be the set of all minimal representatives for the cosets in $W/W_J$. Every element $w\in W$ has a unique decomposition $w = w^J w_J$ with $w^J\in W^J$ and $w_J\in W_J$. 

We define the {\it $J$-length} of $w$ by $$\ell^J(w):=\ell(w^J)-\ell(w_J).$$ We also define the {\it $J$-Bruhat order} $\leq^J$ as follows: for $v,w\in W$, we say $v\leq^J w$ if there is $u\in W_J$ such that $$v^Ju\leq w^J \text{ and } w_J\leq u^{-1}v_J,$$ where $\leq$ denotes the usual Bruhat order on $W$. 

We consider the groups $$B^\pm_J  =L_J\cap B^\pm, \quad \JB^+ = B_J^-\ltimes U_{P_J^+}, \quad \JB^- = B_J^+\ltimes U_{P_J^-}.$$ Note that $\JB^+$ is the opposite Borel subgroup in the standard parabolic subgroup $P^+_J$ of $G$. Let $\JU^\pm$ denote the unipotent radicals of $\JB^\pm$.

Denote by $\mathcal B^J := G/\JB^+$ the {\it $J$-twisted flag variety}, equipped with the ind-variety structure. For $v,w\in W$, define
\begin{itemize}
    \item The {\it $J$-Schubert cell} $\oB^J_w := B^+\dot w\JB^+/\JB^+$ corresponding to $w$, 
    \item The {\it opposite $J$-Schubert cell} $\oB^{J,v} := B^-\dot v\JB^+/\JB^+$ corresponding to $v$, 
    \item The open $J$-twisted Richardson variety $\oB^J_{v,w} := \oB_w^J\cap \oB^{J,v}$ corresponding to $(v,w)$.
\end{itemize}

When $J=\emptyset$, we have $\JB^+ = B^+$; when $J=I$, we have $\JB^+ = B^-$. The twisted flag $\mathcal B^J$ is a generalization of the ordinary flag variety $\mathcal B^+$.

Let $\mathcal B^J_w$ and $\mathcal B^{J,v}$ denote the respective Zariski closures of $\oB^J_w$ and $\oB^{J,v}$, respectively. By \cite[Theorem 4]{BD94}, 
$$\mathcal B^J_w = \bigsqcup_{w'\leq^J w}\oB_{w'}^J,\ \mathcal B^{J,v} = \bigsqcup_{v\leq^J v'}\oB^{J,v'}.$$ 

We will show in \Cref{nonempty_pattern} that $\oB^J_{v, w} \neq \emptyset$ if and only if $v \leq^J w$.

\subsection{Totally nonnegative flag varieties: ordinary and twisted}
Let $U^+_{\geq0}$ be the submonoid of $G$ generated by $x_i(a)$ for $i\in I$ and $a\in\mathbb R_{>0}$, and $U^-_{\geq0}$ be the submonoid of $G$ generated by $y_i(a)$ for $i\in I$ and $a\in\mathbb R_{>0}$. Let $T_{>0}$ be the identity component of $T(\mathbb R)$. The {\it totally nonnegative monoid} $G_{\geq0}$ is defined to be the submonoid of $G$ generated by $U^\pm_{\geq0}$ and $T_{>0}$. It is known that $$G_{\geq0} = U_{\geq0}^+T_{>0}U^-_{\geq0} = U^-_{\geq0}T_{>0}U^+_{\geq0}.$$

The {\it totally nonnegative flag variety} $\mathcal B^+_{\geq0}$ is defined to be the Hausdorff closure $\overline{G_{\geq0}B^+/B^+}$ of $G_{\geq0}B^+/B^+$ in $\CB^+(\mathbb R)$. Since $T_{>0}U_{\geq0}^+\subseteq B^+$, we also have $\mathcal B^+_{\geq0} = \overline{U^-_{\geq0}B^+/B^+}$. 

In analogy to the totally nonnegative flag variety, we define the {\it totally nonnegative $J$-twisted flag variety} $\mathcal B^J_{\geq0}$ to be the Hausdorff closure $\overline{G_{\geq0}\JB^+/\JB^+}$ of $G_{\geq0}\JB^+/\JB^+$ in $\mathcal B^J(\mathbb R)$. 

\begin{example}
Suppose $W_J$ is finite. Let $w_{J,0}$ be the unique longest element in $W_J$. The morphism $\varphi:\mathcal B^J\to\mathcal B^+$ defined by $\varphi(g\JB^+/\JB^+) = g\dot w_{J,0}B^+/B^+$ is a $G$-equivariant isomorphism. We have
$$\varphi(G_{\geq0}\JB^+/\JB^+)=G_{\geq0}\dot w_{J,0}B^+/B^+ = G_{\geq0}U_{w_{J,0},>0}^+\dot w_{J,0}B^+/B^+ = G_{\geq0}B^+/B^+.$$
The last equality follows from Lusztig's duality \cite[Theorem 8.7]{L94}: $$U^+_{w_{J,0},>0}\dot w_{J,0}B^+/B^+ = U^-_{w_{J,0},>0}B^+/B^+.$$ 
We conclude that, when $W_J$ is finite, $\varphi$ restricts to an isomorphism $\mathcal B^J_{\geq0}\simeq\mathcal B_{\geq0}^+$.
\end{example}

Returning to the general situation, for $v, w \in W$, we set $\CB^J_{v, w, >0}=\CB^J_{\geq 0} \cap \oB^J_{v, w}$. Then
$$\CB^J_{\geq0} = \bigsqcup_{v,w}\CB^J_{v,w,>0}.$$
This is the object of major interest in this paper.

Below we recollect some combinatorial and topological results that will be used.

\subsection{Combinatorial results}\label{sec:comb}
Let $(P,\leq)$ be a poset. For every $x,y\in P$ with $x\leq y$, the \textit{interval} between $x$ and $y$ is defined as $[x,y] = \{z\in P|x\leq z\leq y\}$. A subset $C\subseteq P$ is \textit{convex} if $[x,y]\subseteq C$ for $x,y\in C$.

The covering relation is denoted by $\lessdot$. For $x,y\in P$ with $x\leq y$, a maximal chain from $y$ to $x$ is a finite sequence of elements $y = z_0\gtrdot z_1\gtrdot\cdots\gtrdot z_n = x$. In general, a maximal chain may not exist.

The poset $P$ is \textit{pure} if for any $x,y\in P$ with $x\leq y$, a maximal chain from $x$ to $y$ always exist and have the same length. A pure poset is \textit{graded} if it contains a unique minimum and a unique maximum. A pure poset is \textit{thin} if every interval of length $2$ has exactly $4$ elements.

The order complex $\triangle_{ord}(P)$ of a finite poset $P$ is a simplicial complex whose vertices
are the elements of $P$ and whose simplexes are the chains $x_0 > x_1 >\cdots > x_k$ in $P$. A
finite pure poset is \textit{shellable} if its order complex $\triangle_{ord}(P)$ is shellable, that is, its maximal faces can be ordered as $F_1,F_2,\cdots,F_n$ so that $F_k\cap (\cup_{i=1}^{k-1}F_i)$ is a nonempty
union of the maximal proper faces of $F_k$, for $k=2,3,\cdots, n$.

We consider the poset $(Q^J,\leq_{Q^J})$ of intervals in $(W,\leq^J)$.  Namely, 
$$Q^J = \{[x,y]\in W\times W|x\leq^J y\}, $$
and $[x_1,y_1]\leq_{Q^J}[x_2,y_2]$ if $x_2\leq^J x_1\leq^Jy_1\leq^Jy_2$ in $W$. Let $\hat Q^J := Q^J\sqcup\{\hat 0\}$ be the augmented poset with $\hat 0$ an adjoined minimal element. The following proposition can be proved similarly as in \cite[Proposition 3.9]{Dy92} and \cite{Wil07}. See also \cite[Proposition 3.1]{BH24} and \cite[\S4]{BH21b}.

\begin{proposition} \label{Weyl_shellable}
\begin{enumerate}
    \item Every interval in the poset $(W,\leq^J)$ is graded, thin and shellable.
    \item Every interval in the poset $(\hat Q^J,\leq_{\hat Q^J})$ is graded, thin and shellable. 
\end{enumerate}
\end{proposition}

\subsection{Topological results}
Let $X$ be a regular CW complex. The {\it face poset} of $X$ is the poset $(Q,\leq)$, where $\beta\leq\alpha$ if and only if $X_\beta\subseteq\overline{X_\alpha}$. The following result \cite{Bjo84} connects the combinatorial properties introduced in \S \ref{sec:comb} to the regularity of CW complexes:

\begin{proposition} \label{shellable_regular}
Let $X$ be a regular CW complex with face poset $Q$. If $Q\sqcup\{\hat0,\hat 1\}$ (adjoining a minimum $\hat 0$ and a maximum $\hat 1$) is graded, thin, and shellable, then $X$ is homeomorphic to a sphere of dimension $\text{rank}(Q)-1$.
\end{proposition}

Recall that an $n$-dimensional topological manifold with boundary is a Hausdorff space $X$ such that every point $x \in X$ has an open neighborhood homeomorphic to either $\BR^n$, or $\BR_{\ge 0} \times \BR^{n-1}$, with $x$ mapped to a point in $\{0\} \times  \BR^{n-1}$ in the latter case. In this case, we say that $x$ lies on $\partial X$, the boundary of $X$.

The following theorem can be derived from the generalized Poincar\'e conjecture and Brown's collar theorem (see \cite[Theorem 2.2]{BH24} and references therein).

\begin{theorem}\label{thm:poincare} 
Let $X$ be a compact $n$-dimensional topological manifold with boundary such that 
\begin{enumerate}
    \item $\partial X$ is homeomorphic to an $(n-1)$-dimensional sphere, and
    \item $X \setminus\partial X$ is homeomorphic to an $n$-dimensional open ball. 
\end{enumerate}
Then $X$ is homeomorphic to an $n$-dimensional closed ball $D^n$.
\end{theorem}

\section{The Product Structure} \label{sec:product_structure}
\subsection{Product structure on the flag varieties} \label{product_structure}
We begin by recalling the product structure on the flag variety $\mathcal B$ from \cite{BH24}. 

For any $r\in W$, we have the following isomorphisms:
$$(U^-\cap \dot rU^-\dot r^{-1})\times (U^+\cap\dot rU^-\dot r^{-1})\to \dot rU^-\dot r^{-1},\ (g_1,g_2)\mapsto g_1g_2,$$ 
$$(U^+\cap \dot rU^-\dot r^{-1})\times (U^-\cap\dot rU^-\dot r^{-1})\to \dot rU^-\dot r^{-1},\ (h_1,h_2)\mapsto h_1h_2.$$

We define 
$$\tilde\sigma_{r}:\dot rU^-\dot r^{-1}\to(\dot rU^-\dot r^{-1}\cap U^+)\times (\dot rU^-\dot r^{-1}\cap U^-),\ g\mapsto (g_2,h_2).$$
The morphism $\tilde\sigma_r$ is an isomorphism of ind-varieties. It induces an isomorphism
$$\sigma_r = (\sigma_{r,+},\sigma_{r,-}):\dot rU^-B^+/B^+\to\oB^+_r\times\oB^{+,r},$$ given by 
$$ g\dot rB^+/B^+\mapsto (g_2\dot rB^+/B^+, h_2\dot rB^+/B^+), \text{ for }g\in\dot rU^-\dot r^{-1}.$$

This product structure is a powerful tool for studying connected components of open Richardson varieties (see \cite[Theorem 3.5, Corollary 3.7, Theorem 4.1]{BH24}). 

\begin{theorem}\label{thm:product}
    Let $Y_{v,w} $ be a connected component  of $\oB^+_{v,w}(\BR)$.  For any $ v \le v'\le w' \le w$, define 
\[
	Y_{v', w'}  = \overline{Y_{v,w}}  \bigcap \oB^+_{v', w'}.
\]
Assume that for any $ v \le v'\le r \le w' \le w$, we have $Y_{v', w'} \subseteq  \dot{r} B^- B^+ /B^+$. Then 
\begin{enumerate}
    \item The map $\sigma_{r}$ restricts to an isomorphism $\sigma_{r} : Y_{v', w'} \cong Y_{v', r} \times Y_{r, w'}$;

    \item The closure $\overline{Y_{v,w}}=\bigsqcup_{v\leq v'\leq w'\leq w} Y_{v', w'}$ is a remarkable polyhedral space;

    \item The closure $\overline{Y_{v,w}}=\bigsqcup_{v\leq v'\leq w'\leq w} Y_{v', w'}$ has the regularity property. 
\end{enumerate}
\end{theorem}

It was shown in \cite{BH24} that $\CB^+_{v, w, >0}$ satisfies the assumption in Theorem~\ref{thm:product}, thus Theorem~\ref{thm:flag-main} was deduced from the product structure. 

\subsection{Product structure on the $J$-twisted flag varieties} \label{J_product_structure}
For any $r\in W$, we have the following isomorphisms for the $J$-twisted case:
$$(U^-\cap \dot r\JU^-\dot r^{-1})\times (U^+\cap\dot r\JU^-\dot r^{-1})\to \dot r\JU^-\dot r^{-1},\ (g_1,g_2)\to g_1g_2,$$
$$(U^+\cap \dot r\JU^-\dot r^{-1})\times (U^-\cap\dot r\JU^-\dot r^{-1})\to \dot r\JU^-\dot r^{-1},\ (h_1,h_2)\mapsto h_1h_2.$$

We define 
$$\tilde\sigma_{r}^J:\dot r\JU^-\dot r^{-1}\to(\dot r\JU^-\dot r^{-1}\cap U^+)\times (\dot r\JU^-\dot r^{-1}\cap U^-),\ g\to (g_2,h_2).$$
The morphism $\tilde\sigma_r^J$ is an isomorphism of ind-varieties. It induces an isomorphism
$$\sigma_r^J = (\sigma_{r,+}^J,\sigma_{r,-}^J):\dot r\JU^-\JB^+/\JB^+\to\oB_r^J\times\oB^{J,r},$$ given by 
$$ g\dot r\JB^+/\JB^+\mapsto (g_2\dot r\JB^+/\JB^+, h_2\dot r\JB^+/\JB^+), \text{ for }g\in\dot r\JU^-\dot r^{-1}.$$
The isomorphism $\sigma_r^J$ is stratified with respect to the open $J$-twisted Richardson varieties, namely, $$\sigma_r^J(\oB_{v,w}^J\cap\dot r\JU^-\JB^+/\JB^+) = \oB^J_{v,r}\times\oB^J_{r,w}.$$

These morphisms allow us to understand the closure of a subset of $\oB^J_{v,w}$ in $\mathcal B^J$ via the image of $\sigma^J_r$, as shown in the following lemmas.

\begin{lemma} \label{alg_cut}
Let $v\leq^Jr\leq^J w$. Consider a subset $X\subseteq\oB^J_{v,w}(\mathbb R)$. Suppose that $\dot r\JB^+/\JB^+\in\overline{X}$. Let $v',w'\in W$ such that $v\leq^Jv'\leq^Jr\leq^Jw'\leq^Jw$. Then
$$\overline X\cap\oB^J_{v',r} = \overline{\sigma^J_{r,+}(X\cap\dot r\JU^-\JB^+/\JB^+)}\cap \oB^J_{v',r};$$
$$\overline X\cap\oB^J_{r,w'} = \overline{\sigma^J_{r,-}(X\cap\dot r\JU^-\JB^+/\JB^+)}\cap \oB^J_{r,w'}.$$
\end{lemma}

\begin{remark}
Similar results hold when the Hausdorff closure is replaced by the Zariski closure.
\end{remark}

\begin{proof}
Let $X' = X\cap\dot r\JU^-\JB^+/\JB^+$. Since $\dot r\JU^-\JB^+/\JB^+$ is open, we have $\overline{X'}\cap\dot r\JU^-\JB^+/\JB^+ = \overline{X}\cap\dot r\JU^-\JB^+/\JB^+$. Via restriction we have
$$\overline{X'}\cap\dot r\JU^-\JB^+/\JB^+\simeq \overline{\sigma_{r,+}^J(X')}\times\overline{\sigma_{r,-}^J(X')}.$$
Since $\dot r\JB^+/\JB^+\in\overline{X}\cap\dot r\JU^-\JB^+/\JB^+\subseteq\overline{X'}$, we have $\dot r\JB^+/\JB^+\in\overline{\sigma_{r,+}^J(X')}$ and $\dot r\JB^+/\JB^+\in\overline{\sigma_{r,-}^J(X')}$. Since the isomorphism is stratified, we have
$$\overline{X'}\cap\oB_{v',r}^J\simeq (\overline{\sigma^J_{r,+}(X')}\cap\oB^J_{v',r})\times\dot r\JB^+/\JB^+\simeq \overline{\sigma_{r,+}^J(X')}\cap\oB^J_{v',r}.$$
The composition is the identity map. The second equality is proved similarly.
\end{proof}

\begin{lemma} \label{preserve_connected_component}
Suppose that $Y$ is a connected component of $\oB^J_{v,w}(\mathbb R)$. Assume that $Y\subseteq\dot r\JU^-\JB^+/\JB^+$. Let $\overline Y$ be the Hausdorff closure of $Y$ in $\mathcal B^J(\mathbb R)$. Then
\begin{enumerate}
    \item $\dot r\JB^+/\JB^+\in\overline Y$;
    \item $\overline Y\cap\oB^J_{v,r} = \sigma_{r,+}^J(Y)$ is a connected component of $\oB^J_{v,r}(\mathbb R)$;
    \item $\overline Y\cap\oB^J_{r,w} = \sigma_{r,-}^J(Y)$ is a connected component of $\oB^J_{r,w}(\mathbb R)$.
\end{enumerate}

\begin{proof}
 For any $u\in\JU^-$, there is a finite sequence $(\alpha_1,\alpha_2,\cdots,\alpha_m)$ of roots in $(\Phi^-\setminus\Phi_J^-)\sqcup\Phi_J^+$ such that $u\in\prod_{i=1}^mU_{\alpha_i}$. Then there exists a coweight $\mu$, such that $\lim_{t\to0}\mu(t)u\mu(t)^{-1} = e$. Let $\mu' = r(\mu)$. Then $\lim_{t\to0^+}\mu'(t)\dot ru\JB^+/\JB^+ = \dot r\JB^+/\JB^+$. Since $\oB^J_{v,w}$ is stable under $T$-action, and $Y$ is a connected component of $\oB^J_{v,w}(\mathbb R)$, we know that $Y$ is stable under $T_{>0}$-action. This proves part (1). 

Since $\sigma_r^J$ is an isomorphism, we see that $\sigma_{r,+}^J(Y)$ is a connected component of $\oB^J_{v,r}(\mathbb R)$, and thus is closed in $\oB^J_{v,r}$. We have $\sigma_{r,+}^J(Y) = \overline{\sigma_{r,+}^J(Y)}\cap\oB^J_{v,r}$, which equals $\overline{Y}\cap\oB^J_{v,r}$ by \Cref{alg_cut}. This proves part (2); part (3) is proved similarly.
\end{proof}

\end{lemma}

\subsection{Main result of this section}
The main purpose of this section is to establish the following theorem, which parallels \Cref{thm:product}.

\begin{theorem} \label{thm:J_product}
    Let $Y_{v,w} $ be a connected component  of $\oB^J_{v,w}(\BR)$. For any $ v \le v'\le w' \le w$, define 
\[
	Y_{v', w'}  = \overline{Y_{v,w}}  \bigcap \oB^J_{v', w'}.
\]
Assume that for any $ v \le^J v'\le^J r \le^J w' \le^J w$, we have $Y_{v', w'} \subseteq  \dot{r} \JB^- \JB^+ /\JB^+$. Then 
\begin{enumerate}
    \item The map $\sigma^J_{r}$ restricts to an isomorphism $\sigma^J_{r} : Y_{v', w'} \cong Y_{v', r} \times Y_{r, w'}$;

    \item The closure $\overline{Y_{v,w}}=\bigsqcup_{v\leq^J v'\leq^J w'\leq^J w} Y_{v', w'}$ is a remarkable polyhedral space;

     \item The closure $\overline{Y_{v,w}}=\bigsqcup_{v\leq^J v'\leq^J w'\leq^J w} Y_{v', w'}$ has the regularity property. 
\end{enumerate}
\end{theorem}

The general strategy is similar to the proof of \Cref{thm:product} in \cite{BH24}. However, there are additional difficulties in this complicated situation. First, we need some geometric properties of $\oB^J_{v,w}$, which we establish in \S\ref{KL_properties}--\S\ref{sec:dim}. Second, while in the ordinary case one studies the links of $Y_{v,w}$ using highest weight representations, such a theory is not available in the twisted situation. Instead, we embed certain subsets of each $\overline Y_{v,w}$ into finite-dimensional vector spaces and obtain parallel results in \S\ref{links}. 

We prove Theorem \ref{thm:J_product} (1) in \S\ref{component_product_structure}, (2) in \S\ref{general_remarkable}, and (3) in \S\ref{general_regularity}.

\subsection{Nonemptiness pattern of $\oB^J_{v,w}$} \label{KL_properties}
We have the following description of the nonemptiness pattern of $\oB^J_{v,w}$, generalizing \S\ref{sec:flag} (a). 
\begin{proposition} \label{nonempty_pattern}
Let $v, w \in W$. The following are equivalent:
\begin{enumerate}
    \item $\oB_{v,w}^J\neq\emptyset$;
    \item $\mathcal B^J_{v,w}\neq\emptyset$;
    \item $v\leq^J w$.
\end{enumerate}

\begin{proof}
Clearly (1) $\Rightarrow$ (2). For (2) $\Rightarrow$ (3), first note that all $T$-fixed points in $\mathcal B^J$ are $\{\dot u\JB^+/\JB^+|u\in W\}$. Since $\mathcal B^J_{v,w}$ is projective and $T$-invariant, $\mathcal B^J_{v,w}$ contains a $T$-fixed point, say $\dot u\JB^+/\JB^+$. Since $\mathcal B^J_{v,w}\subseteq\mathcal B^J_w =\sqcup_{w'\leq^J w}\oB^J_{w'}$, we have $u\leq^Jw$. Since $\mathcal B^J_{v,w}\subseteq\mathcal B^{J,v} = \sqcup_{v\leq^J v'}\oB^{J,v'}$, we have $v\leq^Ju$. Then $v\leq^Jw$.

Now we show that (3) $\Rightarrow$ (1). Consider the isomorphisms
$$(U^+\cap\dot v\JU^-\dot v^{-1})\times (U^-\cap \dot v\JU^-\dot v^{-1})\to\dot v\JU^-\dot v^{-1}\to \dot v\JU^-\JB^+/\JB^+$$
$$(u_1,u_2)\to u_1u_2\to u_1u_2\dot v\JB^+/\JB^+.$$
Since $U^-\cap\dot v\JU^-\dot v^{-1}$ is isomorphic to $\oB^{J,v}$ via $u_2\to u_2\dot v\JB^+/\JB^+$, we have
$$(U^+\cap\dot v\JU^-\dot v^{-1})\times\oB^{J,v}\to\dot v\JU^-\JB^+/\JB^+.$$
Since $\oB^J_w$ is invariant under left multiplication by $U^+$, via restriction, we have
$$(U^+\cap\dot v\JU^-\dot v^{-1})\times (\oB^{J,v}\cap\oB^J_w)\simeq \dot v\JU^-\JB^+/\JB^+\cap\oB^J_w.$$
It is sufficient to show that $\dot v\JU^-\JB^+/\JB^+\cap\oB^J_w\neq\emptyset$. Since $\dot v\JU^-\JB^+/\JB^+$ is open, we only need to show that $\dot v\JU^-\JB^+/\JB^+\cap\mathcal B^J_w\neq\emptyset$. But it contains the point $\dot v\JB^+/\JB^+$ since $v\leq^J w$.
\end{proof}

\end{proposition}

\begin{corollary} \label{product_structure_nonempty}
Let $v\leq^Jr\leq^J w$. Then $\oB_{v,w}^J\cap\dot r\JU^-\JB^+/\JB^+\neq\emptyset$ and $\dot r\JB^+/\JB^+\in\overline{\oB^J_{v,w}\cap\dot r\JU^-\JB^+/\JB^+}$.

\begin{proof}
The nonemptyness follows from \Cref{nonempty_pattern} and the isomorphism
$$\sigma_r^J:\oB^J_{v,w}\cap\dot r\JU^-\JB^+/\JB^+\simeq\oB^J_{v,r}\times\oB^J_{r,w}.$$
The last statement can be proved similarly as \Cref{preserve_connected_component} (1).
\end{proof}
\end{corollary}

\subsection{Closure relation of $\oB^J_{v,w}$}
We have the following description of the nonemptiness pattern of $\oB^J_{v,w}$, generalizing \S\ref{sec:flag} (c). 

\begin{proposition} \label{KL_closure}
We have $$\mathcal B^J_{v,w} = \bigsqcup_{v\leq^J v'\leq^J w'\leq^J w}\oB^J_{v',w'}.$$

\begin{proof}
By \cite[Theorem 4]{BD94}, we have $\mathcal B^J_{v,w}\subseteq\mathcal B_w^J\cap\mathcal B^{J,v} = \bigsqcup_{v\leq^J v'\leq^J w'\leq^J w}\oB^J_{v',w'}$.

Let $r\in W$ with $v\leq^J r\leq^J w$. Set $X = \oB^J_{v,w}\cap\dot r\JU^-\JB^+/\JB^+$. By \Cref{product_structure_nonempty}, $X\neq\emptyset$ and $\dot r\JB^+/\JB^+\in\overline{X}$. We apply \Cref{alg_cut} to $X$ to obtain
$$\overline{X}\cap\oB^J_{v,r} = \overline{\sigma_{r,+}^J(X)}\cap\oB^J_{v,r} = \mathcal B^J_{v,r}\cap\oB^J_{v,r} = \oB^J_{v,r};$$
$$\overline{X}\cap\oB^J_{r,w} = \overline{\sigma_{r,-}^J(X)}\cap\oB^J_{r,w} = \mathcal B^J_{r,w}\cap\oB^J_{r,w} = \oB^J_{r,w}.$$
This shows that $\oB^J_{v,r}\subseteq\mathcal B_{v,w}^J$ and $\oB^J_{r,w}\subseteq\mathcal B_{v,w}^J$ for any $v\leq^Jr\leq^J w$.

Now suppose $v\leq^Jv'\leq^J w'\leq^J w$. If $v = v'$, then taking $r = w'$ in the previous claim gives $\oB^J_{v,w'}\subseteq\mathcal B^J_{v,w}$. Otherwise, we have $\oB^J_{v',w}\subseteq\mathcal B^J_{v,w}$, and then $\oB^J_{v',w'}\subseteq\mathcal B_{v',w}^J\subseteq\mathcal B^J_{v,w}$.
\end{proof}
\end{proposition}

\subsection{Dimension of $\oB^J_{v,w}$}\label{sec:dim}
We first establish the results for the dimension $0$ and dimension $1$ cases. 

\begin{lemma}
For every $w\in W$, we have $\oB^J_{w,w} = \{\dot w\JB^+/\JB^+\}$.
\begin{proof}
    Let $\pi:G/\JB^+\to G/P_J^+$ be the natural projection. We have $\pi(\oB^J_{w,w})\subseteq B^+\dot wP_J^+/P_J^+\cap B^-\dot wP_J^+/P_J^+ = \{\dot w^JP_J^+/P_J^+\}$. Thus every element in $\oB^J_{w,w}$ is of the form $\dot w^Jl\JB^+/\JB^+$ for some $l\in L_J$. Now $\dot w^Jl\JB^+/\JB^+\in\oB^J_{w,w}$ implies that $lB_J^-\in B_J^+\dot w_JB_J^-/B_J^-\cap B_J^-\dot w_JB_J^-/B_J^- = \{\dot w_JB_J^-/B_J^-\}$. 
\end{proof}
\end{lemma}

\begin{proposition} \label{one_dim_KL}
Let $v,w\in W$ such that $v\leq^J w$ and $\ell^J(w)-\ell^J(v) = 1$. Then $\oB^J_{v,w}(\mathbb R)\simeq\mathbb R^\times$.

\begin{proof}
We prove by induction on $\ell(w^J)$. If $\ell(w^J) = 0$, then $w\in W_J$. We have $v\in W_J$ and $w\leq v$ since $v\leq^J w$. The natural isomorphism $\oB^J_w\simeq B_J^+\dot wB_J^-/B_J^-$ identifies $\oB_{v,w}^J$ with an ordinary open Richardson variety of $L_J$. The statement follows from \cite[Theorem 5]{BD94}.

Suppose that $\ell(w^J)>0$. Take $i\in I$ such that $s_iw^J<w^J$. We have $s_iw^J\in W^J$. By \cite[Corollary 1]{BD94}, an element in $\oB^J_w$ can be written as either 
\begin{itemize}
    \item $y_i(a)h\JB^+/\JB^+$ for some $a\neq0$ and $h\in B^+\dot s_i\dot w\JB^+$, or 
    \item $\dot s_ih'\JB^+/\JB^+$ for some $h'\in B^+\dot s_i\dot w\JB^+$.
\end{itemize}

Suppose there exists $y_i(a)h\JB^+/\JB^+\in\oB^J_{v,w}$ for some $a\neq0$ and $h\in B^+\dot s_i\dot w\JB^+$. Then $h\JB^+/\JB^+\in\oB^J_{v,s_iw}$. Since $\ell^J(s_iw) = \ell^J(w) - 1 = \ell^J(v)$, we have $v = s_iw$ by \Cref{nonempty_pattern} and $h\JB^+/\JB^+ = \dot v\JB^+/\JB^+$. It is clear that $y_i(a)\dot v\JB^+/\JB^+\in\oB^J_{v,w}$ for every $a\neq0$. We claim that there is no element of the form $\dot s_ih'\JB^+/\JB^+$ for some $h'\in B^+\dot s_i\dot w\JB^+$, and the claim follows. If there is, since $v<^J s_iv = w$, we have $h'\JB^+/\JB^+\in\oB^J_{v,s_iw}\cup\oB^J_{s_iv,s_iw}$ by \cite[Corollary 1]{BD94}. Since $s_iv = w\nleq^Js_iw$, we have $\oB^J_{s_iv,s_iw} = \emptyset$, and $\dot s_ih'\JB^+/\JB^+ = \dot w\JB^+/\JB^+\in\oB^J_{v,w}$. That is a contradiction. 

Suppose every element in $\oB^J_{v,w}$ is of the form $\dot s_ih'\JB^+/\JB^+$ for some $h'\in B^+\dot s_i\dot w\JB^+$. Then we have $h'\JB^+/\JB^+\in\oB^J_{v,s_iw}\cup\oB^J_{s_iv,s_iw}$. First assume that $h'\JB^+/\JB^+\in\oB^J_{v,s_iw}$. Since $\ell^J(s_iw)=\ell^J(w)-1=\ell^J(v)$, we have $\oB^J_{v,s_iw}\neq\emptyset$ if and only if $v = s_iw$, in which case, $\dot s_ih'\JB^+/\JB^+ = \dot w\JB^+/\JB^+\in\oB^J_{v,w}$. That is a contradiction. Thus, we have $h'\JB^+/\JB^+\in\oB^J_{s_iv,s_iw}$, and $\oB^J_{v,w}\simeq\oB^J_{s_iv,s_iw}$. Induction applies.
\end{proof}

\end{proposition}

We have the following dimension formula of $\oB^J_{v,w}$, generalizing \S\ref{sec:flag} (b). 

\begin{proposition} \label{dimension_formula}
Let $v,w\in W$ such that $v\leq^J w$. Then the variety $\oB^J_{v,w}$ is of dimension $\ell^J(w)-\ell^J(v)$.

\begin{proof}
Recall the isomorphism $$\sigma_r^J:\oB^J_{v,w}\cap\dot r\JU^-\JB^+/\JB^+\simeq\oB^J_{v,r}\times\oB^J_{r,w}.$$
It follows from \Cref{one_dim_KL} and induction that
\begin{enumerate}[label = (\alph*)]
    \item $\dim\oB^J_{v,w}\geq \ell^J(w)-\ell^J(v)$ for every $v,w\in W$ such that $v\leq^J w$.
\end{enumerate}

Next we show that 
\begin{enumerate}[label = (\alph*)]
    \setcounter{enumi}{1}
    \item $\dim\oB^J_{v,w'} \leq \ell^J(w')-\ell^J(v)$ for $w'\in W^J$.
\end{enumerate}
Let $\pi:G/\JB^+\to G/P_J^+$ be the natural projection. We have $$\pi(\oB^J_{v,w'})\subseteq B^+\dot w' P_J^+/P_J^+\cap B^-\dot v^JP_J^+/P_J^+.$$ Thus $\dim\pi(\oB^J_{v,w'})\leq \ell(w')-\ell(v^J)$. Let $g\in U^-\dot v^J$, then every element in $\pi^{-1}(gP_J^+/P_J^+)$ can be written as $gl\JB^+/\JB^+$ for some $l\in L_J$. By the Bruhat decomposition of $L_J$, we have $gl\JB^+/\JB^+\in\oB^{J,v}$ if and only if $lB_J^-\in B_J^-\dot v_JB_J^-/B_J^-$. This shows $\dim\left(\pi^{-1}(gP_J^+/P_J^+)\bigcap\oB^J_{v,w}\right)\leq \ell(v_J)$. Thus, $$\dim\oB^J_{v,w'}\leq \ell(w')-\ell(v^J)+\ell(v^J) = \ell^J(w')-\ell^J(v).$$

Finally, note that for $w\in W^J$, we have $w\leq^J w^J$. Consider the isomorphism $$\sigma_w^J:\oB^J_{v,w^J}\cap\dot w\JU^-\JB^+/\JB^+\simeq\oB^J_{v,w}\times\oB^J_{w,w^J}.$$
Together with (b), it implies 
\begin{enumerate}[label = (\alph*)]
    \setcounter{enumi}{2}
    \item $\dim\oB^J_{v,w}\leq \ell^J(w)-\ell^J(v)$ for every $v,w\in W$ such that $v\leq^J w$.
\end{enumerate}

The proposition now follows from (a) and (c).
\end{proof}

\end{proposition}

\subsection{Proof of \Cref{thm:J_product} (1)} \label{component_product_structure}
The proof is similar to \cite[Theorem 3.5]{BH24}. First, by \Cref{preserve_connected_component}, for every $r$ such that $v\leq^Jr\leq^J w$, $Y_{v,r} = \sigma_{r,+}^J(Y_{v,w})$ is a connected component of $\oB^J_{v,r}(\mathbb R)$ and $Y_{r,w} = \sigma_{r,-}^J(Y_{v,w})$ is a connected component of $\oB^J_{r,w}(\mathbb R)$.  

By our assumption on $Y_{v',w}$ and \Cref{preserve_connected_component}, we have 
$$\overline{Y_{v',w}}\bigcap\oB^J_{v',w'}(\mathbb R) = \sigma^J_{w',+}(Y_{v',w}).$$ 
Applying \Cref{alg_cut} to $Y_{v,w}$, we obtain
$$Y_{v',w'} = \overline{\sigma^J_{v',-}(Y_{v,w})}\bigcap\oB^J_{v',w'}(\mathbb R) = \overline{Y_{v',w}}\bigcap\oB^J_{v',w'} = \sigma^J_{w',+}(\sigma^J_{v',-}(Y_{v,w})).$$
In particular, $Y_{v',w'}$ is a connected component of $\oB^J_{v',w'}(\mathbb R)$. Similarly, using \Cref{alg_cut} for $v^{\prime\prime},w^{\prime\prime}$ such that $v'\leq^Jv^{\prime\prime}\leq^Jw^{\prime\prime}\leq^Jw'$, we have $Y_{v^{\prime\prime},w^{\prime\prime}} = \overline{Y_{v',w'}}\bigcap\oB^J_{v^{\prime\prime},w^{\prime\prime}}(\mathbb R)$.

The statement now follows from the fact that $\sigma^J_r$ restricts to an isomorphism $$\oB^J_{v',w'}\cap\dot r\JU^-\JB^+/\JB^+\xrightarrow[]{\sim}\oB^J_{v',r}\times\oB^J_{r,w'}$$ for every $(v,w)$ with $v'\leq^Jr\leq^Jw'$.

\subsection{Proof of \Cref{thm:J_product} (2)} \label{general_remarkable}

In \S\ref{component_product_structure}, we established
that $Y_{v',w'}$ is a connected component of $\oB^J_{v',w'}(\mathbb R)$.

By \Cref{KL_closure}, we have
$$\overline{Y_{v',w'}} = \bigsqcup_{v'\leq^Jv^{\prime\prime}\leq^Jw^{\prime\prime}\leq^Jw'}\left(\overline{Y_{v^{\prime},w^{\prime}}}\bigcap\oB^J_{v^{\prime\prime},w^{\prime\prime}}\right) = \bigsqcup_{v'\leq^Jv^{\prime\prime}\leq^Jw^{\prime\prime}\leq^Jw'}Y_{v^{\prime\prime},w^{\prime\prime}}.$$
The last equality was established in \S\ref{component_product_structure}, verifying the closure relation.

Finally, we show that $Y_{v',w'}\simeq\mathbb R_{>0}^{\ell^J(w')-\ell^J(v')}$. By \Cref{one_dim_KL}, when $\ell^J(w')-\ell^J(v')=1$, we have $\oB^J_{v',w'}(\mathbb R)\simeq\mathbb R^\times$, so $Y_{v',w'}\simeq\mathbb R_{>0}$. The general case follows by induction using the product structure from \S\ref{component_product_structure}.

\subsection{Contractive Flows, Links and Cones} \label{links}
Suppose $v'\in W$ and $v\leq^J v'\leq^J w$. We first embed $\mathcal B^J_{v,w}\cap U^-\dot v'\JB^+/\JB^+$ into a finite-dimensional vector space. 

Let $X^{+}$ be the set of dominant weights of $G$. For $\lambda \in X^+$, let $\Lambda_\lambda$ denote the highest weight simple $G$-module with highest weight $\lambda$ and a fixed highest weight vector $\eta_\lambda$. Let ${}^\omega \Lambda_\lambda$ denote the lowest weight simple $G$-module with lowest weight $-\lambda$ and a fixed lowest weight vector $\xi_{-\lambda}$.

For $K\subseteq I$, let $\lambda_K$ be a dominant weight such that the stabilizer of $\lambda_K$ in $W$ is the standard parabolic subgroup $W_K$. Then the morphism $$\JU^-\to\Lambda_{\lambda_J}\otimes \prescript{\omega}{}{\Lambda_{\lambda_{(I\setminus J)}}},\ \ u\mapsto u\cdot(\eta_{\lambda_J}\otimes\xi_{-\lambda_{(I\setminus J)}})$$ 
is an embedding.

For $v,w,v'\in W$ such that $v\leq^J v'\leq^J w$, we consider the following embedding
\begin{align*}
    \mathcal B^J_{v,w}\cap U^-\dot v'\JB^+/\JB^+&\xrightarrow{(\dot v')^{-1}\cdot-} \JU^-\JB^+/\JB^+\simeq\JU^- \\
    &\xrightarrow{u\mapsto u\cdot(\eta_{\lambda_J}\otimes\xi_{-\lambda_{(I\setminus J)}})}\Lambda_{\lambda_J}\otimes {}^{\omega}{}{\Lambda_{\lambda_{(I\setminus J)}}}
\end{align*}

\begin{lemma} \label{finite_roots}
Let $v\leq^Jv'\leq^J w$. There exist finite subsets $\{\beta_1,\beta_2,\cdots,\beta_k\}\subseteq \Phi^+\setminus\Phi^-_J$ and $\{\beta_{k+1},\beta_{k+2},\cdots,\beta_m\}\subseteq\Phi_J^+$ such that $$(\dot v')^{-1}\mathcal B^J_{v,w}\cap \JU^-\JB^+/\JB^+\subseteq\prod_{i=1}^m U_{\beta_i}\JB^+/\JB^+.$$

\begin{proof}
Consider the partial flag variety $\mathcal P_J = G/P_J^+$ and the natural projection $\pi:\mathcal B^J\to\mathcal P_J$. There exists $z\in W^J$ depending on $w$ and $v'$, such that
$$\pi\left((\dot v')^{-1}\mathcal B^J_{v,w}\cap \JU^-\JB^+/\JB^+\right)\subseteq\overline{B^+\dot z P_J^+/P_J^+}\cap U^-P_J^+/P_J^+.$$
By \cite[Theorem 5]{BD94}, there exists a finite subset $\{\beta_1,\beta_2,\cdots,\beta_k\}\subseteq\Phi^-\setminus\Phi^-_J$, such that 
$$\pi\left((\dot v')^{-1}\mathcal B^J_{v,w}\cap\JU^-\JB^+/\JB^+\right)\subseteq\prod_{i=1}^k U_{\beta_i}P_J^+/P_J^+.$$
For $g\JB^+/\JB^+\in (\dot v')^{-1}\mathcal B^J_{v,w}\cap\JU^-\JB^+/\JB^+$, there exists $u\in\prod_{i=1}^kU_{\beta_i}$, such that $$u^{-1}g\JB^+/\JB^+\in U_J^+\JB^+/\JB^+\cap U^-(\dot v')^{-1}\overline{U^-\dot v\JB^+/\JB^+}.$$ 
By \cite[Corollary 1]{BD94} and \Cref{nonempty_pattern}, there exists $z'\in W_J$ depending only on $v$ and $v'$ (and independent of  $g$) such that 
$$u^{-1}g\JB^+/\JB^+\in U_J^+\JB^+/\JB^+\cap\overline{U_J^-\dot z'\JB^+/\JB^+}.$$
By \cite[Theorem 5]{BD94}, there exists a finite subset $\{\beta_{k+1},\cdots,\beta_m\}\subseteq\Phi^+_J$, depending only on $v$ and $v'$ (and independent of $g$) such that $$u^{-1}g\JB^+/\JB^+\in\prod_{i=k+1}^mU_{\beta_i}\JB^+/\JB^+.$$ The result follows.
\end{proof}

\end{lemma}

Let $L^{v'}_{v,w}$ be the vector subspace of $\Lambda_{\lambda_J}(\mathbb R)\otimes \prescript{\omega}{}{\Lambda_{\lambda_{(I\setminus J)}}}(\mathbb R)$ spanned by vectors of the form $u\cdot(\eta_{\lambda_J}\otimes\xi_{-\lambda_{(I\setminus J)}})$ for some $u\in\prod_{i=1}^m U_{\beta_i}$. The vector space $L^{v'}_{v,w}$ is finite-dimensional. The image of $\mathcal B^J_{v,w}(\mathbb R)\cap U^-\dot v'\JB^+/\JB^+$ lies in $L^{v'}_{v,w}$. We identify $\mathcal B^J_{v,w}(\mathbb R)\cap U^-\dot v'\JB^+/\JB^+$ with its image in $L^{v'}_{v,w}$. We pick a basis for $L^{v'}_{v,w}$, and equip it with an Euclidean norm $\lVert\cdot\rVert$.

Recall the assumption in \Cref{thm:J_product}. We have $Y_{v,w}$ a connected component  of $\oB^J_{v,w}(\BR)$, and for any $ v \le v'\le w' \le w$, we have
\[
	Y_{v', w'}  = \overline{Y_{v,w}}  \bigcap \oB^J_{v', w'}.
\]
We assume that for any $ v \le^J v'\le^J r \le^J w' \le^J w$, we have $Y_{v', w'} \subseteq  \dot{r} \JB^- \JB^+ /\JB^+$.

We define the links $Lk_? = Lk_?(\overline{Y_{v,w}})$ of $\overline{Y_{v,w}}$ by
$$Lk_{v',w'} = Y_{v',w'}\bigcap\left\{x\in L^{v'}_{v,w}\Big|\lVert x\rVert = 1\right\},\text{ for }w'\text{ such that }v'<^J w'\leq^J w,$$
$$Lk_{v'} = \bigsqcup_{v'<^Jw'\leq^J w}Y_{v',w'} = \overline{Y_{v,w}}\bigcap\left\{x\in L^{v'}_{v,w}\Big|\lVert x\rVert = 1\right\}.$$ 

Let $\{\beta_i\}_{i=1}^m$ be the (finite) subset of roots as in \Cref{finite_roots}. There exists a coweight $\mu$ such that $$\lim_{a\to 0}\mu(a)\left(\prod_{i=1}^m U_{\beta_i}\right)\mu(a)^{-1} = e,\ \lim_{a\to 0}\mu(a)\cdot(\eta_{\lambda_J}\otimes\xi_{-\lambda_{(I\setminus J)}}) = 0$$
Let $\theta_\mu$ be the $\mathbb R^\times$-action on $L^{v'}_{v,w}$ via $\mu$. Since $\oB^J_{v,w}(\mathbb R)$ is stable under the action of $T(\mathbb R)$, and every connected component of $\oB^J_{v,w}(\mathbb R)$ is stable under the action of $T_{>0}$, the map $\theta_\mu$ defines a contractive flow on the vector space $L^{v'}_{v,w}$ in the sense of \cite[Definition 2.2]{GKL22}. According to \cite[Lemma 3.4]{GKL22}, we have
\begin{enumerate}[label = (\alph*)]
    \item \label{scaling_number} For every $x\in\overline{Y_{v,w}}\cap U^-\dot v'\JB^+/\JB^+$ with $x\neq\dot v'\JB^+/\JB^+$, there exists a unique $t(x)\in\mathbb R_{>0}$ such that $\theta_\mu(t(x))x\in Lk_{v'}$. Moreover, the map $x\mapsto t(x)$ is continuous.
\end{enumerate} 
Given a topological space $A$, the cone over $A$ is defined as $\text{Cone}(A) := (A\times\mathbb R_{\geq0})/(A\times\{0\})$. Let $D^n$ be the $n$-dimensional closed ball. Note that $\text{Cone}(D^n)\simeq\mathbb R^n\times\mathbb R_{\geq0}$.

Consider the map $\overline{Y_{v,w}}\cap U^-\dot v'\JB^+/\JB^+\to\text{Cone}(Lk_{v'})$ defined by mapping $\dot v'\JB^+/\JB^+$ to the cone point and 
$$Y_{v',w'}\to Lk_{v',w'}\times\mathbb R_{>0},\ x\to (\theta_\mu(t(x))x,t(x)),\text{ for }w'\neq v'.$$
This map is bijective by \ref{scaling_number} and stratified. Following the argument in the proof of \cite[Proposition 3.5]{GKL22}, we have
\begin{enumerate}[label = (\alph*)] 
  \setcounter{enumi}{1}
  \item \label{cone_over_link} the map $\overline{Y_{v,w}}\cap U^-\dot v'\JB^+/\JB^+\to\text{Cone}(Lk_{v'})$ defined above is a stratified isomorphism. 
\end{enumerate}

\subsection{Regularity of links}
In this subsection, we show that $Lk_{v'}$ is a regular CW complex.
\begin{proposition} \label{link_isomorphism}
For $v'\leq^Jr<^Jw$, we have stratified isomorphisms
\begin{align*}
Lk_{v'}\cap\dot r\JU^-\JB^+/\JB^+ &= \bigsqcup_{r\leq^Jw'\leq^J w} Lk_{v',w'} \\
&\simeq Lk_{v',r}\times (\overline{Y_{v,w}}\cap U^-\dot r\JB^+/\JB^+) \\ & \simeq Lk_{v',r}\times\textup{Cone}(Lk_r).
\end{align*}

\begin{proof}
By our assumption on $Y$, we have $Lk_{v'}\cap\dot r\JU^-\JB^+/\JB^+ = \bigsqcup_{r\leq^Jw'\leq^J w} Lk_{v',w'}$. The last isomorphism follows from \S\ref{links} \ref{cone_over_link}.  We construct the second isomorphism.

Define a morphism $\alpha$ as the composition
\begin{align*}
Lk_{v',r}\times (\overline{Y_{v,w}}\cap U^-\dot r\JB^+/\JB^+)&\hookrightarrow Y_{v',r}\times (\overline{Y_{v,w}}\cap U^-\dot r\JB^+/\JB^+) \\
\xrightarrow[]{(\sigma_r^J)^{-1}} \bigsqcup_{r\leq^Jw'\leq^J w} Y_{v',w'}&\xrightarrow[]{x\mapsto \theta_\mu(t(x))x}\bigsqcup_{r\leq^Jw'\leq^J w} Lk_{v',w'}.
\end{align*}
Define $\beta$ as the composition
\begin{align*}
 \bigsqcup_{r\leq^Jw'\leq^J w} Lk_{v',w'}&\hookrightarrow \bigsqcup_{r\leq^Jw'\leq^J w} Y_{v',w'} \xrightarrow[]{\sigma_r^J} Y_{v',r}\times (\overline{Y_{v,w}}\cap U^-\dot r\JB^+/\JB^+) \\ & \xrightarrow[]{\phi} Lk_{v',r}\times(\overline{Y_{v,w}}\cap U^-\dot r\JB^+/\JB^+),
\end{align*}
where $\phi(x,y) = (\theta_\mu(t(x))x,\theta_\mu(t(x))y)$. We claim that $\beta$ is the second isomorphism in the statement, with the inverse $\alpha$. 

Let $(x,y)\in Lk_{v',r}\times (\overline{Y_{v,w}}\cap U^-\dot r\JB^+/\JB^+)$ and $z = (\sigma_r^J)^{-1}(x,y)$. Then $\alpha(x,y) = \theta_\mu(t(z))z$. Since $\sigma_r^J$ is $T$-equivariant, we have $$\sigma_r^J(\theta_\mu(t(z))z) = (\theta_\mu(t(z))x,\theta_\mu(t(z))y).$$ 
By the uniqueness of \S\ref{links}\ref{scaling_number}, we have $t(\theta_\mu(t(z))x) = t(z)^{-1}$. Therefore, $$\phi(\theta_\mu(t(z))x,\theta_\mu(t(z))y) = \beta(z) = (x,y).$$ So $\beta\circ\alpha = \text{id}$. The claim for $\alpha\circ\beta= \text{id}$ is similar.
\end{proof}

\end{proposition}

\begin{corollary} \label{link_cell}
Let $v'\leq^J w'$. We have $Lk_{v',w'}\simeq\mathbb R_{>0}^{\ell^J(w')-\ell^J(v')-1}$.

\begin{proof}
Thanks to \Cref{thm:J_product} (1) proved in \S\ref{component_product_structure} and \Cref{link_isomorphism}, it suffices to show $Lk_{v',w'}$ is a point when $\ell^J(w')-\ell^J(v')=1$. The latter statement follows from a direct computation.
\end{proof}

\end{corollary}

\begin{proposition} \label{link_regularity}
For $v\leq^Jv'\leq^J w$, the space $Lk_{v'} = \bigsqcup_{v\leq^Jw'\leq^J w}Lk_{v',w'}$ is a regular CW complex homeomorphic to a closed ball of dimension $\ell^J(w')-\ell^J(v')-1$.

\begin{proof}
We prove by induction on $\ell^J(w)-\ell^J(v')$. The base case $\ell^J(w)-\ell^J(v) = 1$ follows from \Cref{link_cell}. In the induction step, we assume that the statement holds for $Lk_?(\overline{Y_{v,w'}})$ with $w'\leq^J w$. Note that $Lk_{v'}(\overline{Y_{v,w'}})$ is a subspace of $Lk_{v'}(\overline{Y_{v,w}})$ and $Lk_{v',w^{\prime\prime}}(\overline{Y_{v,w'}}) = Lk_{v,w^{\prime\prime}}(\overline{Y_{v,w}})$ for $v'<^Jw^{\prime\prime}<^Jw'$.

We first show that
\begin{enumerate}[label = (\alph*)]
    \item $Lk_{v'}$ is a topological manifold with boundary $\partial Lk_{v'} = \bigsqcup_{v'<^J w'<^J w}Lk_{v',w'}$.
\end{enumerate}
Note that $Lk_{v',w} = Lk_{v'}\cap\dot w\JU^-\JB^+/\JB^+\simeq\mathbb R^{\ell^J(w)-\ell^J(v')-1}_{>0}$. For any $r\in W$ with $v'<^Jr<^J w$, applying the stratified isomorphism from \Cref{link_isomorphism} yields:
\begin{align*}
    Lk_{v'}\cap \dot r\JU^-\JB^+/\JB^+ &= \bigsqcup_{r\leq^J w'\leq^J w}Lk_{v',w'} \\ & \simeq Lk_{v',r}\times\text{Cone}(Lk_r) \\
    &\simeq\mathbb R^{\ell^J(r)-\ell^J(v')-1}_{>0}\times\text{Cone}(D^{\ell^J(w)-\ell^J(r)-1}) \\
    &\simeq\mathbb R^{\ell^J(r)-\ell^J(v')-2}\times\mathbb R_{\geq0},
\end{align*}
where $Lk_r\simeq D^{\ell^J(w)-\ell^J(r)-1}$ is obtained by the induction hypothesis. This shows that $Lk_{v'}$ is a topological manifold with boundary, and $Lk_{v',r}$ lies on the boundary for $r\neq w$. This proves (a).

Next, we show 
\begin{enumerate}[label = (\alph*)]
\setcounter{enumi}{1}
    \item $\partial Lk_{v'} = \bigsqcup_{v'<^J w'<^J w}Lk_{v',w'}$ is a regular CW complex homeomorphic to a sphere of dimension $\ell^J(w)-\ell^J(v)-2$.
\end{enumerate}
By the induction hypothesis, $Lk_{v'}(\overline{Y_{v,w'}})$ is a regular CW complex homeomorphic to a closed ball of dimension $\ell^J(w')-\ell^J(v)-1$ for every $w'$ with $v'\leq^J w'\leq^J w$. Therefore $\partial Lk_{v'}$ is a regular CW complex with face poset $(\{w'|v'<^J w'<^J w\},\leq^J)$. After adding a new minimum $\hat 0$ and maximum $\hat 1$, the poset becomes $\{w'|v'\leq^Jw'\leq^J w\}$, which is graded, thin, and shellable by \Cref{Weyl_shellable}. Thus, by \Cref{shellable_regular}, $\partial Lk_{v'}$ is a regular CW complex homeomorphic to a sphere of dimension $\ell^J(w)-\ell^J(v)-2$, proving (b)

Finally, the statement follows from (a), (b) and \Cref{thm:poincare}.
\end{proof}

\end{proposition}

\subsection{Proof of \Cref{thm:J_product} (3)} \label{general_regularity}

The outline of the proof is similar to \Cref{link_regularity} and \cite[\S4.2]{BH24}. We proceed by induction on $\ell^J(w)-\ell^J(v)$. The base case $\dim\oB^J_{v,w}(\mathbb R)= 0$ occurs when $\ell^J(w)-\ell^J(v)=0$, in which case $\overline{Y_{v,w}}$ is a single point. 

We first prove

\begin{enumerate}[label = (\alph*)]
    \item $\overline{Y_{v,w}}$ is a topological manifold with boundary $$\partial \overline{Y_{v,w}} = \bigsqcup_{v\leq^Jv'\leq^J w'\leq^Jw,(v',w')\neq(v,w)}Y_{v',w'}.$$
\end{enumerate}

By \S\ref{general_remarkable}, we have $Y_{v,w}\simeq\mathbb R_{>0}^{\ell^J(w)-\ell^J(v)}$. By \Cref{link_regularity} and \S\ref{links}\ref{cone_over_link}, we have 
$$\overline{Y_{v,w}}\bigcap\dot v\JU^-\JB^+/\JB^+ = \bigsqcup_{v\leq^J w'\leq^J w}Y_{v,w'}\simeq\text{Cone}(D^{\ell^J(w)-\ell^J(v)-1}).$$
Thus, we have $Y_{v,w'}\subseteq\partial\overline{Y_{v,w}}$ for every $w'$ such that $v\leq^J w'<^J w$.

Let $v'\in W$ with $v<^Jv'<^J w$. Applying the statement of \Cref{thm:J_product} (1) proved in \S\ref{component_product_structure}, we have 
\begin{align*}
    \overline{Y_{v,w}}\cap\dot v'\JU^-\JB^+/\JB^+ &= \bigsqcup_{v\leq^Jv^{\prime\prime}\leq^Jv'\leq^Jw^{\prime\prime}\leq^J w}Y_{v^{\prime\prime},w^{\prime\prime}} \\
    &\simeq \bigsqcup_{v\leq^Jv^{\prime\prime}\leq^Jv'}Y_{v^{\prime\prime},v'} \times \bigsqcup_{v'\leq^Jw^{\prime\prime}\leq^J w}Y_{v',w^{\prime\prime}}.
\end{align*}
By the induction hypothesis, $\overline{Y_{v,v'}}$ and $\overline{Y_{v',w}}$ are regular CW complexes homeomorphic to closed balls of dimensions $\ell^J(v')-\ell^J(v)-1$ and $\ell^J(w)-\ell^J(v')-1$, respectively. Therefore, 
\begin{align*}
    \overline{Y_{v,w}}\cap\dot v'\JU^-\JB^+/\JB^+ &\simeq \text{Cone}(D^{\ell^J(v')-\ell^J(v)-1})\times \text{Cone}(D^{\ell^J(w)-\ell^J(v')-1}) \\
    &=\mathbb R^{\ell^J(w)-\ell^J(v)-1}\times\mathbb R_{\geq0}.
\end{align*}
This shows that $Y_{v',w'}\subseteq\partial\overline{Y_{v,w}}$ for every $(v',w')$ with $v<^Jv'\leq^J w'\leq^J w$ and $v'\neq w$.

A similar argument, using a variation of \Cref{link_regularity} and \S\ref{links}\ref{cone_over_link}, shows that $Y_{v',w}\subseteq\partial\overline{Y_{v,w}}$ for every $v'$ such that $v\leq^J v'\leq^J w$. This proves (a).

By \Cref{Weyl_shellable} and \Cref{shellable_regular}, the boundary $\partial\overline{Y_{v,w}}$ is a regular CW complex homeomorphic to a sphere of dimension $\ell^J(w)-\ell^J(v)-1$. The result now follows from \Cref{thm:poincare}.

\section{Totally Nonnegative Twisted Flag Varieties} \label{TNN_twsited}
The totally nonnegative twisted flag variety is defined to be
$$\CB^J_{\geq0} = \overline{G_{\geq0}\JB^+/\JB^+}\text{  (Hausdorff closure in $\CB^J(\mathbb R)$)}.$$
Let $\CB^J_{v,w,>0}:=\CB^J_{\geq0}\bigcap\oB^J_{v,w}(\mathbb R)$. Then $\CB^J_{\geq0} = \bigsqcup_{v\leq^Jw}\CB^J_{v,w,>0}$. This section is devoted to the study of $\CB^J_{\geq0}$ with this decomposition.

\subsection{Marsh-Rietsch parametrization}\label{subsec:marsh-rietsch}
A major difficulty in applying Theorem~\ref{thm:J_product} to the study of $\CB^J_{\geq0}$ is verifying the condition $\CB^J_{v, w, >0} \subseteq \dot r\,^J U^-\,^J B^+/\,^J B^+$ for all $v \leq^J r \leq^J w$. In the ordinary flag variety case, this condition is verified using the Marsh-Rietsch parametrization \cite{MR04} of totally positive cells. We now recall this construction, which will serve as motivation for our generalization in the twisted setting.

An \textit{expression} is a sequence $\underline w = (t_1, t_2,\cdots, t_n)$, where $t_j\in \{s_i:i\in I\}\sqcup\{1\}$. We write $w = t_1t_2\cdots t_n\in W$ and say that $\underline w$ is an expression of $w$. The expression is \emph{reduced} if $n = \ell(w)$.

A \emph{subexpression} of $\underline w$ is a sequence $\underline w' = (t_1',t_2',\cdots, t_n')$ with $t_i'\in\{t_i,1\}$. If $\underline w$ is a reduced expression, the subexpression $\underline w'$ is called \emph{positive} if 
$t_1' t_2'\cdots t_{j-1}' < t_1' t_2'\cdots t_{j-1}' t_j$ for all $j$.

By \cite[Lemma 3.5]{MR04}, for any $v\leq w$ and any reduced expression $\underline{w}$ of $w$, there exists a unique positive subexpression for $v$ in $\underline{w}$, denoted by $\underline{v}_+$.

Let $v\leq w$ and fix a reduced expression $\underline w$ for $w$. Let $\underline v_+$ be the positive subexpression for $v$ in $\underline{w}$. Define
\[
G^-_{\underline v_+,\underline w,>0} := \left\{g_1g_2\cdots g_n \,\middle|\, 
\begin{array}{ll}
g_j = \dot s_{i_j}, & \text{if }t_j = s_{i_j}; \\
g_j\in y_{i_j}(\mathbb R_{>0}), & \text{if }t_j = 1
\end{array}
\right\}.
\]

It was shown in \cite[Theorem 11.3]{MR04} for reductive groups and in \cite[Theorem 4.10]{BH21a} for Kac-Moody groups that 
\[
\mathcal B^+_{v,w,>0} = G^-_{\underline v_+,\underline w,>0}B^+/B^+.
\]
The proof is based on the representation-theoretic interpretation of totally nonnegative flag varieties, involving the theory of the canonical basis. 

We also record some related notation for future use:
\begin{gather*} G^+_{\underline v_+,\underline w,>0} := \{g_1g_2\cdots g_n\ |\ g_j = \dot s_{i_j}\text{ if }t_j = s_{i_j}; g_j\in x_{i_j}(\mathbb R_{>0})\text{ if }t_j = 1\},\\
G^-_{\underline v_+,\underline w,\neq0} := \{g_1g_2\cdots g_n\ |\ g_j = \dot s_{i_j}\text{ if }t_j = s_{i_j}; g_j\in y_{i_j}(\mathbb R^\times)\text{ if }t_j = 1\},\\
G^+_{\underline v_+,\underline w,\neq0} := \{g_1g_2\cdots g_n\ |\ g_j = \dot s_{i_j}\text{ if }t_j = s_{i_j}; g_j\in x_{i_j}(\mathbb R^\times)\text{ if }t_j = 1\}.
\end{gather*}

\subsection{The set $G^J_{(v, w), >0}$}\label{sec:J-par}
We introduce a parametrizing set $G^J_{(v,w),>0}\subseteq G$ in analogue with the Marsh-Rietsch parametrization. In the untwisted case, the Marsh-Rietsch parametrization uses a single reduced expression for $w$. For the $J$-twisted case, our parametrization must account for the coset structure $W/W_J$. Our parametrizing set decomposes as a product of two factors:
\begin{itemize}
\item The factor $G^-_{\underline{v^J c}_+,\underline{w^J},>0}$ handles contributions from $W^J$, and
\item The factor $G^+_{\underline{w_J}_+,\underline{c^{-1}v_J},>0}$ handles contributions within the parabolic subgroup $W_J$.
\end{itemize}

From this point forward, we assume that for each $w \in W$, a reduced expression $\underline w$ is chosen.

For $v\leq^J w$, there exists $u\in W_J$ such that $v^Ju\leq w^J$ and $w_J\leq u^{-1}v_J$. By \cite[Lemma 1.4]{He07}, there is a unique minimal such $u$. We denote this element by 
\[
c = \min\{u\in W_J : v^Ju\leq w^J \text{ and } w_J\leq u^{-1}v_J\}.
\]
This minimal element satisfies the length relations:
\begin{equation}\tag{a}
\ell(cw_J) = \ell(w_J)-\ell(c) \quad \text{and} \quad \ell(c^{-1}v_J) = \ell(c^{-1})+\ell(v_J).
\end{equation}

\begin{definition}
For $v\leq^J w$, we choose reduced expressions  $\underline{w^J},\underline{v_J}$. For any reduced expression $\underline{c^{-1}}$, put $\underline{c^{-1}v_J} = (\underline{c^{-1}},\underline{v_J})$. We define
\begin{align*}
G^J_{(v,w),\neq0} &:= G^-_{\underline{v^Jc}_+,\underline{w^J},\neq0} \cdot G^+_{\underline{w_J}_+,\underline{c^{-1}v_J},\neq0}, \\
G^J_{(v,w),>0} &:= G^-_{\underline{v^Jc}_+,\underline{w^J},>0} \cdot G^+_{\underline{w_J}_+,\underline{c^{-1}v_J},>0}.
\end{align*}
\end{definition}

By the length relations~(a), we have $G^J_{(v,w),\neq0}\simeq (\mathbb R^\times)^{\ell^J(w)-\ell^J(v)}$ and $G^J_{(v,w),>0}\simeq (\mathbb R_{>0})^{\ell^J(w)-\ell^J(v)}$. Moreover, by our choice of $\underline{c^{-1}v_J}$, we have 
$$G^+_{\underline{w_J}_+,\underline{c^{-1}v_J},\neq0} = \dot c^{-1}G^+_{\underline{cw_J}_+,\underline{v_J},\neq0},$$ 
and thus these sets do not dependent on the choice of $\underline{c^{-1}}$.

\begin{lemma}\label{lem:injective}
The map 
\[
G^J_{(v,w),\neq 0} \to \mathcal{B}^J, \quad g \mapsto g\,^J B^+ / \,^J B^+
\]
is injective. Moreover, we have the inclusion
\[
G^J_{(v,w),\neq 0} \,^J B^+ / \,^J B^+ \subseteq \oB^J_{v,w}.
\]

\begin{proof}
We first prove injectivity. 

Suppose $g_1, g_2 \in G^-_{\underline{v^J c}_+, \underline{w^J}, \neq 0}$, $h_1, h_2 \in G^+_{\underline{w_J}_+, \underline{c^{-1} v_J}, \neq 0}$ such that $g_1 h_1 \,^J B^+ / \,^J B^+ = g_2 h_2 \,^J B^+ / \,^J B^+$. Then there exist $b \in B_J^-$ and $u \in U_{P_J^+}$ such that $g_1 h_1 = g_2 h_2 b u$.

Projecting to $G / P_J^+$, we obtain $g_1 P_J^+ = g_2 P_J^+$. By \cite[Proposition 5.2]{MR04}, the map $G^-_{\underline{v^J c}_+, \underline{w^J}, \neq 0} \to G / P_J^+$ is injective, hence $g_1 = g_2$.

Now we have $h_1 = h_2 b u$. Since $h_1, h_2, b \in L_J$ and $u \in U_{P_J^+}$ with $U_{P_J^+} \cap L_J = \{e\}$, it follows that $u = e$. Therefore, $h_1 = h_2 b$, and projecting to $L_J / B_J^-$ yields $h_1 B_J^- = h_2 B_J^-$. Again by \cite[Proposition 5.2]{MR04} (applied to $L_J$), the map $G^+_{\underline{w_J}_+, \underline{c^{-1} v_J}, \neq 0} \to L_J / B_J^-$ is injective, so $h_1 = h_2$. This establishes the injectivity of the map.

For the inclusion, we use the following containment from \cite[Proposition 5.2]{MR04}:
\begin{align*}
G^-_{\underline{v^J c}_+, \underline{w^J}, \neq 0} &\subseteq B^+ \dot{w}^J B^+ \cap U^- \dot{v}^J \dot{c}, \\
G^+_{\underline{w_J}_+, \underline{c^{-1} v_J}, \neq 0} &= \dot{c}^{-1} G^+_{\underline{c w_J}_+, \underline{v_J}, \neq 0} \subseteq \dot{c}^{-1} \left( B_J^- \dot{v}_J B_J^- \cap U_J^+ \dot{w}_J \right).
\end{align*}
It follows that 
\[
G^J_{(v,w),\neq 0} \,^J B^+ / \,^J B^+ \subseteq \oB^J_{v,w},
\]
completing the proof.
\end{proof}

\end{lemma}

By definition, the sets $G^J_{(v,w),\neq0}\JB^+/\JB^+$ and $G^J_{(v,w),>0}\JB^+/\JB^+$ depend a priori on the choice of reduced expressions $\underline{w^J}$ and $\underline{v_J}$. However, the Marsh-Rietsch parametrization applying to $(L_J/B_J)_{\geq0}$ yields that $G^+_{\underline{w_J}_+,\underline{c^{-1}v_J},>0}\JB^+/\JB^+$ does not depend on the choice of $\underline{c^{-1}v_J}$. Therefore, the set $G^J_{(v,w),>0}\JB^+/\JB^+$ does not depend on the choice of $\underline{v_J}$, and we can even remove the assumption that $\underline{c^{-1}v_J}$ is of the form $(\underline{c^{-1}},\underline{v_J})$. We will show in \S\ref{proof_parametrization} that $\CB^J_{v,w,>0} = G^J_{(v,w),>0}\JB^+/\JB^+$, and thus, it does not dependent on the choice of $\underline{w^J}$ either.

\begin{remark} \label{one_dim_KL_parametrization}
By \Cref{one_dim_KL} and its proof, we have $$\oB^J_{v,w}(\mathbb R) = G^J_{(v,w),\neq0}\JB^+/\JB^+$$ when $\ell^J(w)-\ell^J(v) = 1$.
\end{remark} 

\subsection{Main result}

We state the main results on the totally nonnegative twisted flag variety $\CB^J_{\geq0}$.

\begin{theorem}\label{thm:J-flag-main}
Let $v,w\in W$ with $v\leq^J w$. Then 
\begin{enumerate}
    \item For any choices of reduced expressions $\underline{w^J}$ and $\underline{v_J}$, we have
$$\mathcal B^J_{v,w,>0} = G^J_{(v,w),>0}\JB^+/\JB^+.$$
\item For any $r\in W$ such that $v\leq^Jr\leq^J w$, the map $\sigma_r^J$ restricts to an isomorphism
    $$\sigma_r^J:\mathcal B^J_{v,w,>0}\simeq\mathcal B^J_{v,r,>0}\times\mathcal B^J_{r,w,>0}.$$
    \item $\mathcal B^J_{\geq0}$ is a remarkable polyhedral space with regularity property.
\end{enumerate}
\end{theorem}

The difficulty in establishing this theorem is twofold. On one hand, when $W_J$ is infinite, the theory of ''twisted" highest weight representations does not exist. Consequently, we cannot adapt the strategy from \cite{MR04} to establish the parametrization. On the other hand, without the parametrization result at hand, we are not able to describe the action of $G_{\geq0}$ on the candidate set $G^J_{(v,w),>0}\JB^+/\JB^+$, which is important for us to check the key condition
$$G^J_{(v,w),>0}\JB^+/\JB^+\subseteq\dot r\JU^-\JB^+/\JB^+,\text{ for }v\leq^Jr\leq^Jw.$$

Our strategy is to take three steps which depend on the restriction we post on indexing pair $(v,w)$:
\begin{enumerate}
    \item Step 1: $v\in W_J$, $w\in W^J$.
    \item Step 2: $w\in W^J$ and no restriction on $v$.
    \item Step 3: general situation.
\end{enumerate}
At each step, we show that:
\begin{itemize}

\item the candidate set $G^J_{(v,w),>0}\JB^+/\JB^+$ is a connected component of $\oB^J_{v,w}(\mathbb R)$;

\item the candidate set $G^J_{(v,w),>0}\JB^+/\JB^+$ equals to $ \CB^J_{v,w,>0}$ and thus is independent of the choices of reduced expressions $\underline{w^J}$ and $\underline{v_J}$;

\item the candidate set satisfies the inclusion property
$$ G^J_{(v,w),>0}\JB^+/\JB^+\subseteq\dot r\JU^-\JB^+/\JB^+, \quad \text{for all } v\leq^J r\leq^J w.$$

\end{itemize}
These properties at each step enables us to apply the machinery of product structure to establish the same properties in the next step. 

We remark that, in the proof of \Cref{thm:J-flag-main}, we use the parametrization for positive cells of $\CB^+_{\geq0}$ established in \cite{MR04,BH21a}. However, the proof is independent of the fact that $\CB^+_{\geq0}$ is a remarkable polyhedral space with regularity property.

\subsection{Some technical lemmas} \label{tech}
In this subsection, we collect some technical lemmas.

\begin{lemma} \label{conj_positive}
Let $i,j\in I$ and $w\in W$. If $\dot wx_i(a)\dot w^{-1} = x_j(b)$, then $ab\geq0$. Similarly, if $\dot wy_i(a)\dot w^{-1} = y_j(b)$, then $ab\geq0$.

\begin{proof}
It is clear that $a = 0$ if and only if $b = 0$. Without loss of generality, we assume that $a>0$. Then $x_j(b)\dot wB^+/B^+ = \dot wx_i(a)B^+/B^+\in\mathcal B_{w,ws_i,>0}^+$. By the parametrization of positive cells in \cite[Theorem 11.3]{MR04}, it follows that $b > 0$.
\end{proof}

\end{lemma}

\begin{lemma} \label{springer_fiber_KL}
Let $i_0\in I$ such that $u^{-1}(\alpha_{i_0}) > 0$ for every $v^Jc \leq u \leq w^J$. We have
\begin{enumerate}
    \item If $s_{i_0}v^J \in W^J$, then 
    \[
    x_{i_0}(a)h\,^J B^+/\,^J B^+ = h\,^J B^+/\,^J B^+\text{, for all } h \in G^J_{(v,w),>0},\ a \in \mathbb{R}.
    \]
    
    \item If $s_{i_0}v^J = v^Js_j$ for some $j \in J$, then for all $g \in G^-_{\underline{v^Jc}_+,\underline{w^J},>0}$, $g' \in G^+_{\underline{cw_J}_+,\underline{v_J},>0}$, and $a \in \mathbb{R}^\times$,
    \[
    x_{i_0}(a)g\dot{c}^{-1}g'\,^J B^+/\,^J B^+ \in g\dot{c}^{-1}x_j(a\mathbb{R}_{>0})g'\,^J B^+/\,^J B^+.
    \]
    In particular, if $s_jcw_J<s_jv_J$ and $u^{\prime,-1}(\alpha_j) > 0$ for every $s_jcw_J \leq u' \leq s_jv_J$, then
    \[
    x_{i_0}(a)h\,^J B^+/\,^J B^+ = h\,^J B^+/\,^J B^+ \quad \text{for all } h \in G^J_{(v,w),>0},\ a \in \mathbb{R}.
    \]
\end{enumerate}

\begin{remark}
For $v \leq w$ and a reduced expression $\underline{w}$ for $w$, the condition that $u^{-1}(\alpha_{i_0}) > 0$ for every $u$ with $v \leq u \leq w$ is equivalent to $(s_{i_0},\underline{v}_+)$ being a positive subexpression in $(s_{i_0},\underline{w})$.
\end{remark}

\begin{proof}
Let $\underline{w^J} = (s_{i_1}, s_{i_2}, \cdots, s_{i_n})$ and $\underline{v^Jc}_+ = (t_1, t_2, \cdots, t_n)$. Define $v_{(k)} := t_1t_2\cdots t_k$ and $K_{\underline{v^Jc}_+} := \{k \mid v_{(k+1)} = v_{(k)}\}$.

Let $p_J \colon U^+ \to U^+_J$ be the natural projection. Our assumption on $i_0$ implies by \cite[Lemma 2.2.4]{He05} that 
\[
g^{-1}x_{i_0}(a)g \in U^+ \quad \text{for all } a \in \mathbb{R},\ g \in G^-_{\underline{v^Jc}_+,\underline{w^J},>0}.
\]
Since $g'U_{P_J^+}g^{\prime,-1} = U_{P_J^+}$ for $g'\in G^+_{\underline{w_J}_+,\underline{c^{-1}v_J},>0}\subseteq L_J$, we shall consider the interaction between $p_J(g^{-1}x_{i_0}(a)g)$ and $G^+_{\underline{w_J},\underline{c^{-1}v_J},>0}$. The element $p_J(g^{-1}x_{i_0}(a)g)\in U_J^+$ can be written as product of root group elements for some roots in $\Phi_J^+$, and we shall examine the roots appearing in such a product.

Consider
$$R_{\underline{v^Jc}_+,\underline{w^J}} := \left\{\beta\in\mathbb N_+\Phi\bigg|\ \beta =(v^Jc)^{-1}\left(n_0\alpha_{i_0}-\sum_{k\in K_{\underline{v^Jc}_+}}n_kv_{(k)}(\alpha_{i_k})\right),n_0\in\mathbb N_+, n_k\in\mathbb N\right\}.$$
By \cite[Lemma 2.2.4]{He05}, we have
$g^{-1}x_{i_0}(a)g\in\prod_{\beta\in R_{\underline{ v^Jc}_+,\underline{w^J}}\bigcap\Phi^+}U_\beta$.

Our assumptions yield:
\begin{itemize}
    \item since $\underline{v^Jc}_+$ is a positive subexpression in $\underline{w}$, we have $v_{(k)}(\alpha_{i_k}) > 0$ for all $k$;
    \item since $(s_{i_0},\underline{v^Jc}_+)$ is a positive subexpression in $(s_{i_0},\underline{w})$, we have $v_{(k)}(\alpha_{i_k}) \neq \alpha_{i_0}$ for all $k$.
\end{itemize}

Now suppose $\beta \in R_{\underline{v^Jc}_+,\underline{w^J}} \cap \Phi_J^+$. Then
\[
v^Jc(\beta) = n_0\alpha_{i_0} - \sum_{k \in K_{\underline{v^Jc}_+}} n_k v_{(k)}(\alpha_{i_k}).
\]
We have $v^Jc(\beta) > 0$ if and only if $n_0 = 1$ and $n_k = 0$ for all $k$. Otherwise $v^Jc(\beta) < 0$, and then $c(\beta) \in \Phi_J^-$.

By assumption, $s_{i_0}v^Jc>v^Jc$. We consider two cases:
\begin{enumerate}[label=(\roman*)]
    \item $s_{i_0}v^J > v^J$ and $s_{i_0}v^J \in W^J$. Then for $\beta \in \Phi_J^+$, we have $v^Jc(\beta) \neq \alpha_{i_0}$. Thus $\beta \in \Phi_J^- \cap c^{-1}(\Phi_J^-)$. Part (1) now follows from \cite[Lemma 11.8]{MR04} or \cite[Proposition 2.2.3]{He05}.
    
    \item $s_{i_0}v^J = v^Js_j$ for some $j\in J$ and $s_jc > c$. Then for $\beta \in \Phi_J^+$ with $v^Jc(\beta) > 0$, we must have $c(\beta) = \alpha_j$. Thus $\beta \in \Phi_J^+ \cap (s_jc)^{-1}(\Phi_J^-)$. By \cite[Lemma 11.8]{MR04} or \cite[Proposition 2.2.3]{He05}, for every $a \in \mathbb{R}^\times$, there exists $b \in \mathbb{R}^\times$ such that
    \[
    x_{i_0}(a)G^J_{(v,w),>0}\,^J B^+/\,^J B^+ \subseteq G^-_{\underline{v^Jc}_+,\underline{w^J},>0} \dot{c}^{-1} x_j(b) G^+_{\underline{cw_J}_+,\underline{v_J},>0}\,^J B^+/\,^J B^+.
    \]
    Moreover, $ab > 0$ by \Cref{conj_positive}.
\end{enumerate}

The final statement of part (2) also follows from \cite[Lemma 11.8]{MR04} or \cite[Proposition 2.2.3]{He05}.
\end{proof}

\end{lemma}

\begin{corollary} \label{x_action_part_1}
Let $v,w\in W$ such that $v\leq^J w$. Suppose $i\in I$ and $v<^Js_iv$. Then for $a>0$, we have
$$x_i(a)G^J_{(v,w),>0}\JB^+/\JB^+\subseteq G^J_{(v,w),>0}\JB^+/\JB^+.$$
\begin{proof}
Let $c = \min\{c'\in W_J|w_J\leq (c')^{-1}v_J\}$. Then $v^Jc\leq w^J$. By assumption, either $s_iv^J>v^J$ and $s_iv^J\in W^J$, or $s_iv^J = v^Js_j$ and $s_jv_J<v_J$. Since $\ell(c^{-1}v_J) = \ell(c^{-1})+\ell(v_J)$, we have $s_jc>c$. In either case, $s_iv^Jc>v^Jc$.

We fix a reduced expression $\underline{w^J} = (s_{i_1},s_{i_2},\cdots, s_{i_n})$. Denote $\underline{v^Jc}_+ = (t_1,t_2,\cdots,t_n)$. Let $g\in G^-_{\underline{v^Jc}_+,\underline{w^J},>0}$ and $g'\in G_{\underline{w_J},\underline{c^{-1}v_J},>0}^+$. We consider two cases:
\begin{enumerate}[label = (\roman*)]
    \item $s_iv^Jc\nleq w^J$. Then it follows from \cite[Corollary 1]{He09} that $s_iw^J>w^J$ and $s_iv^Jc\leq s_iw^J$, and from \cite[\S1.9(a)]{L21} that $(s_i,\underline{v^Jc}_+)$ is a positive subexpression of $(s_i,\underline{w^J})$. If $s_iv^J\in W^J$, the conclusion follows from \Cref{springer_fiber_KL} (1). Otherwise, $s_iv^J = v^Js_j$ for some $j\in J$ and $s_jv_J<v_J$. By \Cref{springer_fiber_KL} (2), there is $b>0$, such that $x_i(a)g\dot c^{-1}g'\JB^+/\JB^+ = g\dot c^{-1}x_j(b)g'\JB^+/\JB^+$. Now it is sufficient to show that 
    $$x_j(b)g'B_J^-/B_J^-\in G^+_{\underline{cw_J}_+,\underline{v_J},>0}B_J^-/B_J^- = (L_J/B_J)_{cw_J,v_J,>0}.$$ 
    This follows from \cite[Proposition 5.2]{BH21a} (see also \cite[\S1.10]{L21}).
    
    \item $s_iv^Jc\leq w^J$. Let $j$ be the smallest number such that $t_j = s_{i_j}$ and $t_k = 1$ for $k<j$. Let $w_1 = s_{i_1}s_{i_2}\cdots s_{i_{j-1}}$ and $w_2 = s_{i_j}s_{i_{j+1}}\cdots s_{i_n}$. Then there is $b>0$ such that
    $$x_i(a)gh\JB^+/\JB^+\in U^-_{w_1,>0}x_i(b)G^J_{(v,w_2w_J),>0}\JB^+/\JB^+.$$
    It is clear that $s_iv^Jc\nleq w_2$. The conclusion from the previous case applies.
\end{enumerate}
The proof is complete.
\end{proof}
\end{corollary}

Next, we consider an operation on the Weyl group. For $i\in I$ and $w\in W$, we define $s_i\circ^J_lw$ by
\begin{equation*}
    s_i\circ^J_l w := \left\{
\begin{aligned}
    &s_iw &&\text{ if }s_iw\leq^J w,  \\
    & w &&\text{ if }w\leq^J s_iw.
\end{aligned}
\right.
\end{equation*}
It follows from definition that
\begin{itemize}
    \item if $s_i\circ_l^J w = s_iw$, then $w^{-1}(\alpha_i)\in\Phi_J^+\sqcup(\Phi^-\setminus\Phi_J^-)$;
    \item if $s_i\circ_l^J w = w$, then $w^{-1}(\alpha_i)\in\Phi_J^-\sqcup(\Phi^+\setminus\Phi_J^+)$.
\end{itemize}
This combinatorial operation has a geometric description. Write $w' = s_i\circ_l^J w$. It follows from \cite[Corollary 1]{BD94} that
$$\overline{B^-\dot s_iB^-\dot w\JB^+/\JB^+} = \overline{B^-\dot w'\JB^+/\JB^+}.$$

\begin{lemma} \label{twisted_downward}
Let $v,w\in W$ such that $v\leq^J w$ and $i\in I$. Then $s_i\circ^J_l v\leq^J s_i\circ_l^Jw$.

\begin{proof}
Write $v' = s_i\circ_l^J v$ and $w' = s_i\circ_l^Jw$.

Since $v\leq^J w$, we have $B^-\dot v\JB^+/\JB^+\subseteq\overline{B^-\dot w\JB^+/\JB^+}$. For $b\in B^-$, we have
$$b\dot s_iB^-\dot v\JB^+/\JB^+\subseteq b\dot s_i\overline{B^-\dot w\JB^+/\JB^+} = \overline{b\dot s_iB^-\dot w\JB^+/\JB^+}\subseteq\overline{B^-\dot w'\JB^+/\JB^+}.$$
Therefore, $B^-\dot s_iB^-\dot v\JB^+/\JB^+\subseteq\overline{B^-\dot w'\JB^+/\JB^+}$, and then $\overline{B^-v'\JB^+/\JB^+}\subseteq\overline{B^-\dot w'\JB^+/\JB^+}$. We conclude that $v'\leq^J w'$.
\end{proof}

\end{lemma}

\subsection{Candidate sets are nonnegative}
Before tackling the general parametrization, we verify that our candidate sets $G^J_{(v,w),>0}\JB^+/\JB^+$ are contained in $\mathcal{B}^J_{\geq 0}$.

\begin{proposition} \label{candidates_are_nonnegative}
For $v,w\in W$ such that $v\leq^Jw$, we have $$G^J_{(v,w),>0}\JB^+/\JB^+\subseteq\mathcal{B}^J_{\geq 0}.$$

\begin{proof}
We proceed by induction on $\ell(w^J)$. 
For the base case $\ell(w^J) = 0$, we have $w = w_J\in W_J$ and $v = v_J\in W_J$. By \cite[Theorem 4.10]{BH21a} applying to $(L_J/B_J)_{\geq0}$, we have $G^+_{\underline{w_J},\underline{v_J},>0}\JB^+/\JB^+\subseteq\overline{U_{v_J,>0}^+\JB^+/\JB^+}$.

We assume that $\ell(w^J)>0$. Let $c$ be the minimal element in $W_J$ such that $w_J\leq c^{-1}v_J$. Then $v^Jc\leq w^J$. We write $\underline{w} = (s_{i_1},s_{i_2},\cdots,s_{i_n})$ and $\underline{v^Jc}_+ = (t_1,t_2,\cdots,t_n)$. Let $w_k = s_{i_k}s_{i_{k+1}}\cdots s_{i_n}w_J$, $w_k = (w_k)^Jw_J$ and $v_k = t_kt_{k+1}\cdots t_nc^{-1}v_J$ for $k = 1,2,\cdots, n$. Then $c_k = (t_kt_{k+1}\cdots t_n)_J$ and $(v_k)_J = c_kc^{-1}v_J$. Note that $c_{k+1}$ is a suffix of $c_k$, for $k=1,2,\cdots, n-1$.

Set $w_{n+1} = w_J$ and $v_{n+1} = c^{-1}v_J$, we have
$$G^J_{(v_{n+1},w_{n+1}),>0}\JB^+/\JB^+ = G^+_{\underline{w_J},\underline{c^{-1}v_J},>0}\JB^+/\JB^+\subseteq\CB^J_{\geq0}.$$

We assume that $G_{(v_{k+1},w_{k+1}),>0}^J\JB^+/\JB^+\subseteq\CB^J_{\geq0}$. If $t_k = 1$, we have 
$$G^J_{(v_k,w_k),>0}\JB^+/\JB^+ = y_{i_k}(\mathbb R_{>0})G^J_{(v_{k+1},w_{k+1}),>0}\JB^+/\JB^+\subseteq\CB^J_{\geq0}.$$
Now suppose $t_k = s_{i_k}$. Let $h\in G^J_{(v_{k+1},w_{k+1}),>0}$. For $a>0$, consider
$$x_{i_k}(a)\dot s_{i_k}h\JB^+/\JB^+ = \alpha_{i_k}^\vee(a)y_{i_k}(a)x_{i_k}(-a^{-1})h\JB^+/\JB^+.$$

We have two cases:
\begin{enumerate}[label=(\roman*)]
    \item $s_{i_k}v_k^J>v_k^J$. It follows from \Cref{springer_fiber_KL} (1) that 
    $$x_{i_k}(-a^{-1})h\JB^+/\JB^+ = h\JB^+/\JB^+.$$
  \item $s_{i_k}v_{k+1}^J = v_{k+1}^Js_j$ for some $j\in J$, and $s_jc_{k+1} = c_k$. Since $c_{k+1}$ is a suffix of $c_k$ and they are both suffixes of $c$, it follows from \Cref{springer_fiber_KL} (2) that
   $$x_{i_k}(-a^{-1})h\JB^+/\JB^+ = h\JB^+/\JB^+.$$
\end{enumerate}

In both cases, we have by induction:
\[
x_{i_k}(a) \dot{s}_{i_k} h\,^J B^+/\,^J B^+ = \alpha_{i_k}^\vee(a) y_{i_k}(a) h\,^J B^+/\,^J B^+ \in \mathcal{B}^J_{\geq 0}.
\]

Taking the limit as $a \to 0^+$ on the left-hand side, we have $\dot{s}_{i_1} h\,^J B^+/\,^J B^+ \in \mathcal{B}^J_{\geq 0}$, completing the induction.
\end{proof}
\end{proposition}

\subsection{Inclusion property in the case $v\in W_J$ and $w\in W^J$} \label{inclusion_1}
The goal in this subsection is to show that, for $v\in W_J$ and $w\in W^J$, we have
$$G^J_{(v,w),>0}\JB^+/\JB^+\subseteq\dot r\JU^-\JB^+/\JB^+,\text{ for any }v\leq^Jr\leq^Jw.$$

We begin by recalling some fundamental notations and results about totally positive cells in the monoid $G_{\geq0}$. Let $T_{>0}$ be the identity component of $T(\mathbb R)$.

For $w \in W$ with a fixed reduced expression $\underline{w} = s_{i_1}s_{i_2}\cdots s_{i_n}$, define:
\begin{align*}
U^-_{w,>0} &= \{y_{i_1}(a_1)y_{i_2}(a_2)\cdots y_{i_n}(a_n) \mid a_j \in \mathbb{R}_{>0}\}, \\
U^+_{w,>0} &= \{x_{i_1}(a_1)x_{i_2}(a_2)\cdots x_{i_n}(a_n) \mid a_j \in \mathbb{R}_{>0}\}.
\end{align*}
Note that, by definition, $U^\pm_{w,>0} = G^\pm_{\underline{1}_+,\underline{w},>0}$.

Lusztig \cite{L94} showed that these sets are independent of the choice of reduced expression $\underline{w}$, and we have the decompositions:
\[
U^-_{\geq 0} = \bigsqcup_{w \in W} U^-_{w,>0}, \quad U^+_{\geq 0} = \bigsqcup_{w \in W} U^+_{w,>0}.
\]
In particular, when $v\in W_J$ and $w\in W^J$, we have $G_{(v,w)>0}^J = U_{w,>0}^-U_{v,>0}^+$ is independent of the choices of reduced expressions $\underline{w}$ and $\underline{v}$. Moreover, for $w_1, w_2\in W$, we have
\begin{equation}\tag{*}
   U_{w_1,>0}^-T_{>0}U^+_{w_2,>0} = U_{w_2,>0}^+T_{>0}U_{w_1,>0}^-.  
\end{equation}

Now we prove the main result of this subsection.

\begin{proposition} \label{key_assumption_1}
Suppose that $v\in W_J$ and $w\in W^J$. Then for every $r\in W$ such that $v\leq^Jr\leq^J w$, we have 
$$G^J_{(v,w),>0}\JB^+/\JB^+\subseteq\dot r\JU^-\JB^+/\JB^+.$$

\begin{proof}
We argue by induction on $\ell(w)$. For $\ell(w) = 0$, we have $r^J = 1$ and $r_J\leq v$, and the statement follows from \cite[Lemma 5.1]{BH24}. 

Fix a reduced expression $\underline{w} = (s_{i_1},s_{i_2},\cdots,s_{i_n})$. Set $w' = s_{i_1}w$, $r' = s_{i_1}\circ_l^J r$, and $v' = s_{i_1}\circ_l^J v$ (\S\ref{tech}). It follows from \Cref{twisted_downward} that $v'\leq^J r'\leq^J w'$. Let $g\in G_{(v,w),>0}^J$.

We divide the computation into two cases.
\begin{enumerate}[label = (\roman*)]
    \item $r'= r$. Then $v\leq^J r\leq^J w'$, and $\dot r^{-1}y_{i_1}(\mathbb R^\times)\dot r\in\JU^-$. We have  
    $$\dot r^{-1}G_{(v,w),>0}^J\subseteq\JU^-\dot r^{-1}G^J_{(v,w'),>0}.$$
    \item $r' = s_{i_1}r$. Then $\dot r^{-1}x_{i_1}(\mathbb R^\times)\dot r\subseteq\JU^-$. Write $g = y_{i_1}(a)g'$ with $a\in\mathbb R_{>0}$ and $g'\in G_{(v,w'),>0}^J$. Using the identity
    $$y_{i_1}(a) = x_{i}(a^{-1})\dot s_{i_1}x_{i_1}(a)\alpha^\vee_{i_1}(a),$$
    we have
    \begin{align*}
        \dot r^{-1}y_{i_1}(a)g'\JB^+/\JB^+ &=  \left(\dot r^{-1}x_{i_1}(a^{-1})\dot r\right)(\dot r')^{-1}x_{i_1}(a)\alpha^\vee_{i_1}(a)g'\JB^+/\JB^+\\
        &\in\JU^-(\dot r')^{-1}x_{i_1}(a)U_{w',>0}^-U_{v,>0}^+\JB^+/\JB^+.
    \end{align*}
    By \S\ref{inclusion_1} (*), we have
    $x_{i_1}(a)U_{w',>0}^-U_{v,>0}^+\JB^+/\JB^+\subseteq G_{(v',w'),>0}^J\JB^+/\JB^+$, and thus, $\dot r^{-1}g\JB^+/\JB^+ \in\JU^-(\dot r')^{-1}G_{(v',w'),>0}^J\JB^+/\JB^+$.
\end{enumerate}
The statement follows from induction assumption.
\end{proof}

\end{proposition}

\subsection{Connected component in the case $w \in W^J$} \label{special_case}
In this subsection, we show that for $w\in W^J$ and $v\in W$ with $v\leq^Jw$, 
\begin{enumerate}[label = (\alph*)]
\item $G^J_{(v,w),>0}\JB^+/\JB^+$ is a connected component of $\oB^J_{v,w}(\mathbb R)$. \label{special_connected_component}
\end{enumerate}

We establish this result through a series of lemmas, first proving openness and then closedness in the Hausdorff topology of $\oB^J_{v,w}(\mathbb{R})$.

\begin{lemma}\label{lem:openness}
Suppose that $w\in W^J$. The set $G^J_{(v,w),>0}\JB^+/\JB^+$ is an open subset in $\oB^J_{v,w}(\mathbb R)$.  
\begin{proof}
Consider the natural projection $\pi_J \colon \mathcal{B}^J \to \mathcal{P}_J$ and its restriction
\[
\pi'_J \colon \oB^J_{v,w} \to B^+ \dot{w} P_J^+ / P_J^+ \cap B^- \dot{v}^J P_J^+ / P_J^+.
\]
We have
\[
\pi_J(G^J_{(v,w),>0}\,^J B^+/\,^J B^+) = G^-_{\underline{v^J}_+,\underline{w},>0} P_J^+ / P_J^+,
\]
which is open in $B^+ \dot{w} P_J^+ / P_J^+ \cap B^- \dot{v}^J P_J^+ / P_J^+$. 

Therefore, the preimage $(\pi'_J)^{-1}(G^-_{\underline{v^J}_+,\underline{w},>0} P_J^+ / P_J^+)$ is open in $\oB^J_{v,w}(\mathbb{R})$.

Using the Bruhat decomposition $L_J = \bigsqcup_{u \in W_J} B^-_J \dot{u} B_J^-$ and Birkhoff decomposition $L_J = \bigsqcup_{u \in W_J} B^+_J \dot{u} B_J^-$, we obtain the inclusion:
\[
(\pi'_J)^{-1}(G^-_{\underline{v^J}_+,\underline{w},>0} P_J^+ / P_J^+) \subseteq G^-_{\underline{v^J}_+,\underline{w},>0} (U^+_J \cap B^-_J \dot{v}_J B_J^-)\,^J B^+/\,^J B^+.
\]

Since $\pi_J^{-1}(g P_J^+ / P_J^+) \simeq L_J / B_J^-$ for any $g \in G$, the map
\[
G^-_{\underline{v^J}_+,\underline{w},>0} \times (U_J^+ \cap B_J^+ \dot{v}_J B_J^+) \to (\pi'_J)^{-1}(G^-_{\underline{v^J}_+,\underline{w},>0} P_J^+ / P_J^+), \quad (g,h) \mapsto g h\,^J B^+/\,^J B^+
\]
is an isomorphism. Note that $U_{v_J,>0}^+$ is an open subset of $U_J^+ \cap B_J^+ \dot{v}_J B_J^+$, and its image under this map is exactly $G^J_{(v,w),>0}\,^J B^+/\,^J B^+$. 

We conclude that $G^J_{(v,w),>0}\,^J B^+/\,^J B^+$ is open in $(\pi'_J)^{-1}(G^-_{\underline{v^J}_+,\underline{w},>0} P_J^+ / P_J^+)$, and hence open in $\oB^J_{v,w}(\mathbb{R})$.
\end{proof}
\end{lemma}

To proceed, we introduce an auxiliary map that will help us study the closeness properties. For $w\in W^J$ and $z\in W_J$, let
$$\phi_{w,z}:B^+\dot w \JB^+/\JB^+\to B^+\dot w\dot z B^+/B^+,\ \ b\dot w \JB^+/\JB^+\mapsto b\dot w\dot z B^+/B^+.$$

\begin{lemma}\label{lem:well-defined}
The map $\phi_{w,z}$ is well-defined.
\begin{proof}
Suppose $u_1, u_2 \in U^+$ satisfy $u_1 \dot{w}\,^J B^+/\,^J B^+ = u_2 \dot{w}\,^J B^+/\,^J B^+$. Then $u_1^{-1} u_2 \in U^+ \cap \dot{w}\,^J U^+ \dot{w}^{-1}$. For any $\alpha \in \Phi_J^-$, we have $\alpha \notin w^{-1}(\Phi^+)$ since $w \in W^J$. This implies:
\[
(U^+ \cap \dot{w}\,^J U^+ \dot{w}^{-1}) \dot{w} \dot{z} = \dot{w} (\dot{w}^{-1} U^+ \dot{w} \cap \,^J U^+) \dot{z} \subseteq \dot{w} U_{P_J^+} \dot{z} \subseteq \dot{w} \dot{z} B^+.
\]
Therefore, $u_1 \dot{w} \dot{z} B^+ / B^+ = u_2 \dot{w} \dot{z} B^+ / B^+$, proving that $\phi_{w,z}$ is well-defined.
\end{proof}
\end{lemma}

\begin{lemma}\label{twisted_reduction}
Suppose that $w \in W^J$ and $z \in W_J$. For any $g \in B^+ \dot{w} B^+$, we have
\[
\phi_{w,z}(g\,^J B^+/\,^J B^+) = g \dot{z} B^+ / B^+.
\]
\begin{proof}
Note that $B^+ \dot{w} B^+ = B^+ \dot{w} U_{P_J^+}$. Write $g = b \dot{w} u$ with $b \in B^+$ and $u \in U_{P_J^+}$. Then $g\,^J B^+/\,^J B^+ = b \dot{w}\,^J B^+/\,^J B^+$, and
\[
\phi_{w,z}(g\,^J B^+/\,^J B^+) = b \dot{w} \dot{z} B^+ / B^+ = b \dot{w} \dot{z} (\dot{z}^{-1} u \dot{z}) B^+ / B^+ = g \dot{z} B^+ / B^+,
\]
where the last equality follows since $\dot{z}^{-1} u \dot{z} \in U_{P_J^+} \subseteq B^+$.
\end{proof}
\end{lemma}

We now establish the closeness of our candidate set.

\begin{lemma} \label{closedness}
Suppose that $w\in W^J$. The set $G^J_{(v,w),>0}\JB^+/\JB^+$ is a closed subset in $\oB^J_{v,w}(\mathbb R)$.

\begin{proof}
Consider the projection $\pi_J \colon \mathcal{B}^J \to \mathcal{P}_J$ and its restriction
\[
\pi_J^{\prime\prime} \colon \oB^J_w = B^+ \dot{w}\,^J B^+/\,^J B^+ \to B^+ \dot{w} P_J^+ / P_J^+.
\]
Note that $\pi_J^{-1}(g P_J^+ / P_J^+) \simeq L_J / B_J^-$ for any $g \in G$. By the Schubert decomposition of $L_J / B_J^-$, we have
\[
(\pi_J^{\prime\prime})^{-1}(g P_J^+) = g U_J^+\,^J B^+/\,^J B^+.
\]

By the parametrization of positive cells \cite[Theorem 11.3]{MR04}, we have
\[
\mathcal{P}_{J,\geq 0} \cap B^+ \dot{w} P_J^+ / P_J^+ = \bigsqcup_{v \leq w} G^-_{\underline{v}_+,\underline{w},>0} P_J^+ / P_J^+,
\]
which is closed in $B^+ \dot{w} P_J^+ / P_J^+$. Thus,
\[
C:= (\pi_J^{\prime\prime})^{-1}(\mathcal{P}_{J,\geq 0} \cap B^+ \dot{w} P_J^+ / P_J^+) = \bigsqcup_{v\leq w}G^-_{\underline{v}_+,\underline{w},>0}U_J^+\JB^+/\JB^+
\]
is closed in $\oB^J_w(\mathbb{R})$.

For $w'\in W$, let $\CB^+_{w',\geq 0} := \mathcal{B}^+_{\geq 0} \cap B^+ \dot{w}' B^+ / B^+$. For $i\in J$, let $p_i:U_J^+\to U_{\alpha_i}$ be the natural projection. Denote $(U_J^+)^{p\geq0} =\bigcap_{i\in J}p_i^{-1}(U_{\alpha_i}\cap U^+_{\geq0})$. We claim that
\begin{equation}\tag = {*}
C \cap \bigcap_{i \in J} \phi_{w,s_i}^{-1}(\mathcal{B}^+_{w s_i,\geq 0}) = \bigsqcup_{v_1 \in W^J, v_1 \leq w} G^-_{\underline{v_1}_+,\underline{w},>0} (U_J^+)^{p\geq0}\,^J B^+/\,^J B^+.
\end{equation}

Take $g \in G^-_{\underline{v}_+,\underline{w},>0}$, with $v_J \neq e$ and $h\in U^+_J$. Choose $i \in J$ such that $v_J s_i < s_i$. By Lemma~\ref{twisted_reduction},
\[
\phi_{w,s_i}(gh\,^J B^+/\,^J B^+) \in G^-_{\underline{v}_+,\underline{w},>0} x_i(\mathbb R) \dot{s}_i B^+ / B^+.
\]
This set lies in a Deodhar component of $\oB_{v s_i, w s_i}^+$ that is not of top dimension (see \cite[Definition 4.1 \& Proposition 5.2]{MR04}). By \cite[Lemma 11.6]{MR04}, $\phi_{w,s_i}(g\,^J B^+/\,^J B^+) \notin \CB^+_{w,\geq 0}$. On the other hand, suppose that $g\in G_{\underline{v_1},+,\underline{w},>0}^-$ for $v_1\in W^J$. Then
$$\phi_{w,s_i}(gh\JB^+/\JB^+) \in G^-_{\underline{v_1}_+,\underline{w},>0}p_i(h)\dot s_iB^+/B^+,$$
and $\phi_{w,s_i}(gh\JB^+/\JB^+)\in\CB_{ws_i,\geq0}^+$ if and only if $p_i(h)\in U_{\alpha_i}\cap U^+_{\geq0}$. The claim (*) is proved.

Now, for $v\leq^J w$, let $$C_{v,w}:=\oB^J_{v,w}(\mathbb R)\cap C\cap\bigcap_{i\in J}\phi^{-1}_{w.s_i}(\CB_{ws_i,\geq0}^+).$$ 
Since $C\cap\bigcap_{i\in J}\phi_{w,s_i}^{-1}(\CB^+_{ws_i,\geq0})$ 
is a closed subset of $\oB^J_w(\mathbb R)$, we know that $C_{v,w}$ is a closed subset of $\oB^J_{v,w}(\mathbb R)$. Using the decomposition $L_J/B_J^- = \bigsqcup_{z\in W_J} B^-_J\dot zB_J^-/B_J^-$, we see that
\begin{align*}
C_{v,w} &= G^-_{\underline{v^J}_+,\underline{w},>0}\left((U_J^+)^{p\geq0}\cap B^-_J\dot v_JB_J^-\right)\JB^+/\JB^+ \\
&\simeq G^-_{\underline{v^J}_+,\underline{w},>0}\times \left((U_J^+)^{p\geq0}\cap B^-_J\dot v_JB_J^-\right).
\end{align*}
Since $U_{v_J,>0}^+$ is a closed subset in $U^+_J\cap B_J^-\dot v_J B_J^-$ and $U_{v_J,>0}^+\subseteq (U_J^+)^{p\geq0}$, it is also a closed subset in $(U_J^+)^{p\geq0}\cap B_J^-\dot vB_J^-$. We conclude that $G^-_{\underline{v^J}_+,\underline{w},>0}U^+_{v_J,>0}\JB^+/\JB^+$ is a closed subset in $\oB^J_{v,w}(\mathbb R)$.
\end{proof}

\end{lemma}

Since $G^J_{(v,w),>0}\,^J B^+/\,^J B^+ \simeq \mathbb{R}_{>0}^{\ell^J(w) - \ell^J(v)}$ is connected, the combination of Lemmas~\ref{lem:openness} and~\ref{closedness} establishes claim~\ref{special_connected_component}.

\subsection{parametrization in the case $w \in W^J$}
It follows from \cite[\S 2.12]{L94} that we can express the totally nonnegative monoid as:
\[
G_{\geq 0} = \bigsqcup_{w' \in W} U^-_{w',>0} U^+_{\geq 0} T_{>0} = \bigsqcup_{w \in W^J} U^-_{w,>0} U^+_{\geq 0} (U^-_{\geq 0} \cap U_J^-) T_{>0}.
\]

Set $U_{J,\geq 0}^+ := U_J^+ \cap U_{\geq 0}^+$. It is straightforward to verify that $U^+_{\geq 0}\,^J B^+/\,^J B^+ = U^+_{J,\geq 0}\,^J B^+/\,^J B^+$. We conclude that:
\[
\mathcal{B}^J_{\geq 0} = \overline{\bigsqcup_{w \in W^J} U^-_{w,>0} U^+_{J,\geq 0}\,^J B^+/\,^J B^+} = \overline{\bigsqcup_{\substack{w \in W^J \\ v \in W_J}} U^-_{w,>0} U^+_{v,>0}\,^J B^+/\,^J B^+}.
\]

We now establish the parametrization in the special case $w \in W^J$.

\begin{proposition} \label{three_sets_match}
Let $w \in W^J$ and $v \in W$ with $v \leq^J w$. The following three sets are equal:
\begin{enumerate}
    \item $\overline{U^-_{w,>0} U^+_{v',>0}\,^J B^+/\,^J B^+} \cap \oB^J_{v,w}(\mathbb{R})$ for any $v' \in W_J$ with $v' \leq^J v$;
    \item $\mathcal{B}^J_{v,w,>0}$;
    \item $G^J_{(v,w),>0}\,^J B^+/\,^J B^+$.
\end{enumerate}
In particular, the set $G^J_{(v,w),>0}\,^J B^+/\,^J B^+$ is independent of the choices of reduced expressions $\underline{w}$ and $\underline{v_J}$.

\begin{proof}
We prove the equalities through a cycle of inclusions.

\medskip\noindent
\textbf{(1) $=$ (2):} By definition,
\[
\mathcal{B}^J_{v,w,>0} = \overline{\bigcup_{\substack{w \in W^J \\ v' \in W_J}} U^-_{w,>0} U^+_{v',>0}\,^J B^+/\,^J B^+} \bigcap \oB^J_{v,w}(\mathbb{R}),
\]
and (1) is clearly contained in (2).

Now consider $v' \in W_J$ and $w' \in W^J$ such that $v' \leq^J v \leq^J w \leq^J w'$. By \Cref{key_assumption_1} and \Cref{preserve_connected_component}, the set $\overline{U^-_{w',>0} U^+_{v',>0}\,^J B^+/\,^J B^+} \cap \oB^J_{v',w}(\mathbb{R})$ is a connected component of $\oB^J_{v',w}(\mathbb{R})$, and it clearly contains $U^-_{w,>0} U^+_{v',>0}\,^J B^+/\,^J B^+$. Since the latter is also a connected component of $\oB^J_{v',w}(\mathbb{R})$ as in \S\ref{special_case}, we have:
\begin{align*}
\overline{U^-_{w',>0} U^+_{v',>0}\,^J B^+/\,^J B^+} \cap \oB^J_{v',w}(\mathbb{R}) 
&= \sigma^J_{w,+}(U^-_{w',>0} U^+_{v',>0}\,^J B^+/\,^J B^+) \\
&= U^-_{w,>0} U^+_{v',>0}\,^J B^+/\,^J B^+.
\end{align*}

By \Cref{key_assumption_1} and \Cref{preserve_connected_component} again, we have
\[
\dot{w}\,^J B^+/\,^J B^+ \in \overline{U^-_{w',>0} U^+_{v',>0}\,^J B^+/\,^J B^+}.
\]
Then by Lemma~\ref{alg_cut}, we obtain:
\[
\overline{U^-_{w',>0} U^+_{v',>0}\,^J B^+/\,^J B^+} \cap \oB^J_{v,w}(\mathbb{R}) = \overline{U^-_{w,>0} U^+_{v',>0}\,^J B^+/\,^J B^+} \cap \oB^J_{v,w}(\mathbb{R}).
\]
This shows that (1) $=$ (2), and in particular, it is a connected component of $\oB^J_{v,w}(\mathbb R)$.

\medskip\noindent
\textbf{(2) $=$ (3):} 
By \Cref{candidates_are_nonnegative}, we have (3) $\subseteq $ (2). By our proof of (1) $=$ (2), the set in (2) is a connected component of $\oB^J_{v,w}(\mathbb{R})$. Since (3) is also a connected component of $\oB^J_{v,w}(\mathbb{R})$, we conclude that (2) $=$ (3), completing the proof.
\end{proof}
\end{proposition}

\subsection{Inclusion property in the case $w\in W^J$} In this subsection, we show the inclusion property in the case $w\in W^J$. Recall the operation $\circ_l^J$ from \S\ref{tech}. 
\begin{lemma}\label{x_action}
Let $w\in W^J$, $v\in W$ such that $v\leq^J w$. For $i\in I$ and $a>0$, we have
$$x_i(a)G^J_{(v,w),>0}\JB^+/\JB^+\subseteq G^J_{(s_i\circ_l^Jv,w),>0}\JB^+/\JB^+.$$

\begin{proof}
Let $v' = s_i\circ_l^J v$. If $v' = v$, the conclusion follows from \Cref{x_action_part_1}. We assume that $v' = s_iv$.
By definition, we have $x_i(a)\CB^J_{\geq0}\subseteq\CB^J_{\geq0}$. By \cite[Corollary 1]{BD94}, we have $x_i(a)\oB^J_{v,w}(\mathbb R)\subseteq\oB^J_{v',w}(\mathbb R)$ Therefore,  $x_i(a)\CB^J_{v,w,>0}\subseteq\CB^J_{v',w,>0}$. The statement now follows from \Cref{three_sets_match}.
\end{proof}

\end{lemma}

\begin{proposition} \label{key_assumption_2}
Suppose that $w\in W^J$. Then for every $r\in W$ such that $v\leq^Jr\leq^J w$, we have 
$$G^J_{(v,w),>0}\JB^+/\JB^+\subseteq\dot r\JU^-\JB^+/\JB^+.$$

\begin{proof}
The proof is similar to \Cref{key_assumption_1}, except that there is one more case to deal with. We argue by induction on $\ell(w)$. For $\ell(w) = 0$, the statement follows from \cite[Lemma 5.1]{BH24}. 

Fix a reduced expression $\underline{w} = (s_{i_1},s_{i_2},\cdots,s_{i_n})$. Set $w' = s_{i_1}w$, $r' = s_{i_1}\circ_l^J r$, and $v' = s_{i_1}\circ_l^J v$. It follows from \Cref{twisted_downward} that $v'\leq^J r'\leq^J w'$. 

Write $g = g_1g_2\cdots g_nh\in G^-_{\underline{v^J}_+,\underline{w},>0}U^+_{v_J,>0} = G^J_{(v,w),>0}$. We consider three cases.
\begin{enumerate}[label = (\roman*)]
    \item $r'= r$. Then $v^J\leq r^J\leq s_{i_1}w$. Consequently, $v\leq^J r\leq^J w'$ and $g_1\in y_{i_1}(\mathbb R_{>0})$. Similar to (i) in the proof of \Cref{key_assumption_1}, we have  
    $$\dot r^{-1}G_{(v,w),>0}^J\subseteq\JU^-\dot r^{-1}G^J_{(v,w'),>0}.$$
    \item $r' = s_{i_1}r$ and $g_1 = y_{i_1}(a)$. Similar to (ii) in the proof of \Cref{key_assumption_1}, we have
    $$\dot r^{-1}g\JB^+/\JB^+\in\JU^-(\dot r')^{-1}x_{i_1}(a)G^J_{(v,w'),>0}\JB^+/\JB^+.$$
    By \Cref{x_action}, we have
    $x_{i_1}(a)G^J_{(v,w'),>0}\JB^+/\JB^+\subseteq G_{(v',w'),>0}^J\JB^+/\JB^+$, and thus $\dot r^{-1}g\JB^+/\JB^+ \in\JU^-(\dot r')^{-1}G_{(v',w'),>0}^J\JB^+/\JB^+$.
    \item $r' = s_{i_1}r$ and $g_1 = \dot s_{i_1}$. Then $s_{i_1}v^J<v^J$, and thus $v' = s_{i_1}v$. Therefore, $g' = g_1^{-1}g\in G^J_{(v',w'),>0}$, and we have
    $$\dot r^{-1}g = (\dot r')^{-1}g'\in(\dot r')^{-1}G^J_{(v',w'),>0}.$$
\end{enumerate}
The statement follows from induction assumption.
\end{proof}

\end{proposition}

\subsection{Connected component in general situation}
In the following subsections, we remove the assumption $w \in W^J$ and consider the general case. 

\begin{proposition} \label{positive_connected_component}
The subset $\mathcal B^J_{v,w,>0}$ is a connected component of $\oB^J_{v,w}(\mathbb R)$. Moreover, $\mathcal B^J_{v,w,>0} = \overline{\mathcal B^J_{v,w^J,>0}}\cap\oB^J_{v,w}(\mathbb R)$.

\begin{proof}
By \Cref{three_sets_match}, $\mathcal{B}^J_{v,w^J,>0}$ is a connected component of $\oB^J_{v,w^J}(\mathbb{R})$. \Cref{key_assumption_2} and \Cref{preserve_connected_component} imply that $\overline{\mathcal{B}^J_{v,w^J,>0}} \cap \oB^J_{v,w}(\mathbb{R})$
is a connected component of $\oB^J_{v,w}(\mathbb{R})$.

Now consider $v' \in W_J$ and $w' \in W^J$ such that $v' \leq^J v \leq^J w \leq^J w'$. Note that $w^J \leq^J w'$. By \Cref{key_assumption_2} and \Cref{preserve_connected_component}, we have:
\[
\dot{w}^J\,^J B^+/\,^J B^+ \in \overline{U^-_{w',>0} U^+_{v',>0}\,^J B^+/\,^J B^+}
\]
and
\[
\dot{v}\,^J B^+/\,^J B^+ \in \overline{U^-_{w^J,>0} U^+_{v',>0}\,^J B^+/\,^J B^+}.
\]

Applying Lemma~\ref{alg_cut}, we obtain:
\begin{align*}
\overline{U^-_{w',>0} U^+_{v',>0}\,^J B^+/\,^J B^+} \cap \oB^J_{v,w}(\mathbb{R}) 
&= \overline{U^-_{w^J,>0} U^+_{v',>0}\,^J B^+/\,^J B^+} \cap \oB^J_{v,w}(\mathbb{R}) \\
&= \overline{\mathcal{B}^J_{v,w^J,>0}} \cap \oB^J_{v,w}(\mathbb{R}).
\end{align*}

We conclude that $\mathcal{B}^J_{v,w,>0} = \overline{\mathcal{B}^J_{v,w^J,>0}} \cap \oB^J_{v,w}(\mathbb{R})$ and that it is a connected component of $\oB^J_{v,w}(\mathbb{R})$.
\end{proof}
\end{proposition}

\begin{corollary}\label{cor:one-dim-positive}
Suppose $\ell^J(w) - \ell^J(v) = 1$. Then
\[
\mathcal{B}^J_{v,w,>0} = G^J_{(v,w),>0}\,^J B^+/\,^J B^+.
\]

\begin{proof}
Since $\ell^J(w) - \ell^J(v) = 1$, \Cref{one_dim_KL} and \Cref{one_dim_KL_parametrization} imply:
\[
\oB^J_{v,w}(\mathbb{R}) = G^J_{(v,w),\neq 0}\,^J B^+/\,^J B^+ \simeq \mathbb{R}^\times.
\]
By \Cref{candidates_are_nonnegative}, we have $G^J_{(v,w),>0}\,^J B^+/\,^J B^+ \subseteq \mathcal{B}^J_{v,w,>0}$. Since $\mathcal{B}^J_{v,w,>0}$ is a connected component of $\oB^J_{v,w}(\mathbb{R})$, the equality follows.
\end{proof}
\end{corollary}


\begin{corollary} \label{positive_with_top_Deodhar}
Let $v,w\in W$ with $v\leq^J w$. Then
$$\mathcal B^J_{v,w,>0}\cap G^J_{(v,w),\neq0}\JB^+/\JB^+\subseteq G^J_{(v,w),>0}\JB^+/\JB^+.$$

\begin{proof}
Consider the partial flag variety $\mathcal{P}_J := G/P_J^+$ and its totally nonnegative part $\mathcal{P}_{J,\geq 0} = \overline{G_{\geq 0} P_J^+ / P_J^+}$. The natural projection $\pi \colon \mathcal{B}^J \to \mathcal{P}_J$ is proper, so $\pi(\mathcal{B}^J_{\geq 0}) = \mathcal{P}_{J,\geq 0}$.

We have the inclusion:
\[
\pi\left(\mathcal{B}^J_{v,w,>0} \cap G^J_{(v,w),\neq 0}\,^J B^+/\,^J B^+\right) \subseteq \mathcal{P}_{J,\geq 0} \cap G^-_{\underline{v^J c}_+,\underline{w^J},\neq 0} P_J^+ / P_J^+.
\]

By \cite[Proposition 5.5(1)]{BH24}, the right-hand side equals $G^-_{\underline{v^J c}_+,\underline{w^J},>0} P_J^+ / P_J^+$. Therefore,
\[
\mathcal{B}^J_{v,w,>0} \cap G^J_{(v,w),\neq 0}\,^J B^+/\,^J B^+ \subseteq G^-_{\underline{v^J c}_+,\underline{w^J},>0} \dot{c}^{-1} G^+_{\underline{c w_J}_+,\underline{v_J},\neq 0}\,^J B^+/\,^J B^+.
\]

Suppose, for contradiction, that there exist $g \in G^-_{\underline{v^J c}_+,\underline{w^J},>0}$ and $h \in G^+_{\underline{c w_J}_+,\underline{v_J},\neq 0} \setminus G^+_{\underline{c w_J}_+,\underline{v_J},>0}$ such that
\[
g \dot{c}^{-1} h\,^J B^+/\,^J B^+ \in \mathcal{B}^J_{v,w,>0}.
\]
Let $C$ be the connected component of $G^+_{\underline{c w_J}_+,\underline{v_J},\neq 0}$ containing $h$. Since $\mathcal{B}^J_{v,w,>0}$ is a connected component of $\oB^J_{v,w}(\mathbb{R})$, we must have:
\[
G^-_{\underline{v^J c}_+,\underline{w^J},>0} \dot{c}^{-1} C\,^J B^+/\,^J B^+ \subseteq \mathcal{B}^J_{v,w,>0} \subseteq \mathcal{B}^J_{\geq 0}.
\]

It follows that $\dot{v}^J C\,^J B^+/\,^J B^+ \subseteq \overline{G^-_{\underline{v^J c}_+,\underline{w^J},>0} \dot{c}^{-1} C\,^J B^+/\,^J B^+} \subseteq \mathcal{B}^J_{\geq 0}$.

Let $(a_1, a_2, \cdots, a_l)$ be the natural coordinate system on $C$. By assumption, there exists $1 \leq j \leq l$ such that $a_j < 0$. Taking the limits $a_i \to 0$ for $i < j$ and $a_i \to \infty$ for $i > j$, we find $v'_J, w'_J \in W_J$ with $w'_J \leq v'_J$ such that:
\[
\dot{v}^J G^+_{w'_J, v'_J, < 0}\,^J B^+/\,^J B^+ \subseteq \mathcal{B}^J_{\geq 0}.
\]
This contradicts Corollary~\ref{cor:one-dim-positive}, completing the proof.
\end{proof}

\end{corollary}

We require one more technical lemma for the proof of Theorem~\ref{thm:J-flag-main} (1).

\begin{lemma} \label{pass_to_partial_flag}
Suppose $w \in W^J$ with reduced expression $\underline{w} = (s_{i_1}, s_{i_2}, \cdots, s_{i_n})$, and let $w' = s_{i_1} w$. If
\[
\dot{s}_{i_1} G^-_{\underline{v}_+,\underline{w'},>0} P_J^+ / P_J^+ \cap \mathcal{P}_{J,\geq 0} \neq \emptyset,
\]
then $(s_{i_1}, \underline{v}_+)$ is a positive subexpression of $s_{i_1} v$ in $\underline{w}$.

\begin{proof}
We first show that $\dot{s}_{i_1} G^-_{\underline{v}_+,\underline{w'},>0} B^+ / B^+$ lies in a single open Richardson variety. Clearly, $\dot{s}_{i_1} G^-_{\underline{v}_+,\underline{w'},>0} B^+ / B^+ \subseteq \oB^+_w(\mathbb{R})$. It remains to show this set lies in a single opposite Bruhat cell.

Let $\underline{v}_+ = (t_2, t_3, \cdots, t_n)$ and define $v_{(j)} = t_2 t_3 \cdots t_j$. For each $\alpha \in \Phi$, fix an isomorphism $u_\alpha \colon \mathbb{R} \xrightarrow{\sim} U_\alpha$, with $u_{\alpha_i} = x_i$ and $u_{-\alpha_i} = y_i$ for every simple root $\alpha_i$. 

An element of $\dot{s}_{i_1} G^-_{v,w',>0} B^+ / B^+$ can be written as:
\[
h B^+ / B^+ = \dot{s}_{i_1} \prod_{\substack{1 \leq j \leq n-1 \\ v_{(j)} = v_{(j+1)}}} u_{-v_j(\alpha_{i_{j+1}})}(b_j) \dot{v} B^+ / B^+.
\]
Since $\underline{v}_+$ is positive, we have $v_{(j)}(\alpha_{i_{j+1}}) > 0$ for all $2 \leq j \leq n$. We have the following three cases:

\begin{enumerate}[label = (\roman*)]
\item $v_{(j)}(\alpha_{i_{j+1}}) \neq \alpha_{i_1}$ for all $1 \leq j \leq n-1$. Then
    \[
    \dot{s}_{i_1} \prod u_{-v_{(j)}(\alpha_{i_{j+1}})}(b_j) \dot{v} B^+ / B^+ \in U^- \dot{s}_{i_1} \dot{v} B^+ / B^+.
    \]
    This condition is equivalent to $(s_{i_1}, \underline{v}_+)$ being a positive subexpression in $\underline{w}$.

\item $v_{(j)}(\alpha_{i_{j+1}}) = \alpha_{i_1}$ for some $j$, and $s_{i_1} v > v$. In this case, we must have $v_j = v_{j+1}$ (otherwise $v_{j+1} = v_j s_{i_{j+1}}$ would imply $s_{i_1} v_{j+1} < v_{j+1}$ and hence $s_{i_1} v < v$). Moreover, $b_j > 0$ by Lemma~\ref{conj_positive}. Then:
    \begin{align*}
        h B^+ / B^+ &\in U^- \dot{s}_{i_1} y_{i_1}(b) \dot{v} B^+ / B^+ \\
        &= U^- y_{i_1}(-b) x_{i_1}(b^{-1}) \dot{v} B^+ / B^+ \\
        &\subseteq U^- \dot{v} B^+ / B^+
    \end{align*}
for some $b>0$.

\item $v_{(j)}(\alpha_{i_{j+1}}) = \alpha_{i_1}$ for some $j$, and $s_{i_1} v < v$. Then
    \[
    h B^+ / B^+ \in \dot{s}_{i_1} U^- \dot{v} B^+ / B^+ \subseteq U^- \dot{s}_{i_1} \dot{v} B^+ / B^+.
    \]
\end{enumerate}

In each case, there exists $v' \leq w$ such that:
\[
\dot{s}_{i_1} G^-_{\underline{v}_+,\underline{w'},>0} P_J^+ / P_J^+ \subseteq \pi(\oB^+_{v',w}(\mathbb{R})).
\]

By assumption, $\dot{s}_{i_1} G^-_{\underline{v}_+,\underline{w'},>0} P_J^+ / P_J^+ \cap \mathcal{P}_{J,\geq 0} \neq \emptyset$. Since the left-hand side is connected and $\pi(\oB^+_{v',w}(\mathbb{R})) \cap \mathcal{P}_{J,\geq 0}$ is a connected component of $\pi(\oB^+_{v',w}(\mathbb{R}))$ (see \cite[\S 5.2]{BH24}), we have:
\[
\dot{s}_{i_1} G^-_{\underline{v}_+,\underline{w'},>0} P_J^+ / P_J^+ \subseteq \pi(\oB^+_{v',w}(\mathbb{R})) \cap \mathcal{P}_{J,\geq 0} \subseteq \mathcal{P}_{J,\geq 0}.
\]

We now show that only case (i) can occur. Suppose instead we are in case (ii) or (iii). Let $j$ be the smallest index such that $v_{(j)}(\alpha_{i_{j+1}}) = \alpha_{i_1}$. Then:
\[
(s_{i_1}, t_2, \cdots, t_{j-1}, s_{i_j}, s_{i_{j+1}}, \cdots, s_{i_n})
\]
is a distinguished but non-positive subexpression in $\underline{w}$ (see \cite[Definition 3.3]{MR04}). Let $D$ be the corresponding Deodhar component (see \cite[Definition 4.1]{MR04}). Then:
\[
\overline{\dot{s}_{i_1} G^-_{\underline{v}_+,\underline{w'},>0} P_J^+ / P_J^+} \cap D \neq \emptyset.
\]
But we have shown $\dot{s}_{i_1} G^-_{\underline{v}_+,\underline{w'},>0} P_J^+ / P_J^+ \subseteq \mathcal{P}_{J,\geq 0}$, and since $\mathcal{P}_{J,\geq 0}$ is closed, we get $D \cap \mathcal{P}_{J,\geq 0} \neq \emptyset$, contradicting \cite[Lemma 11.5]{MR04}.
\end{proof}

\end{lemma}

\subsection{Proof of Theorem \ref{thm:J-flag-main} (1)} \label{proof_parametrization}
We first note that $\mathcal B^J_{v,w,>0}$ is a connected component of $\oB^J_{v,w}(\mathbb R)$ by \Cref{positive_connected_component}, and $G^J_{(v,w),>0}\JB^+/\JB^+$ is a connected subset of $\oB^J_{v,w}(\mathbb R)$. To establish their equality, it suffices to prove the inclusion $\mathcal B^J_{v,w,>0}\subseteq G^J_{(v,w),>0}\JB^+/\JB^+$. 

By \Cref{three_sets_match} and \Cref{positive_connected_component}, we have:
\[
\mathcal{B}^J_{v,w,>0} \subseteq \overline{\mathcal{B}^J_{v,w^J,>0}} \subseteq \overline{\mathcal{B}^J_{v_J,w^J,>0}}.
\]
Thus it suffices to show that:
\[
\overline{U^-_{w^J,>0} U^+_{J,\geq 0}\,^J B^+/\,^J B^+} \cap B^+ \dot{w}\,^J B^+/\,^J B^+ \subseteq \bigsqcup_{v \leq^J w} G^J_{(v,w),>0}\,^J B^+/\,^J B^+.
\]

We proceed by induction on $\ell(w^J)$. 

\noindent
\textbf{Base case:} When $\ell(w^J) = 0$, we have the identification $L_J\,^J B^+/\,^J B^+ \simeq L_J / B_J^-$, and the claim follows from the corresponding result for ordinary totally nonnegative flag varieties \cite[\S 5.1]{BH24}.

\noindent
\textbf{Inductive step:} Suppose $\underline{w^J} = (s_{i_1}, s_{i_2}, \cdots, s_{i_n})$. Consider a sequence:
\[
g_k\,^J B^+ = y_{i_1}(a_{1,k}) y_{i_2}(a_{2,k}) \cdots y_{i_n}(a_{n,k}) u_k\,^J B^+
\]
$U^-_{w^J,>0} U^+_{J,\geq 0} \,^J B^+/\JB^+$. Passing to a subsequence if necessary, we may assume $\lim_{k \to \infty} a_{1,k}$ exists in $\mathbb{R}_{\geq 0} \cup \{\infty\}$.

If $\lim_{k \to \infty} a_{1,k} \in \mathbb{R}_{\geq 0}$, then:
    \[
    \lim_{k \to \infty} g_k B^+ \in y_{i_1}\left(\lim_{k \to \infty} a_{1,k}\right) \overline{U^-_{s_{i_1} w^J,>0} U^+_{J,\geq 0}\,^J B^+/\,^J B^+}.
    \]
    Since $s_{i_1} w^J < w^J$ and $s_{i_1} w^J \in W^J$, the claim follows by induction.
    
If $\lim_{k \to \infty} a_{1,k} = \infty$, using the identity:
    \[
    y_i(a) = x_i(a^{-1}) \dot{s}_i \alpha^\vee_i(a) x_i(a^{-1}),
    \]
    we obtain:
    \[
    g_k\,^J B^+/\,^J B^+ \in x_{i_1}(a_{1,k}^{-1}) \dot{s}_{i_1} U^-_{s_{i_1} w^J,>0} U^+_{J,\geq 0}\,^J B^+/\,^J B^+.
    \]
    By the induction hypothesis, we have:
    \begin{align*}
        g\,^J B^+/\,^J B^+ &:= \lim_{k \to \infty} g_k\,^J B^+/\,^J B^+ \in \dot{s}_{i_1} G^J_{(v',w'),>0}\,^J B^+/\,^J B^+ \\
        &= \dot{s}_{i_1} G^-_{\underline{(v')^J c'}_+,\underline{(w')^J},>0} G^+_{\underline{w'_J}_+,\underline{c^{\prime,-1} v'_J},>0}\,^J B^+/\,^J B^+,
    \end{align*}
    for some $v' \leq^J w'$. Since we only consider limits inside $B^+ \dot{w}\,^J B^+/\,^J B^+$, we have $w' = s_{i_1} w$.

Recall that the natural projection $\pi:G/\JB^+\to G/P_J^+$ preserves the totally nonnegative parts. By \Cref{pass_to_partial_flag},  $(s_{i_1},\underline{(v')^Jc'}_+)$ is a positive subexpression in $\underline{w^J}$. In particular, $s_{i_1}v^Jc>v^Jc$.

If $s_{i_1} v^J \in W^J$, we take $v = s_{i_1} v'$. Then clearly:
    \[
    g\,^J B^+/\,^J B^+ \in G^J_{(v,w),>0}\,^J B^+/\,^J B^+.
    \]

Otherwise, $s_{i_1} v^J = v^J s_j$ and $s_j c' > c'$ for some $j\in J$. Note that:
    \[
    G^+_{\underline{w_J}_+,\underline{c^{\prime,-1} v'_J},>0}\,^J B^+/\,^J B^+ = \dot{c}^{\prime,-1} G^+_{\underline{c' w_J}_+,\underline{v'_J},>0}\,^J B^+/\,^J B^+,
    \]
    and by \cite[Theorem 11.3]{MR04}, this set is independent of the choice of reduced expression.

Let $\underline{v'_J} = (s_{i_1}', s_{i_2}', \cdots, s_{i_m}')$ and $\underline{c' w_J}_+ = (t_1', t_2', \cdots, t_m')$. We have the following three cases:

\begin{enumerate}[label=(\roman*)]
        \item $s_{i_1}' = t_1' = s_j$. Let $v = (v')^J s_j v'_J$. Then:
        \[
        \dot{s}_{i_1} G^-_{\underline{(v')^J c'}_+,\underline{(w')^J},>0} G^+_{\underline{w'_J}_+,\underline{c^{\prime,-1} v'_J},>0}\,^J B^+/\,^J B^+ = G^J_{(v,w),>0}\,^J B^+/\,^J B^+.
        \]
        
        \item $s_{i_1}' = s_j$ and $t_1' = 1$.
        
        \item $s_{i_1}' \neq s_j$, in which case we may assume $s_j v'_J > v'_J$.
\end{enumerate}

We now show that cases (ii) and (iii) lead to contradictions. Suppose we are in case (ii) or (iii). Let $v_J = \min\{s_j v'_J, v'_J\}$. Then there exists $b \geq 0$ such that:
\[
g\,^J B^+/\,^J B^+ \in \dot{s}_{i_1} G^-_{\underline{(v')^J c'}_+,\underline{(w')^J},>0} \dot{c}^{\prime,-1} x_j(b) G^+_{\underline{c' w_J}_+,\underline{v_J},>0}\,^J B^+/\,^J B^+.
\]

For $t > 0$, we have $x_{i_1}(t) g\,^J B^+/\,^J B^+ \in \mathcal{B}^J_{\geq 0}$. Using the identity:
\[
x_i(t) \dot{s}_i = y_i(t^{-1}) x_i(-t) \alpha^\vee_i(t)
\]
and \Cref{springer_fiber_KL} (2), we obtain:
\begin{align*}
x_{i_1}(t)g\,^J B^+/\,^J B^+ &\in y_{i_1}(t^{-1})x_{i_1}(-t)\alpha^\vee_{i_1}(t)G^-_{\underline{(v')^J c'}_+,\underline{(w')^J},>0} \dot{c}^{\prime,-1} x_j(b) G^+_{\underline{c' w_J}_+,\underline{v_J},>0}\,^J B^+/\,^J B^+.
\\&= y_{i_1}(t^{-1})x_{i_1}(-t)G^-_{\underline{(v')^J c'}_+,\underline{(w')^J},>0} \dot{c}^{\prime,-1} x_j(bk_1^2t^2) G^+_{\underline{c' w_J}_+,\underline{v_J},>0}\,^J B^+/\,^J B^+
\\&=y_{i_1}(t^{-1})G^-_{\underline{(v')^J c'}_+,\underline{(w')^J},>0} \dot{c}^{\prime,-1} x_j(bk_1^2t^2-k_2t) G^+_{\underline{c' w_J}_+,\underline{v_J},>0}\,^J B^+/\,^J B^+
\\&\subseteq G^J_{((v')^Js_jv_J,w),\neq0}\JB^+/\JB^+.
\end{align*}
for some constants $k_1\neq0$ and $k_2>0$.

Taking $t > 0$ sufficiently small so that $b k_1^2 t^2 - k_2 t < 0$, we have $$x_{i_1}(t)g\JB^+/\JB^+\in\CB^J_{\geq0}\bigcap \left(G^J_{((v')^Js_jv_J,w),\neq0}\setminus G^J_{((v')^Js_jv_J,w),>0}\right)\JB^+/\JB^+.$$
However, this contradicts Corollary~\ref{positive_with_top_Deodhar}. The proof of Theorem~\ref{thm:J-flag-main} (1) is complete.

\subsection{Inclusion property in general}
With the parametrization result at hand, we are able to show the inclusion property in its full generality.

\begin{proposition} \label{key_assumption_3}
Let $v,w\in W$ such that $v\leq^J w$. For every $r\in W$ such that $v\leq^Jr\leq^J w$, we have 
$$G^J_{(v,w),>0}\JB^+/\JB^+\subseteq\dot r\JU^-\JB^+/\JB^+.$$

\begin{proof}
The proof is almost identical to \Cref{key_assumption_2}. First note that, we have proved that $\CB^J_{v,w,>0} = G^J_{(v,w),>0}\JB^+/\JB^+$. Similar to \Cref{x_action}, we have
$$x_i(a)G^J_{(v,w),>0}\JB^+/\JB^+\subseteq G^J_{(s_i\circ_l^Jv,w),>0}\JB^+/\JB^+.$$

Now we turn to the proof. We argue by induction on $\ell(w)$. For $\ell(w) = 0$, the statement follows from \cite[Lemma 5.1]{BH24}. 

Let $c = \min\{c'\in W_J|w_J\leq c^{-1}v_J\}$. Then $v^Jc\leq w^J$ since $v\leq^J w$. Fix a reduced expression $\underline{w^J} = (s_{i_1},s_{i_2},\cdots,s_{i_n})$. 
Set $w' = s_{i_1}w$, $r' = s_{i_1}\circ_l^J r$, and $v' = s_{i_1}\circ_l^J v$. It follows from \Cref{twisted_downward} that $v'\leq^J r'\leq^J w'$. 

Write $g = g_1g_2\cdots g_nh\in G^J_{(v,w),>0}$. We consider three cases.
\begin{enumerate}[label = (\roman*)]
    \item $r'= r$. Since $v\leq^J r= r'\leq^J w'$, we have $v^Jc\leq (w')^J = s_{i_1}w^J$. Consequently, $g_1\in y_{i_1}(\mathbb R_{>0})$. Similar to (i) in the proof of \Cref{key_assumption_1}, we have  
    $\dot r^{-1}G_{(v,w),>0}^J\subseteq\JU^-\dot r^{-1}G^J_{(v,w'),>0}$.
    \item $r' = s_{i_1}r$ and $g_1 = y_{i_1}(a)$. Similar to (ii) in the proof of \Cref{key_assumption_1}, we have
    $$\dot r^{-1}g\JB^+/\JB^+\in\JU^-(\dot r')^{-1}G_{(v',w'),>0}^J\JB^+/\JB^+.$$
    \item $r' = s_{i_1}r$ and $g_1 = \dot s_{i_1}$. Then $s_{i_1}v^Jc<v^Jc$. We have either $s_{i_1}v^J<v^J$ or $s_{i_1}v^J = v^Js_j$ for some $j\in J$ and $s_jc<c$. In the first case, it is clear that $v' = v$. In the second case, since $\ell(c^{-1}v_J) = \ell(c^{-1})+\ell(v_J)$ and $s_jc<c$, we have $s_jv_J>v_J$, and thus $v' = v$ as well. Therefore, $g' = g_1^{-1}g\in G^J_{(v',w'),>0}$, and we have $\dot r^{-1}g = (\dot r')^{-1}g'\in(\dot r')^{-1}G^J_{(v',w'),>0}$.
\end{enumerate}
The statement follows from induction assumption.
\end{proof}

\end{proposition}

\subsection{Proof of Theorem \ref{thm:J-flag-main} (2) \& (3)} \label{proof_J_main_prop}
Fix $v,w\in W$ such that $v\leq^Jw$. Then $v\leq^J w^J$. We first consider $\overline{\CB^J_{v,w^J,>0}}$.

By \Cref{positive_connected_component}, we have $\overline{\CB^J_{v,w^J,>0}} = \bigsqcup_{v\leq^Jv'\leq^Jw'\leq^J w^J}\CB^J_{v',w',>0}$. Moreover, each $\CB^J_{v',w',>0}$ is a connected component of $\oB^J_{v',w'}(\mathbb R)$. Together with \Cref{thm:J-flag-main}(1) and \Cref{key_assumption_3}, we see that
\[
\CB^J_{v',w',>0} \subseteq \dot{r}\,^J U^-\,^J B^+/\,^J B^+ \quad \text{for all } v' \leq^J r \leq^J w'.
\]
Applying \Cref{thm:J_product} (2) to $\overline{\CB^J_{v,w^J,>0}}$, we have that
$$\overline{\CB^J_{v,w,>0}} = \bigsqcup_{v\leq^Jv'\leq^J w'\leq^J w}\CB^J_{v',w',>0}.$$

Note that, by definition, $\CB^J_{\geq0}$ is a remarkable polyhedral space with regularity property if each $\overline{\CB^J_{v,w,>0}}$ is. We apply \Cref{thm:J_product} again to $\overline{\CB^J_{v,w,>0}}$, and
deduce \Cref{thm:J-flag-main} (2) and (3).

\section{Applications}

\subsection{Double flag variety as twisted product}
The main objects here are the double flag variety and their stratification with respect to $G_{\text{diag}}$-orbits. For the convenience of applying the results from \S\ref{TNN_twsited}, we first introduce another formulation of the double flag varieties.

We consider the action of $B^+$ on $G \times G/B^\pm$ by $b \cdot (g, g' B^\pm)=(g b \i, b g' B^\pm)$. Let
$$\CZ^{+,+} = G\times^{B^+}G/B^+,\ \ \CZ^{+,-} = G\times^{B^+}G/B^-$$ be the quotient spaces of $G \times G/B^\pm$ by the action of $B^+$. 

Define the convolution product $$m: G \times^{B^+} G/B^\pm \to G/B^\pm, \quad (g, g' B^\pm) \mapsto g g' B^\pm.$$

Let $w, u, v \in W$. We define 
$$\oZ^{+, \pm}_{w,v} = (B^+\dot wB^+,B^+\dot vB^\pm/B^\pm), \qquad \oZ^{+, \pm; \, u} = m^{-1}(B^-\dot uB^\pm/B^\pm).$$


Let $\oZ_{w,v}^{+, \pm; \, u} = \oZ^{+, \pm}_{w,v}\cap \oZ^{+, \pm; \, u}$. We have $$\mathcal Z^{+, \pm} = \bigsqcup\oZ_{w,v}^{+, \pm; \, u}$$ as stratified spaces. 

There exists a natural isomorphism of varieties $$\CZ^{+, \pm} \to G/B^+ \times G/B^\pm, \qquad (g, g' B^\pm) \mapsto (g B^+/B^+, g g' B^\pm/B^\pm).$$ We have the stratification $G/B^+=\bigsqcup \, \oB^+_{x, y}$, where $\oB^+_{x,y} = B^+\dot xB^+/B^+\cap B^-\dot yB^+/B^+$ is the corresponding open Richardson variety in $G/B^+$. Similarly, we have the stratification $G/B^-=\bigsqcup \, \oB^-_{x', y'}$, where $\oB^-_{x', y'}=B^- \dot x' B^-/B^- \cap B^+ \dot y' B^-/B^-$ is the corresponding open Richardson variety of $G/B^-$. The variety $G/B^+ \times G/B^\pm$ has a natural stratification $\bigsqcup \oB_{x, y}^+ \times \oB^\pm_{x',y'}$. The isomorphism $\CZ^{+, \pm} \to G/B^+ \times G/B^\pm$ above is not compatible with these stratifications. Instead, the isomorphism takes 
$\oZ^{+, \pm; \, u}_{w,v}$ to $(B^+\dot wB^+/B^+,B^-\dot uB^\pm/B^\pm)\bigcap G_{\text{diag}}\cdot(B^+/B^+,\dot vB^\pm/B^\pm)$. Thus, they serve as the models for the double flag varieties $G/B^+\times G/B^\pm$.

The stratified space $\CZ^{+, +}$ was studied in \cite{BH24}. In this section, we mainly focus on $\CZ^{+, -}$. We will simply write $\CZ$ for $\CZ^{+, -}$, and $\oZ^u_{w, v}$ for $\oZ^{+, -; \, u}_{w, v}$, etc. 

By \cite[Lemma 1.4]{He07}, the set $\{w'v|w'\leq w\}$ contains a unique minimal element (with respect to Bruhat order), we denote this minimal element by $w\circ_l v$ and call it the \textit{left downward Demazure product} of $w$ and $v$. We have the following lemma describing the nonempty pattern of $\oZ^u_{w,v}$.

\begin{lemma}
The subvariety $\oZ^u_{w,v}$ is nonempty if and only if $w\circ_l v\leq u$.

\begin{proof}
Let $w' = w\circ_l v$. Then by \cite[Corollary 1]{BD94}, we have 
$$B^+\dot w'B^-/B^-\subseteq m(\oZ_{v,w}^u)\subseteq\bigsqcup_{w'\leq w''}B^+\dot w''B^-/B^-.$$
We have the equivalence of the following three conditions:
\begin{align*}
w\circ_l v\leq u&\Longleftrightarrow B^-\dot uB^-/B^-\cap B^+\dot w'B^-/B^-\neq\emptyset \\    
&\Longleftrightarrow B^-\dot uB^-/B^-\cap \left(\sqcup_{w'\leq w''}B^+\dot w''B^-/B^-\right)\neq\emptyset,
\end{align*}
and the statement follows.
\end{proof}

\end{lemma}

\subsection{Total positivity in $\CZ^{+,+}$}
In a joint work of the first author with Bao \cite{BH22}, the ``thickening" method was introduced to study the total positivity in $\CZ^{+,+}$. They constructed a larger group $\tilde G$ (see also \S\ref{thickening}) and a locally closed subvariety $\tilde\CZ^{+,+}$ of $\tilde G/\tilde B^+$ that is a fiber bundle over $\CZ^{+,+}$, and the fiber bundle map is compatible with total positivity. By passing to the totally nonnegative flag variety $(\tilde G/\tilde B^+)_{\geq0}$, they deduced the following theorem \cite[Theorem 1]{BH22}. 
\begin{theorem}
The totally nonnegative double flag variety $\CZ^{+,+}_{\geq0}$ is a remarkable polyhedral space with regularity property.
\end{theorem}

In our study of the totally nonnegative part of $\mathcal Z^{+,-}$, we use a similar method, except that the fiber bundle is a subvariety of some twisted flag variety $\tilde\CB^{\tilde J}$.

\subsection{The thickening map} \label{thickening}
Consider a group $G$ and its double flag variety $\CZ$. The goal of this subsection is to construct a Kac-Moody group $\tilde G$ and a subvariety $\tilde\CZ$ of some twisted flag variety $\tilde\CB^{\tilde J}$ of $\tilde G$, such that $\tilde\CZ$ is a trivial fiber bundle over $\mathcal Z$ with fiber $\mathbb C^\times$.

\subsubsection{The thickening group} The thickening group was constructed in \cite{BH22}. We recall the construction here.

Let $I$ be the index set of simple roots of $G$. Let $\tilde I = I\sqcup\{\infty\}$. The
Dynkin diagram of $\tilde I$ is obtained from the Dynkin diagram of $I$ by adding an edge $i\Leftrightarrow\infty$
for every $i\in I$. In other words, the generalized Cartan matrix
$\tilde A = (\tilde a_{i,j})_{i,j\in\tilde I}$ is defined as follows:
\begin{itemize}
    \item $\tilde a_{i,j} = a_{i,j}$ for $i,j\in I$;
    \item $\tilde a_{i,\infty} = \tilde a_{\infty,i} = -2$ for $i\in I$;
    \item $\tilde a_{\infty,\infty} = 2$.
\end{itemize}

Let $\tilde G$ be the minimal Kac-Moody group of simply connected type associated to $(\tilde I, \tilde A)$
and $\tilde W$ be its Weyl group. Let $\tilde W_I$ be the parabolic subgroup of $\tilde W$ generated by simple reflections whose index comes from $I$. Let  $\tilde L_I$ be the standard Levi subgroup of the parabolic subgroup $\tilde P_I^+$ of $\tilde G$ that corresponds to $I$. We abuse notations and denote the natural identifications
$$i: W\to\tilde W_I\text{ and }i: G\to \tilde L_I.$$
We fix a pinning $(\tilde T,\tilde B^+,\tilde B^-,x_i,y_i;i\in\tilde I)$ of $\tilde G$ that is compatible with the pinning of $G$ via the identification $i$. We define the thickening maps
$$th : W^2\to \tilde W,\ \  (w,v)\to ws_{\infty}v,$$
$$th:\mathbb C^\times\times G^2\to\tilde G,\ \ (a;g_1,g_2)
\to g_1y_\infty(a)g_2.$$

Note that, $w_1\leq w_2$ and $v_1\geq v_2$ if and only if $th(w_1,v_1)\leq^I th(w_2,v_2)$.

\subsubsection{A fiber bundle}
We regard $I$ as a subset of $\tilde I$, and consider the $I$-twisted flag variety $\tilde{\mathcal B}^{I}$.

\begin{proposition} \label{fiber_bundle}
Set
$$\tilde\CZ := \tilde P_I^-\IB^+/\IB^+\bigcap\tilde P^+_I(\tilde B^+\dot s_\infty\tilde B^+)\tilde P_I^+\IB^+/\IB^+\subseteq\tilde{\mathcal B}^I.$$
Then
\begin{enumerate}
    \item $\tilde\CZ = \bigsqcup_{i(u)\leq^I th(w,v)}\mathring{\tilde{\CB}}^I_{i(u),th(w,v)}$ is a locally closed subvariety of $\tilde{\mathcal B}^I$;
    \item there exist an isomorphism $s:\mathbb C^\times\times\CZ\xrightarrow[]{\sim}\tilde\CZ$ and a projection $f:\tilde\CZ\to\CZ$ realizing $\tilde\CZ$ as a trivial fiber bundle over $\CZ$.
    \item the isomorphism $s$ sends $\mathbb C^\times\times\oZ_{w,v}^u$ to $\mathring{\tilde{\CB}}^I_{i(u),th(w,v)}$, and the projection $f$ sends $\mathring{\tilde{\CB}}^I_{i(u),th(w,v)}$ to $\oZ_{w,v}^u$.
\end{enumerate}

\begin{proof}
It follows from \cite[Corollary 1]{BD94} that
$$\tilde P_I^+\dot s_\infty\tilde P_I^+ = \left(\bigsqcup_{w\in W}\tilde B^+\dot w\dot s_\infty\tilde B^+\right)\left(\bigsqcup_{v\in W}\tilde B^+\dot v\IB^+\right) = \bigsqcup_{w,v\in W}\tilde B^+\dot w\dot s_\infty\dot v\IB^+.$$
Since $\tilde P_I^-\IB^+ = \bigsqcup_{u\in W}\tilde B^-\dot u\IB^+$, we have the decomposition of $\tilde Z$ as in part (1).

Let $\CZ^\prime = \tilde P^+_I\tilde P^+_\infty\tilde P^+_I\IB^+/\IB^+$. Then $\CZ^\prime$ is closed in $\tilde\CB^I$. Moreover, the subset $\tilde P^+_I(\tilde B^+\dot s_\infty\tilde B^+)\tilde P_I^+\IB^+/\IB^+$ is open in $\CZ^\prime$. This shows that $\CZ$ is a locally closed subvariety of $\CB^I$ and completes the proof of part (1).

An element in $\tilde B^+\dot s_\infty\tilde B^+/\tilde B^+$ is either of the form $y_\infty(a)\tilde B^+/\tilde B^+$ for some $a\neq0$ or of the form $\dot s_\infty\tilde B^+/\tilde B^+$. Thus every element in 
$$\tilde P^+_I\dot s_\infty\tilde P_I^+\IB^+/\IB^+ = G(\tilde B^+\dot s_\infty\tilde B^+)\tilde P_I^+/\IB^+$$ can be written as $g_1yg_2\IB^+/\IB^+$, where $g_1,g_2\in G$ and $y=y_\infty(a)$ for some $a\neq0$ or $y = \dot s_\infty$. 

Suppose that $g_1\dot s_\infty g_2\in\tilde P_I^-\IB^+$. Then 
$$\dot s_\infty g_2\in g_1^{-1}\tilde P_I^-\IB^+\cap \dot s_\infty G\subseteq \tilde P_I^-\IB^+\bigcap\left(\bigsqcup_{v\in W}\tilde B^-\dot s_\infty\dot v\IB^+\right) = \emptyset,$$
a contradiction. Thus, $g_1\dot s_\infty g_2\notin\tilde P_I^-\IB^+$ for any $g_1,g_2\in G$. We conclude that every element in $\tilde\CZ$ is of the form $g_1y_\infty(a)g_2\IB^+/\IB^+$ for $g_1,g_2\in G$ and $a\neq0$. 

Suppose that $(g_1,g_2B^-/B^-)\in\oZ^u_{w,v}$. Then $g_1\in B^+\dot wB^+$, $g_2\in B^+\dot vB^-$ and $g_1g_2\in B^-\dot uB^-$. Note that $g\tilde U_{\tilde P_I^-}g^{-1}\subseteq\tilde U^-$ for $g\in G$. Then 
$$g_1y_\infty(a)g_2\IB^+/\IB^+ = g_1y_\infty(a)g_1^{-1}(g_1g_2)\IB^+/\IB^+\in \tilde B^-\dot u\IB^+/\IB^+.$$

Let $b\in B^+$ such that $bg_2\in \dot vB^-$. Note that $b$ commutes with $y_\infty(a)$, we have
$g_1y_\infty(a)g_2\IB^+/\IB^+ = (g_1b)y_\infty(a)\dot v\IB^+/\IB^+$. Using the identity
$$y_\infty(a) = x_\infty(a^{-1})\dot s_\infty x_\infty(-a)\alpha^\vee_\infty(a),$$
we have
\begin{align*}
g_1y_\infty(a)g_2\IB^+/\IB^+ &= (g_1b)x_\infty(a^{-1})(g_1b)^{-1}(g_1b)\dot s_\infty\dot v\IB^+/\IB^+ \\
&\in\tilde B^+\dot wU^+\dot s_\infty\dot v\IB^+/\IB^+.
\end{align*}
Finally, since $\dot s_\infty^{-1} U^+\dot s_\infty\subseteq\tilde U_{\tilde P_I^+}$, we have
$g_1y_\infty(a)g_2\IB^+/\IB^+\in\tilde B^+\dot w\dot s_\infty\dot v\IB^+/\IB^+$. In particular, $g_1y_\infty(a)g_2\IB^+/\IB+ = g_1'y_\infty(a')g_2'\IB^+/\IB^+$ if and only if $a = a'$ and $(g_1,g_2B^-) = (g_1',g_2'B^-)$ in $\CZ$.

The thickening map
$$th:\mathbb C^\times\times G^2\to\tilde\CZ:(a;g_1,g_2)\to g_1y_\infty(a)g_2\IB^+/\IB^+$$
factors through $\mathbb C^\times\times\CZ$ and defines a stratified isomorphism $s:\mathbb C^\times\times\CZ\xrightarrow[]{\sim}\tilde\CZ$, and the projection $f$ given by 
$$f:\tilde\CZ\to\CZ: g_1y_\infty(a)g_2\IB^+/\IB^+ \to (g_1,g_2B^-/B^-)$$
is also stratified. This proves (2) and (3).
\end{proof}

\end{proposition}

The isomorphism $s:\mathbb C^\times\times\CZ\simeq\tilde\CZ$ and the projection $f:\tilde\CZ\to\CZ$ are both defined over $\mathbb R$.

Combining with \Cref{dimension_formula} and \Cref{KL_closure}, \Cref{fiber_bundle} implies the following corollary.

\begin{corollary}
Let $u,v,w\in W$ such that $w\circ_l v\leq u$. Then
\begin{enumerate}
    \item $\dim \oZ^u_{w,v} = \ell(w)+\ell(u)-\ell(v)$;
    \item the Zariski closure of $\oZ^u_{w,v}$ is $\bigsqcup_{w'\leq w,v\leq v',u'\leq u,w'\circ_l v'\leq u'}\oZ^{u'}_{w',v'}$.
\end{enumerate}

\end{corollary}

This also shows that $\CZ = \bigsqcup\oZ^u_{w,v}$ is a stratification.

\subsection{Thinness and Shellability} 
The poset of the stratification of $\mathcal Z$ is
$$Q = \{(w,v,u)\in W^3|w\circ_l v\leq u\},$$
where $(w',v',u')\leq (w,v,u)$ if $w'\leq w$, $v\leq v'$ and $u'\leq u$. Let $\hat Q = Q\sqcup\{\hat0\}$ be the augmented poset with a minimum adjoined. This subsection is devoted to study some combinatorial properties of $\hat Q$.

Let $(\tilde Q^I,\leq_{\tilde Q^I})$ be the poset of intervals in $(\tilde W,\leq^I)$ with respect to inclusion. Namely, 
$$\tilde Q^I = \{[x,y]\in\tilde W\times\tilde W|x\leq^I y\}, $$
and $[x_1,y_1]\leq_{\tilde Q^I}[x_2,y_2]$ if $x_2\leq^I x_1\leq^Iy_1\leq^Iy_2$ in $\tilde W$.

\begin{proposition} \label{Q_thin}
The poset $\hat Q$ is pure and thin.

\begin{proof}
We define
$$h:Q\to\tilde Q^I:(w,v,u)\to [i(u),th(w,v)] = [u,ws_\infty v].$$
The map $h$ is injective and preserves partial orders. We may identify $Q$ with
$$\text{im}(h) = \{[x,y]\in\tilde Q|x\in W, y\in Ws_\infty W\}.$$
It is easy to see that $\text{im}(h)$ is a convex sub-poset of $\tilde Q$, and thus $Q$ is pure and thin. It is also clear that $\hat Q$ is pure.

It remains to show that, for every $q\in Q$ such that the rank of the interval $[\hat 0,q]$ is $2$, there are exactly $2$ elements in $Q$ that are less than $q$. Suppose that $q = (w,v,u)$. Then $\ell(w)+\ell(u)-\ell(v)=1$. Let $(w',v',u')\in Q$ such that $(w',v',u')<(w,v,u)$. Then $w'\leq w$, $v'\geq v$ and $u'\leq u$. Thus $w\circ_l v\leq w'\circ_l v'\leq u'\leq u$. Since $\ell(u) - \ell(w\circ_l v)\leq 1$, we have $u' = w'\circ_l v'$.

Since $w\circ_l v\leq u$, there exists a minimal $c\in W$ such that $c\leq w$, $cv\leq u$ and $\ell(cv) = \ell(v)-\ell(c)$. We have either $\ell(w)-\ell(c)=1$ and $u = cv$, or $w = c$ and $\ell(u)-\ell(cv)=1$. Similarly, there exists a minimal $c'\in W$ such that $c'\leq w'$, $c'v' \leq u'$ and $\ell(c'v') = \ell(v')-\ell(c')$.

Suppose first that $u = cv$. Then $u = u' = c'v' = cv$ and $\ell(w) -\ell(c) = 1$. Since $c'v' = cv$ and $v\leq v'$, we must have $c\leq c'$. Then $c\leq c'\leq w'\leq w$ implies either $c = c' = w'$ or $c' = w' = w$. We have either $(w',v',u') = (c,v,u)$ or $(w',v',u') = (w,w^{-1}u,u)$.

Finally, we suppose that $\ell(u)-\ell(cv) = 1$. Then $c = w$. By assumption, for any $c''<w$, we have $c''v\nleq u$. Since $v'\geq v$, we must have $c = c'=w' = w$. Since $c'v' = wv' = u'$, we have either $(w',v',u') = (w,w^{-1}u,u)$ or $(w',v',u') = (w,v,wv)$.
\end{proof}

\end{proposition}

The rest of this subsection is devoted to prove that $\hat Q$ is shellable. The augmented poset of stratification of $\CZ^{+,+}$ was shown to be shellable in \cite[Theorem 4.4]{BH22}. We follow the method there.

\subsubsection{EL-labeling} \label{EL-labeling}
Suppose that $P$ is a pure poset. An \textit{edge labeling} of $P$ is a map $\lambda$ from the set of all covering relations in $P$ to a poset $\Lambda$. The labeling $\lambda$ sends the maximal chain $(y=z_0\gtrdot z_1\gtrdot z_2\gtrdot\cdots\gtrdot z_n=x)$ of an interval $[x,y]$ of $P$ to a tuple $(\l(z_0,z_1),\l(z_1,z_2),\cdots,\l(z_{n-1},z_n))$ in $\Lambda$. A maximal chain is called increasing if the associated tuple of $\Lambda$ is increasing. An edge labeling of $P$ is called an \textit{EL-labeling}
(edge lexicographical labeling) if
for every interval $[x,y]$, there exists a unique increasing maximal chain, and this maximal chain is lexicographically minimal among all maximal chains of $[x,y]$. If a pure poset admits an EL-labeling, then it is shellable by \cite[Theorem 3.3]{BW82}. We say that such a poset is \textit{EL-shellable}.

Consider a Coxeter system $(W,I)$. Let $J\subseteq I$ and $Q^J$ be the poset of intervals in $(W,\leq^J)$ (with respect to inclusion of intervals). Let $\hat Q^J = Q^J\sqcup\{\hat 0\}$ with $\hat 0$ an adjoined minimal element. We consider EL-labelings on $W$ and $Q^J$ similar to \cite[\S4.5]{BH22}. Denote $\lessdot^J$ the covering relation of $(W,\leq^J)$.

Let $T$ be the set of all reflections in $W$. A \textit{reflection order} is a total order $\prec$ on $T$ such that for every $s,t\in T$, we have either $t\prec sts\prec tstst\prec\cdots $ or $s\prec tst\prec ststs\prec\cdots$. By \cite[\S3]{Dy92}, for $w_1\lessdot^J w_2$, we have $w_2w_1^{-1}\in T$, and the labeling $\lambda_W(w_1\lessdot^J w_2) = w_2w_1^{-1}$ is an EL-labeling on $(W,\leq^J)$.

Now, we consider a labeling for $Q^J$. Fix a reflection order $\prec$ on $T$. Let
$$\Lambda = \{(t,?)|t\in T, ?\in\{l,r\}\}\sqcup\{\emptyset\}.$$
We define a total order $\prec_\Lambda$ on $\Lambda$ such that $(t_1,r)\prec_\Lambda\emptyset\prec_\Lambda(t_2,l)$for every $t_1,t_2\in T$, and that $(t_1,?)\prec_\Lambda(t_2,?)$ if $t_1\prec t_2$ for $?\in\{l,r\}$.

Denote the covering relation in $\hat Q^J$ by $\lessdot_{\hat Q^J}$. Similar to \cite{Wil07}, for $q_1,q_2\in\hat Q^J$ we have $q_1\lessdot_{\hat Q^J} q_2$ if and only if one of three situations holds, for each of which we attach a label.
\begin{itemize}
    \item $q_1 = [x,y_1]$, $q_2 = [x,y_2]$ and $y_1\lessdot^J y_2$ in $W$, and we take $\lambda_{\hat Q^J}(q_1\lessdot_{\hat Q^J} q_2) =( \lambda_W(y_1\lessdot^J y_2),r) = (y_2y_1^{-1},r)$;  
    \item $q_1 = [x_1,y]$, $q_2 = [x_2,y]$ and $x_2\lessdot^J x_1$ in $W$, and we take $\lambda_{\hat Q^J}(q_1\lessdot_{\hat Q^J} q_2) = (\lambda_W(x_2\lessdot^J x_1),l) = (x_1x_2^{-1},l)$;  
    \item $q_1 = \hat 0$ and $q_2  = [x,x]$ for some $x\in W$, and we take $\lambda_{\hat Q^J}(q_1\lessdot_{\hat Q^J} q_2) = \emptyset$.
\end{itemize}
Similar to \cite{Wil07} and \cite[\S4.5.2]{BH22}, $\lambda_{\hat Q^J}$ defines an EL-labeling on $\hat Q^J$.

\subsubsection{Shellability of $\hat Q$} \label{sec:shellable_Q}
We shall construct an EL-labeling on $\hat Q$ to show that, for every $q\in \hat Q$, the interval $[\hat 0, q]$ is EL-shellable, and thus shellable. 

First, we introduce some notations and results on the combinatorics of Coxeter groups. Recall that we have a Coxeter system $(W,I)$ for $G$ and another one $(\tilde W,\tilde I)$ for $\tilde G$. We regard $I$ as a subset of $\tilde I$. Let $\tilde Q^I$ be the poset of intervals in $(\tilde W,\leq^I)$. Let $\check{\tilde Q}^I :=\tilde Q^I\sqcup\{\check 0\}$ be the augmented poset with $\check 0$ an adjoined minimal element.

We remark that the covering relations $\lessdot^I$ within $Ws_\infty W\cup W$ have a simple description. Similar to \cite[Lemma 4.3]{BH21b}, we have
\begin{enumerate}[label = (\alph*)]
    \item for $y\in Ws_\infty W$, the covering relation $x\lessdot^I y$ is of one of the following forms
    \begin{itemize}
        \item $y = ws_\infty v$, $x = w's_\infty v$, and $w'\lessdot w$;
        \item $y = ws_\infty v$, $x = ws_\infty v'$, and $v\lessdot v'$;
        \item $y = ws_\infty v$, $x = wv$, when $\ell(wv) = \ell(v)-\ell(w)$.
    \end{itemize}
\end{enumerate}

Now we consider a specific EL-labeling on $\check{\tilde Q}^I$. By \cite[(2.3)]{Dy93}, there is a reflection order on $\tilde T$ such that
$$t\prec t',\text{ for }t\in \tilde T\cap\tilde W_I,\text{ and }t'\in \tilde T\setminus\tilde W_I.$$
By \cite[\S4.5.3 (a)]{BH22}, this property and \S\ref{sec:shellable_Q} (a) together imply that
\begin{enumerate}[label = (\alph*)]
\setcounter{enumi}{1}
    \item for $y\in Ws_\infty W$, $y_1\lessdot^I y$ and $y_2\lessdot^I y$ with $y_1\in Ws_\infty W$ and $y_2\in W$, we have $yy_1^{-1}\prec yy_2^{-1}$ in $\tilde T$.
\end{enumerate}
This reflection order gives an EL-labeling $\tilde\lambda$ on $\check{\tilde Q}^I$ with label set $\Lambda$ as in \S\ref{EL-labeling}. 

Before we prove the main result of this subsection, we have the following statement \cite[Proposition 2.5]{Bjo80} to check whether a labeling is an EL-labeling or not.

\begin{proposition}\label{check_EL}
Let $\lambda$ be an edge labeling on a pure poset $P$ such that every interval $[x,y]$ contains a unique increasing maximal chain, denoted by $ch(y,x)=(y = z_0\gtrdot z_2\gtrdot\cdots\gtrdot z_n = x)$. Then such labeling is an EL-labeling if and only if for any such $ch(y,x)$ and $y\gtrdot y'\geq x$ we have $\lambda(y,z_1)\prec\lambda(y,y')$.
\end{proposition}

Now we are ready to prove the EL-shellability of $\hat Q$.

\begin{proposition} \label{prop:shellable_Q}
For every $q\in\hat Q$, the interval $[\hat 0,q]$ is shellable.

\begin{proof}
We define
$$h:\hat Q\to\check{\tilde Q}^I:(w,v,u)\to [i(u),th(w,v)] = [u,ws_\infty v],\ \hat 0\to \check 0.$$
The map identifies $\hat Q$ with a sub-poset of $\check{\tilde Q}^I$, and $h(Q)$ is convex in $\check{\tilde Q}^I$.

We define an edge labeling $\lambda$ on $\hat Q$ by
$$\lambda\left((w',v',u')\lessdot (w,v,u)\right) = \tilde\lambda\left([i(u'),th(w',v')]\lessdot_{\check{\tilde Q}^I}[i(u),th(w,v)]\right),$$
$$\lambda\left(\hat 0\lessdot (w,v,u)\right) = \text{ the first edge label in the chain }ch([i(u),th(w,v)],\check 0)\subseteq\check{\tilde Q}^I\text{ via }\tilde\lambda.$$
Since $\tilde\lambda$ is an EL-labeling on $\check{\tilde Q}^I$, for any $(w',v',u')< (w,v,u)$ in $Q$, there is a unique increasing maximal chain from $(w,v,u)$ to $(w',v',u')$ which is lexicographically minimal. It remains to consider the chain from $(w,v,u)$ to $\hat 0$.

By our construction of the EL-labeling for $\check{\tilde Q}^I$ in \S\ref{EL-labeling}, the unique increasing maximal chain from $\check 0$ to $[i(u), th(w,v)]$ in $\check{\tilde Q}^I$ is 
$$ch([i(u),th(w,v)],\check 0) = \left([i(u),th(w,v)]_{\check{\tilde{Q}}^{I}}\gtrdot[i(u),\tilde w_1]_{\check{\tilde{Q}}^{I}}\gtrdot\cdots_{\check{\tilde{Q}}^{I}}\gtrdot[i(u),i(u)]_{\check{\tilde{Q}}^{I}}\gtrdot\check 0\right).$$
Since $\tilde\lambda$ gives an EL-labeling on $\check{\tilde Q}^I$, it follows from \S\ref{sec:shellable_Q} (b) and \Cref{check_EL}  that $ch([i(u),th(w,v)],[i(u),i(u)])\cap h(Q)$ contains a minimal element in $h(Q)$, denoted by $[i(u),th(w_n,v_n)]$. Then by the convexity of $h(Q)$ in $\tilde Q^I$, we have
\begin{align*}
   ch([i(u),th(w,v)],[i(u),i(u)]) &= \left([i(u),th(w,v)]_{\check{\tilde{Q}}^{I}}\gtrdot[i(u),th(w_1,v_1)]_{\check{\tilde{Q}}^{I}}\gtrdot\cdots\right. \\
&\left._{\check{\tilde{Q}}^I}\gtrdot [i(u),th(w_n,v_n)]_{\check{\tilde{Q}}^{I}}\gtrdot\cdots_{\check{\tilde Q}^I}\gtrdot [i(u),i(u)]\right).
\end{align*}
It is clear that $(w,v,u)\gtrdot(w_1,v_1,u)\gtrdot\cdots (w_n,v_n,u)\gtrdot\hat 0$ is a maximal chain from $(w,v,u)$ to $\hat 0$.

By construction, any increasing maximal chain
from $(w,v,u)$ to $\hat 0$ in $\hat Q$ can be extended to an increasing maximal chain from $[i(u),th(w,v)]$ to $\check 0$ in $\check{\tilde Q}^I$. This implies that there is a unique increasing maximal chain from $(w,v,u)$ to $\hat 0$.

Finally, since $\tilde\lambda$ is an EL-labeling on $\check{\tilde Q}^I$, it follows from \Cref{check_EL} that $\lambda$ is an EL-labeling on $\hat Q$.
\end{proof}

\end{proposition}

\subsection{Total positivity in $\mathcal Z$} \label{positive_Z}
The totally nonnegative double flag variety $\CZ_{\geq0}$ is defined as the Hausdorff closure of $(G_{\geq0},G_{\geq0}B^-/B^-)$ in $\CZ$. We now prove the main results on $\CZ_{\geq0}$.

\begin{theorem} \label{Z_parametrization}
For $v,w,u\in W$ such that $w\circ_l v\leq u$, set $$\oZ_{v,w,>0}^u: = \oZ^u_{v,w}\cap\CZ_{\geq0}.$$  Let $c$ be the minimal element such that $c\leq w$ and $v\leq c^{-1}u$. For any choice of reduced expressions $\underline{w}$ and $\underline{c^{-1}u}$, we have an isomorphism $$G^-_{\underline{c}_+,\underline{w},>0}\times G^+_{\underline{v}_+,\underline{c^{-1}u},>0}\xrightarrow{\sim}\CZ_{w,v,>0}^u,\ \ (g_1,g_2)\to (g_1,g_2B^-).$$

\begin{proof}
For $(w,v,u)\in W^3$ such that $i(u)\leq^I th(w,v)$, we have $$\mathring{\tilde\CB}^I_{i(u),th(w,v)}\subseteq\overline{\mathring{\tilde{\CB}}^I_{i(u),th(w,1)}}.$$
Then
$$\tilde\CB^I_{i(u),th(w,v),>0}\subseteq\overline{U^-_{w,>0}y_\infty(\mathbb R_{>0})U^+_{u,>0}\IB^+/\IB^+}.$$
Therefore
$$\tilde\CB^I_{\geq0}\bigcap\tilde\CZ = \overline{G_{\geq0}y_\infty(\mathbb R_{>0})G_{\geq0}\IB^+/\IB^+}\bigcap\tilde\CZ.$$

Recall that we have the isomorphism
$$s:\mathbb R^\times\times\CZ(\mathbb R)\to\tilde\CZ(\mathbb R),\ \ (a;g_1,g_2B^-)\to g_1y_\infty(a)g_2\IB^+/\IB^+,$$
and the fiberation
$$f:\tilde\CZ(\mathbb R)\to\CZ(\mathbb R),\ \ g_1y_\infty(a)g_2\IB^+/\IB^+\to (g_1,g_2B^-/B^-).$$
We have 
\begin{align*}
    s(\mathbb R_{>0};\CZ_{\geq0}) &= \overline{s(\mathbb R_{>0}; G_{\geq0},G_{\geq0}\IB^+/\IB^+)}=\tilde\CB^I_{\geq0}\bigcap\tilde\CZ \\
    &=\bigsqcup_{(w,v,u)\in W^3, i(u)\leq^Ith(w,v)}\tilde\CB^I_{i(u),th(w,v),>0}
\end{align*}
By \Cref{thm:J-flag-main} (1), we have
$$\CB^I_{i(u),th(w,v),>0} = G^-_{\underline{c}_+,\underline{w},>0}y_\infty(\mathbb R_{>0}) G^+_{\underline{v}_+,\underline{c^{-1}u},>0}\IB^+/\IB^+.$$
By \Cref{fiber_bundle}, we have
$$\mathcal Z^u_{w,v,>0} = f\left(\tilde\CB^I_{i(u),th(w,v),>0}\right) = \left(G^-_{\underline{c}_+,\underline{w},>0},G^+_{\underline{v}_+,\underline{c^{-1}u},>0}B^-/B^-\right).$$
The theorem follows.
\end{proof}

\end{theorem}

In particular, $s:\mathbb R^\times\times\CZ(\mathbb R)\to\tilde{\CZ}(\mathbb R)$ restricts to a stratified isomorphism 
$$s:\mathbb R_{>0}\times\CZ_{\geq0}\simeq\tilde\CB^I_{\geq0}\cap\tilde{\CZ}.$$

\begin{theorem} \label{Z_total_positivity}
The totally nonnegative double flag variety $\CZ_{\geq0} = \bigsqcup\CZ^u_{w,v,>0}$ is a remarkable polyhedral space with regularity property. 

\begin{proof}
It follows from \Cref{Z_parametrization} that $\CZ^u_{w,v,>0}\simeq\mathbb R_{>0}^{\ell(w)+\ell(u)-\ell(v)}$. Since $\tilde\CB^I_{i(u),th(w,v),>0}$ is a connected component of $\mathring{\tilde{\CB}}^I_{i(u),th(w,v)}(\mathbb R)$, and $\mathcal Z^u_{w,v,>0} = f\left(\tilde\CB^I_{i(u),th(w,v),>0}\right)$, we have that $\mathcal Z^u_{w,v,>0}$ is a connected component of $\oZ^u_{w,v}(\mathbb R)$. Since $s$ is stratified, we have
\begin{align*}
    s(\mathbb R_{>0};\overline{\CZ_{w,v>0}^u}) &= \overline{\tilde{\CB}^I_{i(u),th(w,v),>0}}\bigcap\tilde\CZ \\
    &=\bigsqcup_{ i(u)\leq^Ii(u')\leq^Ith(w',v')\leq^I th(w,v)}\tilde\CB^I_{i(u'),th(w',v'),>0}
\end{align*}
Applying $f$, we have
$$\overline{\CZ_{w,v>0}^u} = \bigsqcup_{(w',v',u')\leq (w,v,u)}\CZ_{w',v',>0}^{u'}.$$
This shows that $\CZ_{\geq0}$ is a remarkable polyhedral space.

We next prove
\begin{enumerate}[label = (\alph*)]
    \item $\overline{\CZ_{w,v>0}^u}$ is a topological manifold with boundary $\overline{\CZ_{w,v>0}^u}\setminus\CZ^u_{w,v,>0}$.
\end{enumerate}
For $\tilde w\in\tilde W$, we denote $\tilde\CB^I_{\tilde w,\geq0} = \tilde\CB^I_{\tilde w}\cap\tilde\CB^I_{\geq0}$ and $\tilde\CB^{I,\tilde w}_{\geq0} = \tilde\CB^{I,\tilde w}\cap\tilde\CB^I_{\geq0}$. These are closed subsets of $\tilde\CB^I$.

According to \cite[Lemma 1.4]{He07}, there is a unique maximal element in the set $\{(w')^{-1}u'\in W|w'\leq w,u'\leq u\}$. We denote this element by $w^{-1}*v$. It is clear that for every 
$\tilde w\in Ws_\infty W$ such that $i(u)\leq^I\tilde w\leq^I th(w,v)$, we have $th(1,w^{-1}*u)\leq^I\tilde w$. Similarly, for every $\tilde w\in W$ such that $i(u)\leq^I\tilde w\leq^Ith(w,v)$, we have $\tilde w\leq^I i(w\circ_l v)$.
Then  
$$\overline{\tilde\CB_{i(u),th(w,v),>0}^I}\bigcap\tilde\CZ = \overline{\tilde\CB_{i(u),th(w,v),>0}^I}\big\backslash\left(\tilde\CB^{I,th(1,w^{-1}*u)}_{\geq0}\bigcup\tilde\CB^I_{i(w\circ_l v),\geq0}\right)$$
is an open subset of the topological manifold $\overline{\tilde\CB_{i(u),th(w,v),>0}^I}$, and thus is also a topological manifold.
 
Since the interior of $\overline{\tilde\CB^I_{i(u),th(w,v),>0}}$ is $\tilde\CB^I_{i(u),th(w,v),>0}$ and it lies in $\tilde\CZ$, we know that $$\partial\left(\overline{\tilde 
\CB_{i(u),th(w,v),>0}^I}\bigcap\tilde\CZ\right) = \partial\overline{\tilde \CB_{i(u),th(w,v),>0}^I}\bigcap\tilde\CZ = \bigsqcup_{(w',v',u')<(w,v,u)}\tilde \CB_{i(u'),th(w',v'),>0}^I.$$

The isomorphism $s:\mathbb R^\times\times\CZ(\mathbb R)\to\tilde\CZ(\mathbb R)$ restricts to a stratified isomorphism $$\mathbb R_{>0}\times\overline{\CZ^u_{w,v,>0}}\simeq\overline{\tilde\CB_{i(u),th(w,v),>0}^I}\bigcap\tilde\CZ$$
that is $\mathbb R_{>0}$-equivariant, where $\mathbb R_{>0}$ acts on $\mathbb R_{>0}$ via $a.b = a^{-2}b$, on $\CZ$ trivially, and on $\overline{\tilde\CB^I_{i(u),th(w,v)}}\bigcap\tilde\CZ$ via $\alpha^\vee_\infty$. We see that $\overline{\tilde\CB_{i(u),th(w,v),>0}^I}\bigcap\tilde\CZ$ is a principal $\mathbb R_{>0}$-bundle, and 
$$\overline{\CZ^u_{w,v,>0}}\simeq \left(\overline{\tilde\CB_{i(u),th(w,v),>0}^I}\bigcap\tilde\CZ\right)/\mathbb R_{>0}.$$
This proves (a).

Finally, we show that $\overline{\CZ_{w,v>0}^u}$ is a regular CW complex by induction on the dimension of $\CZ_{w,v>0}^u$. When $\dim\CZ_{w,v,>0}^u = 0$, it is a point. Now by induction assumption, $\partial\overline{\CZ_{w,v>0}^u}$ is a regular CW complex. By \Cref{shellable_regular}, \Cref{Q_thin} and \Cref{prop:shellable_Q}, we have that $\partial\overline{\CZ_{w,v>0}^u}$ is homeomorphic to a sphere of dimension $\dim\CZ_{w,v,>0}^u - 1$, and the result follows from (a) and \Cref{thm:poincare}.
\end{proof}

\end{theorem}

\subsection{Link of identity}
Let $(G/T)_{\geq0}$ be the image of $G_{\geq0}$ under the natural projection $G\to G/T$. In this subsection, we consider the reduced double Bruhat cell $L^{w,u} = (B^+\dot wB^+\cap B^-\dot uB^-)/T\subseteq G/T$ and its totally positive part $L^{w,u}_{>0} := L^{w,u}\cap (G/T)_{\geq0}$. Denote the Hausdorff closure of $L^{w,u}_{>0}$ in $(G/T)(\mathbb R)$ by $L^{w,u}_{\geq0}$. The embedding
$$G/T\to\CZ,\ gT/T\to (g,B^-)$$
identifies $L^{w,u}$ with $\oZ^{u}_{w,1}$ and $L^{w,u}_{>0}$ with $\CZ^u_{w,1,>0}$. We write $\mathcal X := \bigsqcup_{w,u\in W}\oZ^u_{w,1}$. Then we have a stratified isomorphism $G/T\xrightarrow[]{\sim}\mathcal X$, which restricts to a stratified isomorphism $(G/T)_{\geq0}\simeq\mathcal X_{\geq0}: = \mathcal X\cap\mathcal Z_{\geq0}$. For $w,u\in W$, we write $\CX_{w,u,>0}:=\CZ^u_{w,1,>0}$, and $\overline{\CX_{w,u,>0}}$ the Hausdorff closure of $\CX_{w,u,>0}$ in $\CX(\mathbb R)$. 

Let $X^{+}$ be the set of dominant weights of $G$, and $X^{++}\subseteq X^+$ the set of regular dominant weights of $G$. For $\lambda \in X^{+}$, let ${}^\omega \Lambda_\lambda$ denote the simple lowest weight $G$-module with lowest weight $-\lambda$, and $\Lambda_\lambda$ denote the simple highest weight $G$-module with highest weight $\lambda$. Let $\xi_{-\lambda}\in{}^\omega\Lambda_\lambda$ be a lowest weight vector, and $\eta_\lambda\in\Lambda_\lambda$ a highest weight vector. 

For $\lambda\in X^{++}$, we consider the tensor product ${}^\omega \Lambda_\lambda(\mathbb R)\otimes_{\mathbb R}\Lambda_\lambda(\mathbb R)$ equipped with an Euclidean norm $\lVert\cdot\rVert$. There is a well-defined embedding
$$\CX\to {}^\omega \Lambda_\lambda(\mathbb R)\otimes_{\mathbb R}\Lambda_\lambda(\mathbb R),\ (g,B^-)\to g\xi_{-\lambda}\otimes g\eta_{\lambda}.$$
Let $w,u\in W$. We define the link of identity in $L^{w,u}_{\geq0}$ to be
$$Lk^{w,u}_{\geq0} := \left\{x\in\overline{\CX_{w,u,>0}}\Big|\lVert x-\eta_\lambda\otimes\xi_{-\lambda}\rVert=1\right\}.$$
For $w',u'\in W$ such that $w'\leq w$, $u'\leq u$, we write $Lk^{w',u'}_{>0} = Lk^{w,u}_{\geq0}\cap\CX_{w',v',>0}$. Since $\CZ_{w',1,>0}^{u'}\simeq\mathbb R_{>0}^{\ell(w')+\ell(v')}$, we have $Lk^{w',u'}_{>0}\simeq\mathbb R_{>0}^{\ell(w')+\ell(u')-1}$.

\begin{theorem}
Suppose that $(w,u)\neq(1,1)$. The space $Lk^{w,u}_{\geq0} = \bigsqcup_{w'\leq w,u'\leq u}Lk^{w',u'}_{>0}$ is a regular CW complex. In particular, it is homeomorphic to a closed ball.

\begin{proof}
Note that $\mathcal X$ is open in $\mathcal Z$. We have that  $\overline{\CZ_{w,1,>0}^u}\bigcap\CX$ is an open subset of the topological manifold $\overline{\CZ_{w,1,>0}^u}$, and thus is also a topological manifold. Since $\CZ_{w,1,>0}^u\subseteq\overline{\CX_{w,1,>0}^u}$, we have   
$$\partial\overline{\CX_{w,u,>0}} = \partial\overline{\CZ^u_{w,1,>0}}\bigcap\CX = \bigsqcup_{w'\leq w,u'\leq u,(w',u')\neq(w,u)}\CX_{w',u',>0}.$$
This implies that $Lk^{w,u}_{\geq0}$ is a topological manifold with boundary 
$$\partial Lk^{w,u}_{\geq0} = \bigsqcup_{w'\leq w,u'\leq u,(w',u')\neq (w,u)}Lk^{w',u'}_{>0}.$$
We show that $Lk^{w,u}_{\geq0}$ is a regular CW complex by induction on $\ell(w)+\ell(u)$. When $\ell(w)+\ell(u)=1$, it follows by a direct calculation that $Lk^{w,u}_{\geq0}$ is a point. Now by induction assumption, $\partial Lk^{w,u}_{\geq0}$ is a regular CW complex. By \Cref{Weyl_shellable} and \Cref{shellable_regular}, we have that $\partial Lk^{w,u}_{\geq0}$ is homeomorphic to a sphere of dimension $\ell(w)+\ell(u)-2$, and the result follows from \Cref{thm:poincare}.
\end{proof}
    
\end{theorem}

\subsection{Connection to canonical basis} \label{sec:canonical_basis}
In this subsection, we assume that $G$ is simply-laced. 

For $\lambda_1,\lambda_2\in X^+$, the canonical basis $\mathbf B({}^\omega\Lambda_{\lambda_1}\otimes\Lambda_{\lambda_2})$ for the tensor product ${}^\omega\Lambda_{\lambda_1}\otimes\Lambda_{\lambda_2}$ was introduced by Lusztig \cite{L92b}. We consider the projective space $\mathbb P({}^\omega\Lambda_{\lambda_1}\otimes\Lambda_{\lambda_2})$ of ${}^\omega\Lambda_{\lambda_1}\otimes\Lambda_{\lambda_2}$. For $v\in {}^\omega\Lambda_{\lambda_1}\otimes\Lambda_{\lambda_2}$, denote $[v]\in\mathbb P({}^\omega\Lambda_{\lambda_1}\otimes\Lambda_{\lambda_2})$ the corresponding line in the projective space. We have a morphism
$$\rho_{\lambda_1,\lambda_2}:\CZ\to\mathbb P({}^\omega\Lambda_{\lambda_1}\otimes\Lambda_{\lambda_2}),\quad (g_1,g_2B^-/B^-)\to [g_1g_2\xi_{-\lambda_1}\otimes g_1\eta_{\lambda_2}].$$
We simply write $V$ for ${}^\omega\Lambda_{\lambda_1}\otimes\Lambda_{\lambda_2}$ and $\rho$ for $\rho_{\lambda_1,\lambda_2}$.

\begin{proposition} \label{prop_canonical_basis}
Let $\mathbb P(V)_{\geq0}$ be the set of lines spanned by a vector whose coordinates with respect to $\mathbf B(V)$ are all nonnegative. Then we have 
$$\rho(\CZ_{\geq0})\subseteq \mathbb P(V)_{\geq0}.$$

\begin{proof}
Let $\mathfrak g$ be the Lie algebra of $G$, and $\{e_i,h_i,f_i\}_{i\in I}$ be the Chevalley generators of the derived subalgebra of $\mathfrak g$, such that $x_i(a) = \exp(ae_i)$ and $y_i(a) = \exp(af_i)$. According to \cite[Theorem 6.5 (2)]{FH25}, for any $b\in\mathbf B(V)$ we have
$$(e_i\otimes 1+1\otimes e_i)\cdot b\in\mathbb N[\mathbf B(V)],\quad (f_i\otimes 1+1\otimes f_i)\cdot b\in\mathbb N[\mathbf B(V)].$$
Therefore, for $a>0$, we have $$(x_i(a),x_i(a))\cdot \mathbb P(V)_{\geq0} = \exp(a(e_i\otimes 1+1\otimes e_i))\cdot\mathbb P(V)_{\geq0}\subseteq\mathbb P(V)_{\geq0},$$
and similarly, $(y_i(a),y_i(a))\cdot\mathbb P(V)_{\geq0}\subseteq\mathbb P(V)_{\geq0}$.
Moreover, every vector in $\mathbf B(V)$ is a weight vector. We see that $\mathbb P(V)_{\geq0}$ is invariant under the action of $G_{\text{diag},\geq0}$.

Since $\xi_{-\lambda_1}\otimes\eta_{\lambda_2}\in\mathbf B(V)$, we have $(e,B^-/B^-)\in\mathbb P(V)_{\geq0}$. Since $\mathbb P(V)_{\geq0}$ is invariant under the action of $G_{\text{diag},\geq0}$, we have $$(G_{\geq0},G_{\geq0}B^-/B^-)\subseteq\mathbb P(V)_{\geq0},$$ and then 
$$\CZ_{\geq0} = \overline{(G_{\geq0},G_{\geq0}B^-/B^-)}\subseteq\mathbb P(V)_{\geq0}.$$
The conclusion follows.
\end{proof}

\end{proposition}

When $\lambda_1$ and $\lambda_2$ are regular dominant, $\rho$ is an embedding, and \Cref{prop_canonical_basis} verifies one side inclusion of \Cref{conjecture_canonical_basis}.

\end{document}